\documentclass[reqno,11pt]{amsart}
\usepackage{amsthm,amsfonts,amssymb,euscript, 
mathrsfs,graphics,color,amsmath,latexsym,marginnote}
\usepackage{cite}   
\usepackage{dsfont}       
\usepackage[english]{babel} 
\usepackage{mathtools}
\usepackage{times}

\numberwithin{equation}{section}

\usepackage{datetime}
\usepackage{lipsum}
\usepackage{enumerate}
\usepackage{tikz}
\usetikzlibrary{matrix}
\usepackage[pdftex]{hyperref}
\RequirePackage{amscd}
\RequirePackage{epic}
\RequirePackage[all]{xy}
\RequirePackage{url}
\RequirePackage[shortlabels]{enumitem}
\usepackage{empheq}
\usepackage{epsfig}
\usepackage{lineno} 
 \usepackage{perpage}
 \usepackage[english]{babel}

\usepackage[utf8]{inputenc}
\usepackage{cite}
\RequirePackage{amscd}
\RequirePackage{epic}
\RequirePackage{eepic}
\RequirePackage[all]{xy}
\RequirePackage{url}
\RequirePackage[shortlabels]{enumitem}
\usepackage{empheq}
\usepackage{nomencl}
\usepackage{epsfig}
\usepackage{graphicx}

\theoremstyle{plain}

\newtheorem{theorem}{Theorem}[section]
\newtheorem{proposition}[theorem]{Proposition}
\newtheorem{lemma}[theorem]{Lemma}
\newtheorem{corollary}[theorem]{Corollary}

\newtheorem{remark}[theorem]{Remark}

\newtheorem{definition}[theorem]{Definition}

\renewcommand{\Re}{\mathrm{Re}\,}
\renewcommand{\Im}{\mathrm{Im}\,}

\providecommand{\vect}[2]{{\bigl[\begin{smallmatrix}#1\\#2\end{smallmatrix}\bigr]}}   
\providecommand{\sm}[4]{{\bigl[\begin{smallmatrix}#1&#2\\#3&#4\end{smallmatrix}\bigr]}}

    \newcommand{\set}[1]{{\left\{#1\right\}}}
\newcommand{\pa}[1]{{\left(#1\right)}}
\newcommand{\norm}[1]{{\left |#1\right |}}

\RequirePackage{epic}
\RequirePackage{eepic}
\RequirePackage[all]{xy}
\RequirePackage{url}
\RequirePackage[shortlabels]{enumitem}
\usepackage{empheq}
\usepackage{nomencl}
\usepackage{epsfig}
\usepackage{graphicx}

\newcommand{\T}{\mathbb{T}}
\newcommand{\Z}{\mathbb{Z}}
\newcommand{\R}{\mathbb{R}}
\newcommand{\C}{\mathbb{C}}
\newcommand{\teta}{\theta}

\newcommand{\eps}{\varepsilon}

\renewcommand{\Re}{\operatorname{Re}}
\renewcommand{\Im}{\operatorname{Im}}

\newcommand{\na}{\widehat{n}}

\newcommand{\co}[1]{\textit{#1}}
\newcommand{\gr}[1]{\textbf{#1}}

\newcommand{\id}{\operatorname{id}}

\usepackage{amsthm}

\newcommand{\g}{\gamma}

\newcommand{\s}{{\sigma}}

\def\wc{ {}}

\def\norma#1{\left |#1\right |}

\newcommand{\N}{{\mathbb N}}

\newcommand{\cA}{{\mathcal A}}
\newcommand{\cB}{{\mathcal B}}
\newcommand{\cC}{{\mathcal C}}

\newcommand{\cF}{{\mathcal F}}

\newcommand{\cH}{{\mathcal H}}

\newcommand{\cK}{{\mathcal K}}

\newcommand{\cM}{{\mathcal M}}

\newcommand{\cR}{{\mathcal R}}

\newcommand{\td}{{\mathtt{d}}}

\newcommand{\tf}{{\mathtt{f}}}

\newcommand{\tm}{{\mathtt{m}}}

\newcommand{\tB}{{\mathtt{B}}}
\newcommand{\tC}{{\mathtt{C}}}
\newcommand{\tD}{{\mathtt{D}}}

\newcommand{\tH}{{\mathtt{H}}}
\newcommand{\tK}{{\mathtt{K}}}

\newcommand{\tN}{{\mathtt{N}}}

\usepackage{bm}

\newcommand{\al}{{\alpha}}
\newcommand{\bt}{{\beta}}

\newcommand{\abs}[1]{\left| #1 \right|}

\newcommand{\jap}[1]{\langle #1 \rangle}
\newcommand{\und}[1]{\underline{#1}}

\newcommand{\e}{{\varepsilon}}

\newcommand{\jml}[1]{\lfloor #1 \rfloor}

\newcommand{\meas}{{\mathtt{meas}}}

\newcommand{\tw}{{\mathtt{w}}}
\renewcommand{\th}{{\mathtt{h}}}
\newcommand{\nnorm}[1]{{\left\vert\kern-0.25ex\left\vert\kern-0.25ex\left\vert #1 
    \right\vert\kern-0.25ex\right\vert\kern-0.25ex\right\vert}}

\newcommand{\es}{e^{L_{S} }}

\newcommand{\bbt}{{\bm \bt}}

\usepackage{framed,enumitem} 

\newcommand{\SE}{\mathtt{sE}}
\newcommand{\SO}{\mathtt S}

\newcommand{\suca}{\mathtt N}
\newcommand{\ri}{r}
\newcommand{\rs}{{r^*}}
\newcommand{\rf}{{r'}}

\newcommand{\twi}{\tw}
\newcommand{\twf}{{\tw'}}

\newcommand{\imm}{{\rm{i}}}

\newcommand{\Calg}{{C_{\mathtt{alg}}(p)}}

\newcommand{\CalgM}{{C_{\mathtt{alg},\mathtt M}(p)}}

\newcommand{\CM}{{C_0}}

\newcommand{\jjap}[1]{\lfloor #1 \rfloor}

\definecolor{aqua}{RGB}{10,150,200}

\makeatletter

\begin{document}

\title{{Sub-exponential stability for the Beam equation}}

\date{}

\author{Roberto Feola}
\address{\scriptsize{Dipartimento di Matematica e Fisica, Universit\`a degli Studi RomaTre, Largo San Leonardo Murialdo 1, 00144}}
\email{roberto.feola@uniroma3.it}

\author{Jessica Elisa Massetti}
\address{\scriptsize{Dipartimento di Matematica e Fisica, Universit\`a degli Studi RomaTre, Largo San Leonardo Murialdo 1, 00144}}
\email{jessicaelisa.massetti@uniroma3.it}


\begin{abstract}   
We consider a one-parameter family of beam equations with Hamiltonian non-linearity in one space dimension under periodic boundary conditions.
In a unified functional framework we study the 
long time evolution of  initial data in  two categories  of differentiability: (i) a subspace of 
Sobolev regularity, (ii) a subspace of infinitely many differentiable functions which is strictly contained in the Sobolev space but which strictly contains the Gevrey one.
In both cases we prove exponential type times of stability. 
The result holds for almost all mass parameters and it is obtained by combining normal form techniques
with a suitable Diophantine condition weaker than the one proposed by Bourgain.
This is the first result of this kind in Sobolev regularity for a degenerate equation, where only one parameter
is used to tune the linear frequencies of oscillations.
\end{abstract}  

  
\maketitle
\tableofcontents

\section{Introduction}

In this paper we consider  the one dimensional beam equation
\begin{equation}\label{eq:beam}
 \partial_{tt}\psi+\partial_{xxxx} \psi+\mathtt{m}\psi+f(\psi)=0\,,
\end{equation}
 where $\psi=\psi(t,x)$,  
 $x\in \mathbb{T}:=\R/2\pi\Z$ and $\mathtt{m}\in [1,2]$.
 The nonlinearity $f(\psi)$ has the form
 \begin{equation}\label{hypFF}
 f(\psi):=(\partial_{\psi}F)(\psi)
 \end{equation}
 for some function $F(y)$ which is real analytic in $y$ in a neighbourhood of $y=0$.
We shall assume that $F$ has a zero in $y=0$ and, 
by analyticity, for some $R>0$ we have
\begin{equation}\label{nonlineF}
F(y)=\sum_{d=3}^{\infty}F^{(d)}y^{d}\,, 
\quad |F|_{R}:=\sum_{d=3}^{\infty}|F^{(d)}| R^d<+\infty\,.
\end{equation}
We are interested in stability times of initial data, belonging to some 
Hilbert subspaces $E$ of $L^{2}(\T,\R^2)$. 
Let $\psi_0 := \psi(0)$ and $\psi_1:= \partial_t\psi(t)_{|t=0}$ be the initial conditions of respectively position and velocity of $\psi$ at time $t=0$,  by classical local (in time) well-posedness theory
we know that given initial data $|(\psi_0,\psi_1)|_{E}\leq \delta$,
the solution exists for a certain time $T=T(\delta)>0$ 
depending only on  $\delta$.
At least in the case $\delta\ll1$ we are interested in understanding the \emph{optimal} time of stability  of solutions , i.e. we want to give a lower bound  on $T(\delta)$ 
which is the supremum of times 
$t$ such that for any $|(\psi_0,\psi_1)|_{E}\leq \delta\ll1$
one has $(\psi(t,x), \partial_{t}\psi(t,x))\in E$ with 
$|(\psi(t,\cdot), \partial_{t}\psi(t,\cdot))|_{E}\leq 2\delta$.
We refer to this time as \emph{stability time}.
Since the nonlinearity in \eqref{eq:beam} is quadratic, local theory provides
the trivial lower bound $T(\delta)\gtrsim\delta^{-1}$.
Moreover, since we are working on a compact domain, no dispersive effects can help to control the behaviour of solutions for longer times. 
A fruitful approach, in this case, 
is to reduce the ``size'' of the non linearity through a convenient normal form analysis.

In this line of thoughts, a fundamental feature of equation \eqref{eq:beam} is that, under a convenient variables' change, we can write it as a Hamiltonian system whose corresponding   Hamiltonian has an elliptic fixed point at the origin. 
Passing to the Fourier side and in appropriate elliptic coordinates $u_j $ the Hamiltonian has the form 
\begin{equation}\label{marconi1}
H = \sum_{j\in \Z} \omega_j |u_{j}|^{2}+O(u^3)\,,\qquad 
\omega_{j}:=\omega_{j}(\mathtt{m})=\sqrt{j^4+\mathtt{m}}\,,
\end{equation}
where $O(u^3)$ denotes a non linearity with a zero at the origin of order at least $3$.
In this Hamiltonian view point,  through suitable symplectic change of coordinates, we shall pull 
 the Hamiltonian $H$ back to 
 \emph{Birkhoff Normal Form} (BNF)
\[
\widetilde{H}=\sum_{j\in \Z} \omega_j |u_{j}|^{2}+\mathfrak{Z}+\mathfrak{R}\,,
\] 
where $\mathfrak{Z}$ depends only on the ``actions''  $|u_j|^2$ (and does not affect the dynamics)
while $\mathfrak{R}$ has an \emph{high} degree of homogeneity $\sim O(|u|^{\suca+2})$ 
for some natural $\suca\gg1$. 
{Then, the natural time of stability of the flow of $\widetilde{H}$ becomes $T(\delta)\sim O(\delta^{-\suca})$.}\\
The crucial difficulties in this approach
 regard the regularity 
of the phase space
of initial data and interactions among linear frequencies of oscillations $\omega_j$.

The problem of \emph{long time stability} has been widely studied 
in particular for Sobolev initial data.
For instance, we mention the seminal works \cite{Bambusi:2003} for the Klein-Gordon equation, and Bambusi-Gr\'ebert \cite{Bambusi-Grebert:2006}, where
polynomial bounds for a wide class of {\it tame-modulus} PDEs are proved. More precisely, it is shown that for any $\suca\gg 1$ there exists $p(\suca) $ (tending to infinity as $\suca\to \infty$) such that for all $p\ge p(\suca)$ and all $\delta-$small initial data in $H^p$ one has 
$T\ge C(\suca,p)\delta^{-\suca}$, provided $\delta<\delta_0(\suca,p)$.
  Similar results were also proved for the Klein-Gordon equation on tori and Zoll manifolds in \cite{DS0},\cite{DS},\cite{BDGS}. More recently, similar results have been obtained also for nonlinearities containing derivatives, see \cite{Yuan-Zhang},\cite{Delort-2009},\cite{Delort-2015},\cite{Berti-Delort},\cite{FI}. 
 The above mentioned results are deeply based on 
 the requirement that $\mathtt{N}$-waves interactions are \emph{non-resonant}.
 In other words one should impose some  \emph{diophantine conditions}
 on the vector $\omega=(\omega_{j})_{j\in\Z}$ of linear frequencies of oscillations
 in order to ensure suitable lower bounds on the quantity 
 $\omega\cdot\ell$ with $\ell\in \Z^{\Z}$, 
 $|\ell|\leq \mathtt{N}$. Such quantities arise as the eigenvalues  of an appropriate linear operator 
 that must be inverted at each step of the Birkhoff procedue.
 These arithmetic conditions are typically achieved by exploiting the presence of some 
 ``parameter'' which modulates the linear frequencies, either ``internal", such as the \emph{mass} 
 parameter $\mathtt{m}$ for the Klein-Gordon or Beam equations (see \eqref{marconi1}),
 the capillarity of the fluid in the case of water waves, or ``external", as convolutions and multiplicative potentials for the Schr\"odinger equation for example. 
  However, there are cases in which it is not possible to get
such a lower bound at any order $|\ell|\le \mathtt{\suca}$ for any $\suca\in\N$,
 for instance equations posed in high dimensional \emph{generic} tori. Nevertheless
 normal form approach has been successfully applied also to obtain   ``partial results'', i.e.
 time of stability $T(\delta)\sim O(\delta^{-\bar\suca})$ for some \emph{fixed} $\bar\suca\geq1$.
 We quote for instance
 \cite{Delort:2009vn}, \cite{FGI20} on the Klein-Gordon equation, \cite{Imekraz2016}, \cite{BFGI2021}
 on the Beam equation in high space dimension,
 \cite{FM2022} on the Schr\"odinger on generic tori, 
 \cite{HIT2016}, \cite{IT2017},
 \cite{ID2019}, \cite{BFP2022}, \cite{BFF2021} on the water waves equation
 (see also \cite{FIM2022} on a different fluid model).
  
 All the results mentioned above regard polynomial stability times in Sobolev spaces. 
Passing from polynomial estimates to exponential-type ones 
is not trivial and it is related to the regularity of initial data.  
Faou and Gr\'ebert  in \cite{Faou-Grebert:2013} made a first step forward in this direction, 
by considering the case of analytic initial data,
proving super-exponential bounds of the form $T\ge e^{\ln(\frac1\delta)^{1+b}}$,
$b>0$, 
for classes of NLS equations in $\T^d$. See also\cite{Cong}. \\
The ``time-regularity" connection emerges also in finite dimension, where long time behavior of initial data is carried out through Nekoroshev theory, which gives information over exponentially or superexponentionally long times in the analytic category (see \cite{Nekhoroshev_1977,Lochak_1992,Bounemoura_Fayad_Niederman_2020,Guzzo_Chierchia_Benettin_2016}). In contrast, polynomial stability times are proven in finite differentiability settings, see \cite{Bounemoura_2011} for $C^k, k\in\N$ quasi-convex Hamiltonians and the recent \cite{BMM:hal}, where optimal polynomial stability times are proved for the H\"older steep ones. Also in finite dimension, a sharp  BNF theory  can be constructed near elliptic equilibria, to get exponential or super-exponential times of stability nearby (see \cite{Bounemoura_Fayad_Niederman_2020_2} and references therein).

For the 1-d NLS with convolution potential, a recent achievement is represented by \cite{BMP:CMP}
(see also \cite{BMP:linceiStab}), 
where Biasco, Massetti, and Procesi prove exponential-type times of stability, \emph{both} in Sobolev \emph{and} Gevrey category, by introducing a suitable functional setting allowing the optimality of time also in finite regularity spaces.  It is worth to mentioning that a key ingredient for this result comes from their Diophantine condition, firstly introduced by Bourgain in \cite{Bourgain:2005}, that allows a control on small divisors which is \emph{uniform} w.r.t. the dimension of the support of the frequencies. The presence of a convolution potential $V\ast$, which provides an infinite sequence of parameters $(V_j)_{j\in\Z}\in\ell^\infty$ to modulate the frequencies, plays a fundamental role in guaranteeing such arithmetic conditions, and this translates in optimal lower bounds in the divisors, independently of the iterative step.
Afterwards, the flexibility of the functional frame of \cite{BMP:CMP}, developed for the stability study, 
turns out to be 
 pivotal  in the study of existence of almost-periodic tori for the same model of NLS \cite{BMP:AHP} 
in Gevrey regularity. We refer also to the recent result \cite{BMP:weak} 
concerning \emph{weak} and \emph{Sobolev} almost periodic solutions
(see also \cite{BMP:linceiAlmost} for a simple case study).

Following this line of thoughts, in this paper we shall investigate whether it is possible  
to adapt the functional setting introduced in \cite{BMP:CMP} 
to the degenerate case of equation \eqref{eq:beam}, where \textit{only one} physical parameter, 
the mass $\tm$, has to be used for frequency's modulation. 

The underlying motivation is to construct a \emph{degenerate KAM theory}
(in the sense of \cite{Russmann01})
for infinite dimensional invariant tori on a model like equation \eqref{eq:beam}.
A milestone step in this direction is to 
understand precisely the type of diophantine conditions
 one is able to impose by moving just $\tm$ in the frequencies \eqref{marconi1}. 
 
 \medskip
Let us be more precise and introduce our main results.
We introduce here 
the functions spaces we are working on:
\begin{equation}\label{normnorma}
\begin{aligned}
H^{s,p}:=\Big\{
\psi(x)&=\sum_{j\in\Z}\psi_j e^{{\rm i} jx}\in L^{2}(\T,\C)\; :\; 
\|\psi\|_{s,p}^2:=
\sum_{j\in\Z}|\psi_j|^2\lfloor j\rfloor^{2p} e^{2s\lambda(j)}<+\infty
\Big\}\,,
\end{aligned}
\end{equation}
for $s\geq 0$, $p>1/2$ where the weight-function $\lambda:\R\to\R^+$ is defined as\footnote{note that the weight is sub-linear in the sense $\lambda(y_1 + y_2) \le \lambda(y_1) + \lambda(y_2)$ for any $y_1,y_2\in\R$.}
\begin{equation}\label{es:fgrowth}
\begin{aligned}
\lambda(y)&:= (\ln(2+\langle y\rangle))^\mathtt{q}\,, 
\quad 
1< \mathtt{q} \leq 2\,,
\\
\langle j\rangle&:=\max\{1,|j|\}\,,
\qquad 
\lfloor j\rfloor:={\rm max}\{2,|j|\}\,,\quad j\in \Z\,.
\end{aligned}
\end{equation}

Using the Fourier representation we can identify
\begin{equation}\label{Foucoeff}
 \psi(x)=\sum_{j\in\Z}\psi_je^{{\rm i}j x}\in L^{2}(\T;\C)\,\qquad 
\psi_j=\frac{1}{2\pi}\int_{\T}\psi(x)e^{-{\rm i}jx}{\rm d}x
\end{equation}
with its Fourier coefficients, i.e.\footnote{We denote by $\ell^2(\R)$ 
the subspace of $\ell^{2}(\C)$ made of sequences  $(\psi_j)_{j\in\Z}$ 
such that $\psi_{j}=\overline{\psi}_{-j}$.}
\begin{equation}\label{identif1}
L^{2}(\T,\C)\ni \psi(x) \mapsto \psi=(\psi_j)_{j\in\Z}\in \ell^{2}(\C)\,.
\end{equation}
We are interested  in understanding the \emph{stability times} 
of initial data $(\psi_0,\psi_1)$ belonging to some
subspace of $H^{2}(\T,\R)\times L^{2}(\T,\R)$, which are small w.r.t. an appropriate norm. 

\medskip

\noindent
{\begin{remark}
Note that, for $s=0$,
the space $H^{0,p}$ is the standard Sobolev space $H^{p}(\T,\C)$. 
While for $s>0$, 
the weight
$\lambda$ in \eqref{es:fgrowth} being logarithmic, 
the space $H^{s,p}$ is strictly larger than the space of \emph{Gevrey} functions. 
However, we are able to guarantee a lifespan for the solutions 
which is sub-exponential
(but super-polynomial) in $1/\delta$, where $\delta$ is the size of initial conditions, both in the purely Sobolev and sub-exponential categories $H^p$ and $H^{s,p}$ respectively, see Corollary \ref{coroOttimo} and Theorem \ref{main:subexp}.
\end{remark}}

\medskip

\noindent
In the following  we denote by  $\meas: [1,2] \to \R^+$  the Lebesgue probability measure.

\begin{theorem}{\bf (Sobolev stability).}\label{main:sobol}
Let $s=0$, 
$p>1+2^6(36)^2$ and fix any $0<\gamma<1$. 
There is a positive measure set $\mathfrak{M}_{\gamma}\subset[1,2]$
with $\meas([1,2]\setminus\mathfrak{M}_{\gamma})=O(\gamma)$, 
 an absolute constant $\mathtt{c}>0$ 
such that for any $\mathtt{m}\in \mathfrak{M}$
the following holds.
For any 
\begin{equation}\label{smalldelta1Sob}
0<\delta\leq\delta_{\SO}\gamma^{\mathtt{c}p}\,,
\qquad \delta_{\SO}:=\frac{R}{2^{5}|F|_{R}}\,,
\end{equation}
and any initial datum {$(\psi_0,\psi_1)\in H^{0,p+1}\times H^{0,p-1}$}
satisfying 
\begin{equation}\label{main:smallcondSob}
\|\psi_0\|_{0,p+1}+\|\psi_1\|_{0,p-1}\leq \frac{\delta}{4}\,,
\end{equation}
the solution $(\psi(t), \partial_t\psi(t))$ of \eqref{eq:beam} with 
$(\psi(0), \partial_t\psi(0))=(\psi_0,\psi_1)$ exists and satisfies
\begin{equation}\label{boundsol1Sob}
\|\psi(t)\|_{0,p+1}+\|\partial_t\psi(t)\|_{0,p-1}\leq 8\delta\,, \qquad \forall\;  |t|\leq T_0\,,
\end{equation}
with 
\begin{equation}\label{longtime1Sob}
T_0\geq\frac{R\gamma^{\mathtt{c}p^2}}{2 |F|_{R}\delta}
\left(\frac{\delta_{\SO}}{\delta}\right)^{\frac{1}{\mathtt{c}}(p-1)^{1/3}}
\,.
\end{equation}
\end{theorem}

%

As a consequence of the theorem above  we get the following.

\begin{corollary}{\bf (Sobolev stability: optimization).}\label{coroOttimo}
 Let  $\delta_{\SO}, \mathtt{c}>0$ 
 be the constants of Theorem \ref{main:sobol}, for any 
 $$
 \delta \leq \bar\delta:=\delta_{\SO} {\gamma^b}\,,\qquad 
\mathtt{b}:=24\mathtt{c}^{2}\big[2^6 (36)^{2}\big]^{5/3}
 $$
and any $(\psi_0,\psi_1)$
satisfying 
\begin{equation}\label{main:smallcondSobcoro}
\|\psi_0\|_{0,p+1}+\|\psi_1\|_{0,p-1}\leq \frac{\delta}{4}\,,
\qquad p=p(\delta):=
1+\left(\frac{1}{24\mathtt{c}^2\ln(1/\gamma)} \ln\big(\frac{\delta_{\SO}}{\delta}\big)\right)^{3/5}
\end{equation}
the solution $(\psi(t), \partial_t\psi(t))$ of \eqref{eq:beam} with 
$(\psi(0), \partial_t\psi(0))=(\psi_0,\psi_1)$ exists and satisfies
\begin{equation}\label{boundsol1coro}
\|\psi(t)\|_{0,p+1}+\|\partial_t\psi(t)\|_{0,p-1}\leq 8\delta\,, \qquad \forall\;  |t|\leq T_0\,,
\end{equation}
with 
\begin{equation}\label{longtime1Sobcoro}
T_0\geq\frac{R}{2 |F|_{R}\delta}
\exp\Big\{
\frac{\mathtt{c}(\ln(1/\gamma))^{-1/5})}{(24\mathtt{c}^2)^{6/5}}(\ln(\delta_{\SO}/\delta) )^{1+\frac{1}{5}}
\Big\}\,.
\end{equation}

\end{corollary}

{Concerning initial data with sub-exponential decay, belonging to $H^{s,p}\, s>0$, we prove the following stability result. }

\begin{theorem}{\bf (Sub-exponential stability).}\label{main:subexp}
Let $p>1+1/2$, $s>0$, $1<\mathtt{q}\leq2$ and fix any $\gamma>0$. 
There is a positive measure set $\mathfrak{M}_{\gamma}\subset[1,2]$
such that $\meas([1,2]\setminus\mathfrak{M}_{\gamma})=O(\gamma)$, 
an absolute constant $\mathtt{c}>0$ 
and
constants $C_i=C_i(p,R) >0$, $i=1,2,3$,
such that for any $\mathtt{m}\in \mathfrak{M}$
the following holds.
For any
\begin{equation}\label{smalldelta0}
0<\delta\leq \delta_{\SE}:=\min\Big\{\frac{1}{C_1 |F|_{R}}
\exp\exp\pa{
-\big(\frac{\mathtt{c}}{\gamma^{4}s} \big)^{\frac{1}{\mathtt{q}-1}}
}, 
\frac{1}{C_2|F|_{R}}
\Big\}\,,
\end{equation}
and any $(\psi_0,\psi_1)\in H^{s,p+1}\times H^{s,p-1}$
satisfying 
\begin{equation}\label{main:smallcond}
\|\psi_0\|_{s,p+1}+\|\psi_1\|_{s,p-1}\leq \frac{\delta}{4}\,,
\end{equation}
the solution $(\psi(t), \partial_t\psi(t))$ of \eqref{eq:beam} with 
$(\psi(0), \partial_t\psi(0))=(\psi_0,\psi_1)$ exists and satisfies
\begin{equation}\label{boundsol1}
\|\psi(t)\|_{s,p+1}+\|\partial_t\psi(t)\|_{s,p-1}\leq 8\delta\,, \qquad \forall\;  |t|\leq T_0\,,
\end{equation}
with 
\begin{equation}\label{longtime1}
T_0\geq C_{3}\frac{\delta_{\SE}}{\delta}
\exp\pa{\frac{1}{2}\ln(\delta_{\SE}/\delta)
\Big(\gamma^4 s\mathtt{c}^{-1}\ln\ln(\delta_{\SE}/\delta)\Big)^{\frac{\mathtt{q}-1}{2}}
}\,.
\end{equation}
\end{theorem}

\subsection*{Remarks on Theorem \ref{main:sobol}-\ref{main:subexp} and Corollary \ref{coroOttimo}}
Some remarks are in order.

\vspace{0.5em}
\begin{enumerate}[a)]
\item \emph{The lack of parameters}. {As mentioned above, in order to put the Hamiltonian associated to equation \eqref{eq:beam} in a suitable Birkhoff Normal Form, one must require arithmetic conditions on the linear frequencies $(\omega_j)_{j\in\Z}$. The Diophantine condition \`a la Bourgain reads like\footnote{Here as usual for integer vectors we denote $|\ell|:=\sum_{i\in\Z}|\ell_i|$.}
\begin{equation*}
\mathtt{D}_{\gamma,\tB}:=
\Big\{ \omega\in \R^{\Z} : |\omega\cdot\ell|\geq\prod_{n\in\mathbb{Z}} 
\frac{\gamma}{(1+|\ell_{n}|^{2}\langle n\rangle^{2})^{\tau}}\,,\; \forall \ell\in\Lambda: 0 < |\ell| < \infty \Big\}\,,
\end{equation*}
where  $\langle n\rangle:=\max\{1,|n|\}$ for any $n\in \mathbb{Z}$, $\gamma, \tau>0$,  
and $\Lambda$ is a suitable \emph{non-resonant} sub-lattice of $\Z^{\Z}$
which, in the applications, depends on the frequencies $\omega$.
In the case of $\omega_j\sim j^2 + V_j$ for an infinite sequence of $(V_j)_{j\in\Z}\in\ell^\infty$ as in the model considered in \cite{BMP:CMP}, 
it has been proved that, for a positive measure set of $V_{j}$, the frequency vector
$\omega=(j^2+V_j)_{j\in\Z}$ belongs to the Diophantine set $\tD_{\gamma,\mathtt{B}}$
with
the non-resonant sub-lattice  $\Lambda:=\Z^{\Z}\setminus\{0\}$,
 $\gamma$ and $\tau$ being fixed positive \emph{absolute constants}.
In that case, this translates into uniformity of the scheme w.r.t. the 
number of frequencies $\omega_{j}$
that are lighted up, and sharp estimates on small divisors.
In contrast, in our degenerate case where the frequencies have the form 
$\omega_{j}=\sqrt{j^{4}+\mathtt{m}}$,
the definition of the Diophantine set $\tD_{\gamma,\mathtt{B}}$ is more delicate.
The non-resonant sub-lattice $\Lambda$  is \emph{strictly} contained  in $\Z^{\Z}\setminus\{0\}$
and consists of 
$\ell\in \Z^{\Z}$ such that
the function $\omega\cdot\ell$ in \emph{not identically zero} as function of the mass $\mathtt{m}$.
We refer to Section \ref{sec:small div} for more details on this, see definition \eqref{restrizioni indici}. 
In addition to this,
the degeneracy of equation \eqref{eq:beam}, involving only one parameter for modulating \emph{any} subset of frequencies, obliges $\tau = \tau(\td(\ell))$ to strongly depend on the ``dimension'' $\td(\ell)$ of $\ell$ and $\gamma\rightsquigarrow \gamma^{\td(\ell)}$, $\td(\ell)$ being the number of non zero elements of $\ell\in\Z^\Z$ (see Section \ref{sec:small div} for the precise definitions). This of course reflects in worst estimates on divisors, 
which are not uniform in the length $|\ell|$.
Remarkably, we are able to control these 
a priori bad estimates 
by taking full advantage of the flexibility of our norms, which are extremely effective in keeping sharp track of the constants during the iteration of BNF (see in particular Section \ref{sez: belle norme} with Lemmata \ref{giggina}-\ref{crescenza}). Thus, despite the presence of only one modulation-parameter, Theorem \ref{main:sobol} entails the same qualitative estimates as the ones in \cite{BMP:CMP}. As a byproduct a final optimization step can be performed as in \cite{BMP:CMP}, through which we get optimal estimates of exponential type. This is the first result of this kind for a degenerate equation, in Sobolev regularity. }

\vspace{0.5em}
\item { Since the norm $\|\cdot\|_{0,p}$ 
is equivalent  to the norm  
$\|\cdot\|_{H^p}:=\|\cdot\|_{L^{2}}+\|\partial_x^p \cdot\|_{L^{2}}$ (with equivalence constants independent of $p$), we can reformulate
 condition \eqref{main:smallcondSob} as  
\[
\|\psi_0\|_{H^{p+1}}+\|\psi_1\|_{H^{p-1}}\leq \frac{\delta}{8}\,,
\qquad 
\|\psi_0\|_{L^{2}}+\|\psi_1\|_{L^{2}}\leq 2^{-p}\frac{\delta}{8}\,.
\]
We also remark that in \cite{Bambusi-Grebert:2006} (see for instance Proposition $1.1$ in \cite{BMP:CMP})
the required smallness condition w.r.t. the classical Sobolev norms reads 
$\|\psi_0\|_{H^{p+1}}+\|\psi_1\|_{H^{p-1}}\leq \delta\le  \delta_S p^{-3p}$, which is
more restrictive with respect to our condition  
$\delta\leq \delta_S \g^{p}$.
In other words the  condition above allows us to work on a slightly bigger ball in $H^{p}$, provided we impose 
a stronger condition on the $L^2$-norm.
However our condition is still more restrictive w.r.t. the one in Theorem $1.2$ of \cite{BMP:CMP}.
This is again a consequence of our small divisors estimates (see  item $a)$). 
}

\vspace{0.5em}
\item The function  in \eqref{main:smallcondSobcoro} defining $p=p(\delta)$ is explicitly invertible, so one can rephrase the stability time in terms of the regularity, namely, fixing $H^{0,p}$
with $p$ sufficiently large, all the  initial data in the ball of radius $\delta = \delta(p)$ remains  
in the ball of radius 
$8\delta$ for times of order 
\[
T_0\sim \exp\Big\{
\mathtt{c}(p-1)^{5/3}\big(1+(p-1)^{1/3}\big)\ln(1/\gamma) 
\Big\}\sim \left(\frac{1}{\delta(p)}\right)^{1+(p-1)^{1/3}}\,.
\]
\item We stress the fact that dealing with functions belonging to the space $H^{s,p}$ with $\exp(\ln^\mathtt{q}(\cdot))$-decay is more delicate than dealing with the classical sub-exponential Gevrey ones. In particular the monotonic character of the norm together with the estimates on the solution of the Homological equation require a more refined analysis. See Lemmata \ref{constance beam sub-immersion} and 
\ref{lem:constance2SE} for instance.  
We restrict here to this more difficult case. One can easily recover the Gevrey regularity case 
using the corresponding Lemmata in \cite{BMP:CMP} and following verbatim our proof.
In this case we expect a time of stability like \eqref{longtime1} with only one logarithm at the exponent.

\end{enumerate}

\textbf{Acknowledgments.} We thank L. Biasco and M. Procesi for enlightening  discussions 
and suggestions.  
We also thank D. Bambusi and L. Corsi for very useful comments.
The authors have been  supported by the  research project PRIN 2020XBFL ``Hamiltonian and dispersive PDEs" of the Italian Ministry of Education and Research (MIUR). J.E. Massetti also acknowledges the support of the INdAM-GNAMPA research project ``Chaotic and unstable behaviors of infinite-dimensional dynamical systems".

%
%
%
%
%
%
%

\section{Spaces of Hamiltonians and conjugacies}\label{sec:2}

\subsection{Hamiltonian structure of the beam equation}
\label{sec:hamstructurebeam}

We shall deal with the stability problem through a Birkhoff Normal Form apporach, taking advantage of the Hamiltonian structure enjoyed by the beam equation. In fact, 
by introducing the variable
$v=\partial_{t}\psi$, solutions of equation \eqref{eq:beam} 
correspond to the flow of
\begin{equation*}\label{eq:beam2}
X_\cB := 
\left\{\begin{aligned}
\partial_{t}\psi&=v 
\\ 
\partial_{t}v&=-\omega^2\psi-f(\psi)\,,
\end{aligned}\right.
\end{equation*}
where $\omega$ 
is the Fourier multiplier defined by linearity as
\begin{equation}\label{omegoneBeam}
\omega e^{{\rm i} j\cdot x}=\omega_{j} e^{{\rm i} j\cdot x}\,,
\qquad 
\omega_{j}:=\sqrt{|j|^{4}+\mathtt{m}}\,,
\qquad
\forall \,j\in \mathbb{Z}\,,\quad \mathtt{m}\in[1,2]\,.
\end{equation}
Observe that
\[
X_{\cB}\equiv J\nabla {H_{\mathbb{R}}}(\psi,v)
=
J\left(\begin{matrix}
\partial_{\psi}H_{\mathbb{R}}(\psi,v)\\
\partial_{v}H_{\mathbb{R}}(\psi,v)
\end{matrix}\right)\,,
\quad 
J=\sm{0}{1}{-1}{0}
\]
where the Hamiltonian $H_\R: H^{2}(\mathbb{T};\mathbb{R})\times L^{2}(\mathbb{T};\mathbb{R})\to \R$ is defined through
\begin{equation}\label{BeamRealHam}
 H_{\mathbb{R}}(\psi,v)=
\int_{\mathbb{T}}
\big(
\frac{1}{2}v^{2}+\frac{1}{2}(\omega^{2}\psi) \psi+F(\psi)
\big){\rm d}x\,,
\end{equation}
and
$\nabla := (\partial_\psi,\partial_v)$ denotes its $L^{2}$-gradient.
Indeed we have 
\begin{equation}\label{eq:1.14bis}
\mathrm{d}H_{\mathbb{R}}(\psi,v)\cdot W
=
\Omega_{\mathbb{R}}(X_{H_{\mathbb{R}}}(\psi,v),W)
\end{equation}
for any 
$W \in H^{2}(\mathbb{T};\mathbb{R})\times L^{2}(\mathbb{T};\mathbb{R})$, 
where $\Omega_{\mathbb{R}}$ 
is the non-degenerate symplectic form
\[
\Omega_{\mathbb{R}} := \int_{\T}d\psi \wedge dv\,dx\,,
\qquad
\Omega_{\mathbb{R}}(W_1,W_2) 
= \int_{\T}  d\psi \wedge dv (W_1,W_2)\, dx 
= \int_{\mathbb{T}}(\psi_1v_2-v_1\psi_2)dx\,,
\]
for any 
$ W_1 = (\psi_1,v_1), W_2=(\psi_2, v_2) 
\in H^{2}(\mathbb{T};\mathbb{R})\times L^{2}(\mathbb{T};\mathbb{R})$.
The Poisson brackets between two Hamiltonian 
$H_{\mathbb{R}}, G_{\mathbb{R}}: 
H^{2}(\mathbb{T};\mathbb{R})\times L^{2}(\mathbb{T};\mathbb{R})\to \mathbb{R}$
are defined in the classical manner as
\begin{equation}\label{realpoisson}
\{H_{\mathbb{R}},G_{\mathbb{R}}\}
:=\Omega_{\mathbb{R}}(X_{H_{\mathbb{R}}},X_{G_{\mathbb{R}}})\,.
\end{equation}

\noindent
Let  now
\[
\cR = 
\set{(u^+, u^-)\in H^{1}(\mathbb{T};\mathbb{\C})\times  H^{1}(\mathbb{T},\C) \, : u^- = \bar{u}^+ }\,,
\] 
and let us define  the linear isomorphism 
\[
\cC:H^{2}(\mathbb{T};\mathbb{R})\times  L^{2}(\mathbb{T};\mathbb{R})
\to 
H^{1}(\mathbb{T};\mathbb{\C})\times  H^{1}(\mathbb{T};\mathbb{\C}) \cap \cR \,,
\]

\begin{equation}\label{beam5}
\vect{\psi}{v} \mapsto \mathcal{C}\vect{\psi}{v} = \vect{u}{\bar{u}}\,,\quad \mathcal{C}:=
\frac{1}{\sqrt{2}}\left(
\begin{matrix}
\omega^{\frac{1}{2}} & {\rm i} \omega^{-\frac{1}{2}}\\
\omega^{\frac{1}{2}} & -{\rm i} \omega^{-\frac{1}{2}}
\end{matrix}
\right)\,,
\end{equation}
where $\omega$ is the Fourier multiplier defined in \eqref{omegoneBeam}.
The vector field $X_\cB$ is then pushed forward to the new Hamiltonian one 
\begin{equation}\label{eq:beamComp}
\cC_* X_{\cB} = X_H = (\dot u, \dot{\bar u}) \,,
\quad \quad 
\dot{u} 
= - {\rm i}\omega u- \frac{{\rm i} }{\sqrt 2}
\omega^{-1/2}f\left(\omega^{-1/2}\left(\frac{u+\bar u}{\sqrt 2}\right)\right) 
= - {\rm i}\partial_{\bar{u}}H(u,\bar{u})
\end{equation}
where 
$\partial_{\bar{u}}=(\partial_{\Re u}+{\rm i} \partial_{\Im u})/2$,
$\partial_{u}=(\partial_{\Re u}-{\rm i} \partial_{\Im u})/2$
and 
\begin{equation}\label{beamHam}
H(u,\bar{u})=H_{\mathbb{R}}(\mathcal{C}^{-1}\vect{u}{\bar{u}})
=
\int_{\mathbb{T}}\bar{u}\, \omega u\ \mathrm{d}x 
+\int_{\mathbb{T}} F\Big( \frac{\omega^{-1/2}(u+\bar{u})}{\sqrt{2}}\Big)\ {dx}\,.
\end{equation}
The (complex) induced $2$-form is 
\begin{equation}\label{complex form}
\Omega:= (\cC^{-1})^*\Omega_\R = \int_{\T}{\rm{i}} du\wedge d\bar u\,dx\,,
\end{equation}
which yields, for any 
$w_1 = (\xi , \bar\xi), w_2 
= (\eta, \bar\eta) \in H^{1}(\mathbb{T};\mathbb{\C})
\times  H^{1}(\mathbb{T};\mathbb{\C}) \cap \cR$,
\begin{equation}\label{symcompform}
\Omega (w_1,w_2) = \int_{\T} \xi \bar{\eta} - \bar\xi\eta \, dx\, \, \in\R
\end{equation}
and intrinsically defines the Hamiltonian through
\begin{equation}\label{diffHam}
\Omega(X_{H}(u), w) = \mathrm{d}H(u)\cdot w                                                                                                                                                                                                                                                                                                                                                                                                                                                                                                                                                                                                                                                                                                                                                                                                                                                                                                                                                                                                                                                                                                                                                                                                                                                                                                                                                                                                                                                                                                                                                                                                                                                                                                                                                                                                                                                                                                                                                                                                                                                                                                                                                                                     
\end{equation}
for any $w$.
Accordingly, we set the (complex) Poisson brackets as
\[
\set{H, G} := 
\Omega(X_H, X_G)
\]  
where  $H=H_{\mathbb{R}}\circ\mathcal{C}^{-1}$  
(resp $G=G_{\mathbb{R}}\circ\mathcal{C}^{-1}$ ) 
and $X_H = \cC_*X_{H_\R}$ (resp $X_G = \cC_*X_{G_\R}$)
which yields\footnote{
Note that the naturality of the Poisson brackets 
holds $\set{H_\R, G_\R}\circ \cC^{-1} = \set{H,G}$.
}
\begin{equation}\label{Poissonbrackets}
\begin{aligned}
\{H,G\} 
={\rm i} \int_{\mathbb{T}}
\big(\partial_{u}G\partial_{\bar{u}}H-
\partial_{\bar{u}}G\partial_{u}H\big) \mathrm{d}x\,.
\end{aligned}
\end{equation}

\subsection{Fourier's representation} 
The stability result will be a consequence of a normalization procedure 
that amounts in transforming the Hamiltonian 
into a suitable normal form (the Birkhoff normal form). 
In order to set the convenient functional setting, 
we shall rather work in the space of sequences 
that correspond to the above functional spaces, 
by systematically identifying  $ L^2(\T,\C)$ with  the Banach space 
$\cF(\ell^2(\C))$ of $2\pi$-periodic functions  (recall \eqref{Foucoeff})
$u(x) = u(x,t) = \sum_{j\in\Z} u_j e^{ijx}$ 
such that their Fourier's coefficients $(u_j)_{j\in\Z}\in \ell^2(\C)$. 

\noindent
Then,  the  Hamiltonian  in \eqref{beamHam} reads 
\begin{equation}\label{beamHamFourier}
H(u,\bar{u})= 
\sum_{j\in\Z} \omega_{j}|u_{j}|^{2} + \sum_{p=3}^{\infty}
\sum_{\substack{j_i\in\Z,\s_i\in\{\pm\} \\
\sum_{i=1}^{p}\s_i j_i=0}} 
F_{j_1,\ldots,j_p}^{\s_1\ldots\s_p}
u_{j_1}^{\s_1}\cdots u_{j_p}^{\s_p}\,=: \sum_{j\in\Z} \omega_{j}|u_{j}|^{2} + \tH_{\ge 3}\,,
\end{equation}
where we used the analyticity in the neighborhood 
of the origin of $F$ to expand 
the second integral in \eqref{beamHam}
for some coefficients 
$|F_{j_1,\ldots,j_p}^{\s_1\ldots\s_p}|\lesssim C^{p}$ for some $C>0$, and,  rearranging the sum in multi-index notation we set

$$ \tH_{\geq 3}(u,\bar u) = \sum_{\substack{\al,\bt\in\N^\Z\\ |\al| + |\bt|\ge 3 \\ \sum_j j(\al_j - \bt_j)=0}} 
H_{\al,\bt} u^{\al}{\bar{u}}^{\bt}. $$

\noindent
Accordingly, by defining 
\begin{align*}
d u_j = \frac{1}{\sqrt 2}\pa{d x_j+ \imm d y_j}\,,
\quad 
d \bar u_j = \frac{1}{\sqrt 2}\pa{d x_j- \imm d y_j}\,,
\\  
\frac{\partial}{\partial  u_j} =  
\frac{1}{\sqrt 2}\pa{\frac{\partial}{\partial x_j} 
- \imm \frac{\partial}{\partial y_j}}\,,
\quad  
\frac{\partial}{\partial  \bar u_j} =  
\frac{1}{\sqrt 2}\pa{\frac{\partial}{\partial x_j} 
+ \imm \frac{\partial}{\partial y_j}}\,,
\end{align*}
the corresponding $2$-form and Hamiltonian vector field read
\begin{equation}\label{symp form seq}
\imm \sum_{j\in\Z} du_j\wedge d\bar{u}_j\,, 
\quad \quad 
X_H^{(j)} = -\imm \frac{\partial }{\partial{\bar{u}_j}}H(u) \,.
\end{equation}
Note that when some real analytic $H$ admits 
a holomorphic extension $\widehat H$ 
on some ball of radius $r>0$ in $\ell^2(\C)$ that is  
\[ 
(u_+, u_{-})\in B_r(\ell^2(\C))\times B_r(\ell^2(\C)) \to \widehat H(u_+,u_-)\, 
:\quad \quad 
H(u)=\widehat H (u,\bar u)\,,
\]  
then it admits a Taylor expansion
\[
\widehat H(u_+,u_-)  = \sum^\ast_{\al,\bt\in\N^\Z} H_{\al,\bt}u_+^\al u_-^\bt\,,
\]
where we denote by $\sum^\ast$ the sum 
restricted to those $\al,\bt: |\al|+ |\bt|<\infty$.

\noindent
One can see that
\[
\frac{\partial}{\partial \bar u_j} H(u) 
= 
\frac{\partial \widehat H(u_+,u_-)}{\partial u_{-,j}} \Big\vert_{u_+=\bar u_-=u}\,.
\]

From now on we shall pass to the Fourier side and work on spaces of weighted sequences.

\noindent
Let  $\mathtt w=(\mathtt w_j)_{j\in\Z}$ be  the  real sequence  
(recall \eqref{es:fgrowth})
\begin{align}
&\tw=\tw(s,p):= \pa{\lfloor j\rfloor^{ p}e^{ s\lambda(j)}}_{j\in\Z}\,,
\qquad \text{($\SE$) \;\; Sub-exponential  case}\,,
\label{peso sub}
\\
&\tw=\tw(p):=\tw(p,0)=\pa{\lfloor j\rfloor^{p}}_{j\in\Z}\,,
\qquad \text{($\SO$) \;\; Sobolev case}\,,
\label{peso sob} 
\end{align}
and  let us set the 
Hilbert space
\begin{equation}\label{pistacchio}
\th_{\mathtt w}:=
\Big\{u:= \pa{u_j}_{j\in\Z}\in\ell^2(\C)\; : \; \abs{u}_{\tw}^2:= 
\sum_{j\in\Z} \mathtt w_j^2 \abs{u_j}^2 < \infty
\Big\}\,,
\end{equation} 
endowed with the scalar product
\begin{equation}\label{scalarW}
(u,v)_{\th_\tw}:=\sum_{j\in\Z} \mathtt w_j^2 u_j \bar v_j\,,\qquad u,v\in\th_{\tw}\,.
\end{equation}
Moreover, given $r>0$, we denote by $B_r(\th_{\mathtt w})$ 
the closed
ball of radius $r$ centred at the origin of $\th_{\mathtt w}$.

\vspace{0.5em}
\noindent
In the following we shall systematically identify $2\pi$-periodic 
functions with their Fourier coefficients, writing $\mathtt{h}_{\mathtt{w}(s,p)}$ instead of $H^{s,p}$.
In particular, with abuse of notation, we shall write (recall \eqref{normnorma})
\begin{align*}
 \|u\|_{s,p}&=|u|_{\tw(s,p)}\,, \qquad\qquad\quad
 \forall \, u\in H^{s,p}\simeq \mathtt{h}_{\tw(s,p)}\,, & (\SE)
\\
\|u\|_{p}&:=\|u\|_{0,p}=|u|_{\tw(p)}\,, \qquad\forall \, u\in H^{0,p}\simeq \mathtt{h}_{\tw(p)}\,\,.
&(\SO)
\end{align*}

\subsection{Spaces of Hamiltonians}
In this section we 
introduce a suitable graded Poisson algebra of  Hamiltonians
that we need in order to prove our main abstract result.

\noindent
Let $\tw$ as in \eqref{peso sub} or \eqref{peso sob} and let
$\star:\th_{\tw} \times \th_{\tw} \to \th_{\tw}$ be the convolution operation defined as
\[
\pa{f,g} \mapsto f\star g :=\Big(
\sum_{j_1+j_2=j
} 
f_{j_1}g_{j_2}\Big)_{j\in \Z}\,.
\]
The map $\star: \pa{f,g} \mapsto f\star g$ is continuous in the following sense:
\begin{lemma}\label{algebra}
For $p>1/2$ we have 
\begin{align}
\norm{f\star g}_{\tw{(p,s)}} &\le \Calg \norm{f}_{\tw{(p,s)}}\norm{g}_{\tw{(p,s)}}\,,
\quad\quad \forall f,g\in \th_{\tw(s,p)}\,,\label{stimacalgcalg}
\\	
|f\star g|_{\tw(p)}&\leq \CalgM |f|_{\tw(p)} |g|_{\tw(p)}\,,
\qquad\qquad \forall f,g\in \th_{\tw(p)}\,,\label{stimacalgcalgMM}
\end{align}
where 
\begin{equation}\label{costanti algebra}
\Calg := 8^p \Big(\sum_{i\in\Z} \jap{i}^{-p}\big)^{1/2}\,, 
\qquad \quad 
\CalgM:= \sqrt{2}\sqrt{2 + \frac{2p+1}{2p-1}}\,.
\end{equation}
\end{lemma}

\begin{proof}
The proof works verbatim as the one in Lemma 5.5 of \cite{BMP:CMP}  
where $\jap{j}^p e^{s\jap{j}^\teta +  a |j|}\rightsquigarrow \jjap{j}^p e^{s\lambda(j)}$ 
 for the $\abs{\,\cdot \,}_{p,s,0}\equiv \abs{\,\cdot\,}_{\tw_{p,s}}$ 
 norm (just noticing that $\jap{j}\le \jjap{j}$ and the sublinearity of $\lambda$).
 The case of the Sobolev norm $\|\cdot\|_{p}$ is the same.
\end{proof}

\vspace{0.5em}
\noindent
By endowing the space $\th_{\tw}$ in \eqref{pistacchio}
with the symplectic structure induced by the symplectic form 
$\Omega$ in \eqref{symp form seq}, 
we introduce the following class of Hamiltonians.

\begin{definition}{\bf (Admissible Hamiltonians).}\label{def:admiHam}
Let $r>0$ and consider a Hamiltonian
$ H : B_r(\th_{\mathtt w}) \to \R$
such that there exists a pointwise  
absolutely convergent power series expansion\footnote{As usual given 
a vector $k\in \Z^\Z$, $|k|:=\sum_{j\in\Z}|k_j|$.}
\begin{equation}\label{HamPower}
H(u)  = 
\sum_{\substack{
\al,\bt\in\N^\Z\,, \\2\leq |\al|+|\bt|<\infty} }
\!\!\!
H_{\al,\bt}u^\al \bar u^\bt\,,
\qquad
u^\al:=\prod_{j\in\Z}u_j^{\al_j}\,.
\end{equation}
We say that $H$ as in \eqref{HamPower}
is \emph{admissible} if  the following properties hold: 
\begin{enumerate}
\item \emph{Reality condition}:
\begin{equation}\label{real}
H_{\al,\bt}= \overline{ H}_{\bt,\al}\,,\qquad \forall\, \al,\bt\in\mathbb{N}^{\Z}\,;
\end{equation}
\item \emph{Momentum conservation}:	
\begin{equation}\label{momento}
H_{\al,\bt}\neq0\quad \Rightarrow\quad
\pi(\al - \bt) := \sum_{j\in\Z}j\pa{\al_j - \bt_j}= 0\,.
\end{equation}
\end{enumerate}	
\end{definition}

\noindent
Finally, given two admissible Hamiltonians $H,G$
the Poisson brackets are given by 
\begin{equation}\label{poipoisson}
\{H,G\}={\rm i}
\sum_{j\in\Z}
\big(\partial_{u_j}G\partial_{\bar{u}_j}H-
\partial_{\bar{u}_j}G\partial_{u_j}H\big)\,.
\end{equation}

\noindent
Let 
$\mathcal{A}_{r}(\th_{\mathtt w})$ be the the space of admissible
 Hamiltonians such that the \co{majorant}
\begin{equation}\label{etamag}
\und { H} (u):= 
\sum_{(\alpha,\beta)\in\cM} \abs{{H}_{\alpha,\beta}}u^{\alpha}\bar{u}^{\beta}
\end{equation} 
is point-wise  absolutely convergent on $B_r(\th_{\mathtt w})$,
where we set
\begin{equation}\label{mass-momindici}
\mathcal{M}:=
\left\{
(\alpha,\beta)\in \mathbb{N}^{\Z}\times\N^{\Z} : \pi(\al - \bt)= 0, 
|\alpha|+|\beta|<\infty
\right\}\,,
\end{equation}
and introduce the following class of Hamiltonians.
 
\begin{definition}{\bf (Regular Hamiltonians).}\label{Hreta}
We denote by $\cH_{r}(\th_{\mathtt w})$ 
the subspace of 
$\cA_{r}(\th_{\mathtt w})$ of  Hamiltonians $H$ such that
\begin{equation}\label{normaHamilto}
|H|_{\cH_{r}(\th_{\mathtt w})}
=
|H|_{r,\tw}
:=
r^{-1} \Big(
\sup_{\norm{u}_{{\mathtt w}}\leq r} 
\norm{{X}_{{\underline H}}}_{{\mathtt w}} 
\Big) < \infty\,.
\end{equation}
\end{definition}

\begin{remark}
\label{embeddignsspaces}
We remark the following facts:

\noindent
$\bullet$
Given two positive sequences $\tw = \pa{\tw_j}_{j\in\Z},\tw' = (\tw'_j)_{j\in\Z}$
we write that $\tw\leq \tw'$ if the inequality holds
point wise, namely
\[
\tw\leq \tw' \quad
\iff\quad
\tw_j\leq \tw'_j\,,\ \ \ \forall\, j\in\Z\,.
\]
In this way if $r'\le r$ and $\tw\leq \tw'$ 
then $B_{r'}(\th_{\mathtt w'}) \subseteq B_r(\th_\tw)$.

\noindent
$\bullet$ If a Hamiltonian $H$ satisfies \eqref{real},  it means that it 
is real analytic in the real and imaginary part of $u$.

\noindent
$\bullet$ If a Hamiltonian $H$ satisfies \eqref{momento} then it 
Poisson commutes with $\sum_{j\in\Z} j\abs{u_j}^2$.

\noindent
$\bullet$ The Hamiltonian functions being defined modulo a constant term, 
we shall assume without loss of generality that $H(0)=0$. 
\end{remark}
Finally,  let us consider  a regular Hamiltonian 
$S\in\cH_r(\th_\tw)$ and  its flow $\Phi_{S,t}$
which is well-defined (see Lemma \ref{ham flow} for details), 
and let 
\begin{equation}\label{quadraticBeam}
D_{\omega}:=\sum_{j\in\Z}\omega_{j}|u_j|^{2}\,,
\end{equation}
and its flow $\phi_{\omega, t}$,  
where $\omega_j = \sqrt{j^4 + \mathtt{m}}$.

\begin{definition} \label{def:adjoperator}
$(i)$ The \emph{Lie derivative} of $H$ along the flow of $S$ is given by 
\begin{equation}\label{liederiv}
L_S H =  \frac{d}{dt}_{|t=0} \phi_{S,t}^* H(u) = \frac{d}{dt}_{|t=0} H(\phi_{S,t}(u))\,.
\end{equation}

\noindent
$(ii)$ Given $H\in \cH_{r}(\th_{\mathtt w})$  we define the \emph{adjoint action} 
of the Hamiltonian $D_{\omega}$ 
as the Lie derivative operator 
\begin{equation}\label{def:adjaction}
L_{\omega} H:= \frac{d}{dt}_{|t=0} \phi_{\omega,t}^* H 
= 
\sum_{(\al,\bt)\in \mathcal{M}} 
-{\rm i}\big(\omega \cdot(\al-\bt)\big)H_{\al,\bt}u^{\al}\bar{u}^{\bt}\,,
\end{equation}
where $\mathcal{M}$ is the set of indexes defined in \eqref{mass-momindici}.
\end{definition}

\begin{remark}{\bf (Change of variables).}\label{rmk:change}
Along the paper we shall study how a Hamiltonian 
$H$ behaves  along the flow of 
a given regular Hamiltonian $S$.
In fact one has
\[
\begin{aligned}
\frac{d}{dt} H(\phi_{S,t}(u))=dH(\phi_{S,t}(u))\cdot X_{S}(\phi_{S,t}(u))
&=
\Omega(X_{H}(\phi_{S,t}(u)), X_{S}(\phi_{S,t}(u)))
\\&
=\{H,S\}\circ \phi_{S,t}(u)
\stackrel{\eqref{liederiv}}{=}(L_{S}H)\circ \phi_{S,t}(u)\,.
\end{aligned}
\]
Then \eqref{liederiv} corresponds to $L_{S}H=\{H,S\}$.
Moreover,
from the formula above,
one formally  deduces that the well-known ``Lie expansion''
\[
H(\phi_{S,t}(u))= e^{L_{S}}H
=
\sum_{k=0}^{\infty} \frac{t^k}{k!} L_{S}^{k} H\,,
\qquad 
L_{S}^{k}H:= \{L_{S}^{k-1}H,S\} \,, \forall k\geq1\,,\;\; L_{S}^0={\rm Id}\,.
\]
\end{remark}

\noindent
In our work the crucial point is 
that all the dependence on the parameters $r,\tw$ 
of the  norm in \eqref{normaHamilto}  can be {\it encoded}
in the coefficients
\begin{equation}\label{persico}
c^{(j)}_{r,\tw}(\al,\bt)
:=
r^{|\al|+|\bt|-2}
\frac{\tw_j^2}{\tw^{\al+\bt}}\,,\qquad \tw^{\al+\bt} 
= \prod_{j\in\Z}\tw_j^{\al_j+\bt_j}\,,
\end{equation}
defined for any $\al,\bt\in\N^\Z$ and $j\in\Z$.
In view of our choices of the weights in \eqref{peso sub} and \eqref{peso sob}
we have that the coefficients in \eqref{persico} have the following form:
\begin{align}
&\SE) \;\; {\rm case:}\qquad 	
c^{(j)}_{r,\tw}(\al,\bt)=r^{|\al|+|\bt|-2}\frac{\lfloor j\rfloor^{ 2p}}{\prod_{i\in\Z}\lfloor j\rfloor^{p(\al_i+\bt_i)}}
e^{ s\big(2\lambda({j})-\sum_{i\in\Z}(\al_i+\bt_i)\lambda({i}) \big)}\,;\label{coeffSE}
\\
&\SO)\;\; {\rm case:}
\qquad 	
c^{(j)}_{r,\tw}(\al,\bt)=
r^{|\al|+|\bt|-2}\frac{\lfloor j\rfloor^{ 2p}}{\prod_{i\in\Z}\lfloor i\rfloor^{p(\al_i+\bt_i)} }\,.\label{coeffSobo}
\end{align}
	
\begin{remark}{\bf (Basic embeddings of spaces of Hamiltonians).}\label{rmk:basicembHam}
Recalling Remark \ref{embeddignsspaces} one can notice that
 if $r'\le r  $ and $\tw\leq \tw'$ then
$ \cA_{r}(\th_\tw) \subseteq \cA_{r'}(\th_{\tw'})$. 
In the following (see Proposition \ref{crescenza}) 
we give conditions on the parameters 
that $(r,\tw),(\rs,\tw')$ (with $\rs\le r$) which ensure the (not trivial) inclusion 
$ \cH_{r}(\th_\tw) \subseteq \cH_{\rs}(\th_{\tw'})$. 
That condition will be given in terms 
of the ratio of the coefficients $c^{(j)}_{r,\tw}(\al,\bt)$, $c^{(j)}_{r',\tw'}(\al,\bt)$.
\end{remark}

\subsection{Properties of regular Hamiltonians}\label{sez: belle norme}
We now collect some properties 
of the norm in \eqref{normaHamilto}. 
For any $H\in \cH_{r}(\th_{\mathtt w})$ we define  a map
\[
B_1(\ell^2)\to \ell^2 \,,\quad y=\pa{y_j}_{j\in \Z }\mapsto 
\pa{Y^{(j)}_{H}(y;r,\tw)}_{j\in \Z}
\]
by setting
\begin{equation}\label{giggina}
Y^{(j)}_{H}(y;r,\tw) := 
\sum_{(\al,\bt)\in\mathcal{M}} |H_{\al,\bt}| \frac{(\al_j+\bt_j)}{2}c^{(j)}_{r,\tw}(\al,\bt) y^{\al+\bt-e_j}
\end{equation}
where  $e_j$ is the $j$-th basis vector in $\N^\Z$, while the coefficient
$c^{(j)}_{r,\tw}(\al,\bt)$ is defined right above in \eqref{persico}.
The following properties give a systematic way for 
computing the norm of a given Hamiltonian 
and its relation w.r.t. another one.

By Lemma $3.1$ in \cite{BMP:CMP} (
see also Lemmata $3.3$, $3.4$ and $A.1$ in \cite{ProStolo})
we have the following.

\begin{lemma} \label{norme proprieta} 
Let $\ri, r'>0,$ $\twi,\twf\in \R_+^\Z$. 
The following properties hold.

\vspace{0.3em}
\noindent
$(1)$ The norm of $H$ can be expressed as
\begin{equation}\label{ypsilon}
\abs{H}_{r,\tw}= 
\sup_{|y|_{\ell^2}\le 1}\abs{Y_H(y;r,\tw)}_{\ell^2}\,.
\end{equation}

\vspace{0.3em}
\noindent
$(2)$ Given $H^{(1)}\in \cH_{r',\twf}$ and $H^{(2)}\in \cH_{\ri,\twi}$\,,
such that $\forall\, \al,\bt\in \N^\Z$ and  $j\in \Z$ with $\al_j+\bt_j\neq 0$ 
one has
\begin{equation}\label{alberellobello}
|H^{(1)}_{\al,\bt}| c^{(j)}_{\rf,\twf}(\al,\bt)  \le 
c|H^{(2)}_{\al,\bt}| c^{(j)}_{\ri,\twi}(\al,\bt),
\end{equation}
for some $c>0$, then
\[
|H^{(1)}|_{\rf,\twf}
\le c |H^{(2)}|_{\ri,\twi}\,.
\]
\end{lemma}
The following proposition gathers the immersion properties of the norm
$|\cdot|_{r,\tw(p,s)}$ with respect to the parameters $p,s$. 

\begin{proposition}{\bf (Monotonicity).}\label{crescenza}
For any $p>1, s>0$ the norm $\norm{\cdot}_{r,\tw}$ is monotone increasing in $r$. Moreover, letting $r>0$  the following holds. 

\noindent
$\SE)$ Consider $\mathtt{w}$ as in \eqref{peso sub}. 
For any $\s,s>0$ we have
\begin{equation}\label{emiliaparanoica}
|H|_{r,\tw(s+\s,p)} \le  |H|_{r,\tw(s,p)}\,.
\end{equation}

\noindent
$\SO)$ Consider $\mathtt{w}$ as in \eqref{peso sob}.
For any $p' >0$, $p>1$,
we have
\begin{equation}\label{emiliapara2}
|H|_{r,\tw(p+p')} \le |H|_{r,\tw(p)}\,.
\end{equation}
\end{proposition}
For the moment we omit the proof of the proposition 
above and we refer the reader to Appendix \ref{proofofcrescenza}.

\vspace{0.7em}
\noindent
By Proposition $2.1$ in \cite{BMP:CMP}
and Lemma $3.5$ in \cite{ProStolo} we have that 
the scale $\{\cH_{r}(\th_{\mathtt w})\}_{r>0} $ is a  
Banach-Poisson algebra in the following sense.
\begin{proposition}{\bf (Poisson Brackets).}\label{fan}
For $0 <\rho\leq r$  we have
\begin{equation}\label{commXHK}
|\{F,G\}|_{r,\tw}
\le 
4\Big(1+\frac{r}{\rho}\Big)
|F|_{r+\rho,\tw}
|G|_{r+\rho,\tw}\,.
\end{equation}
\end{proposition}

\subsection{Graded Poisson structure and conjugations }
We start by defining a degree decomposition  which endows $\cH_{r}(\th_{\tw})$ 
with a graded Poisson algebra structure.
\begin{definition}{\bf (Scaling degree).}\label{def:scalingdegree}
Given $\td \in\N$, let  $\cH^{(\td)}$  be the vector space 
of homogeneous polynomials of degree $\td + 2$, 
that is admissible  Hamiltonians of the form
\[
\sum_{\substack{(\al,\bt)\in\mathcal{M} \\ |\al| + |\bt| = \td + 2}} H_{\al,\bt} u^{\al} \bar{u}^{\bt}\,.
\]
We shall say that a Hamiltonian $H$ has 
\emph{scaling degree $\ge\td=\td(H)$} if 
\[
H\in\cH^{(\ge\td)} = \cH^{(\td)} \oplus_{h>\td} \cH^{(h)}\,.
\]
Accordingly, we shall define projections 
associated with this direct sum decomposition and write
\begin{equation}
\label{proietto}
\Pi^{(\td)} H = \sum_{
\substack{(\al,\bt)\in\mathcal{M}\\
|\al| + |\bt| = \mathtt{d} + 2}} 
H_{\al,\bt}u^{\al} \bar{u}^{\bt}\,,
\quad \quad 
\Pi^{(>\td)} H = \sum_{
\substack{(\al,\bt)\in\mathcal{M} \\ |\al| + |\bt| > \mathtt{d} + 2}} 
H_{\al,\bt}u^{\al} \bar{u}^{\bt}\,.
\end{equation}
We say that $\td(0)=+\infty$.
\end{definition}

\begin{remark}
With this definitions, quadratic Hamiltonians have scaling degree $0$. 
Essentially  $H$ has scaling degree $\td$ if and only if  
it has a zero of order $\td+2$ at zero.
\end{remark}
Definition \ref{def:scalingdegree} produces a graded Poisson algebra structure. 
Moreover one has the following result.

\begin{lemma}\label{gasteropode} 
The projection operators are continuos. In particular, the following hold.

\noindent
$(i)$ If $H\in\cH_r(\th_\tw)$ with $\td(H)=\td$, then one has 
\begin{equation}\label{bound proiezione}
\abs{\Pi^{(\td)} H}_{r,\tw} \le \abs{H}_{r,\tw}\,, \qquad \abs{\Pi^{(>\td)} H}_{r,\tw} \le \abs{H}_{r,\tw}\,.
\end{equation}

\noindent
$(ii)$	If $H\in \cH_{r}(\th_\tw)$ with $\td(H)\geq\td$, then for all $\rf\le r$ one has
\begin{equation*}
\abs{H}^{\wc}_{\rf,\tw} \le \pa{\frac{\rf}{r}}^{\td} \abs{H}^{\wc}_{r,\tw}\,.
\end{equation*}
\end{lemma}

\begin{proof} 
$(i)$ We only prove the first in \eqref{bound proiezione}. 
The estimate for $\Pi^{(>\mathtt{d})}$ follows similarly.

\noindent
By absolute convergence of $H$, we can rearrange the terms and write 
$H = \sum_{\td\ge 0} \Pi^{(\td)} H$
where each term reads as in \eqref{proietto}. 
In general, for any Hamiltonian $H$ by definition of majorant norm, we have  that
\[
X^{(j)}_{\und{H}}(u) = -{\rm i} \sum_{\substack{(\al,\bt)\in\mathcal{M}}} 
|H_{\al,\bt}| \bt_j u^\al\bar{u}^{\bt - e_j}\,,
\]
\noindent
which trivially yields 
$|X_{\und{H}}(u) |_{\tw} \le |X_{\und{H}}(\und{u})|_{\tw}$, 
where $\und{u} = (|u_j|)_{j\in\Z}$.  
Observing that when evaluating the supremum of 
$X_{\und{H}}$ over $|u|_\tw \le r$ we can restrict to the case in which 
$u = (u_j)_{j\in\Z}$ has positive real components, we get
\begin{equation*}
\abs{H}_{r,\tw} 
= r^{-1}\sup_{\abs{u}_\tw \leq r} \abs{\pa{W^{j}_H(u)}_{j\in\Z}}_\tw\,,  
\, \quad \quad 
W^{j}_H(u) = \sum_{\substack{(\al,\bt)\in\mathcal{M}}} 
\abs{H_{\al,\bt}}\bbt_j \abs{u}^{\al + \bt - e_j}\,.
\end{equation*}
Then inequality \eqref{bound proiezione} follows trivially 
from $W^{j}_{\Pi^{(\td)} H}(u) \le W^{j}_H(u)$. 

\noindent
$(ii)$ By item (2) in Lemma \ref{norme proprieta}, 
it suffices to observe that 
\[
\frac{c^{j}_{r',\tw(\al,\bt)}}{c^{j}_{r,\tw}(\al,\bt)} 
= \pa{\frac{r'}{r}}^{|\al| + |\bt| - 2} \le \pa{\frac{r'}{r}}^{\td + 2 - 2}\,.
\]
This concludes the proof.
\end{proof}

\begin{remark}\label{rmk:scalaPoi}
If $F$ and $G$ are Hamiltonians in $\cH_{r}(\th_{\mathtt w})$
with \emph{scaling degree} $\mathtt{d}_1, \mathtt{d}_2$ respectively, then
the Poisson $\{F,G\}$ has scaling degree equal to 
$\mathtt{d}_1+\mathtt{d}_2$. 
In general, if the scaling degrees are $\ge \mathtt{d}_1, \mathtt{d}_2 $, 
then the scaling degree of  $\{F,G\}$ is $\ge \mathtt{d}_1 + \mathtt{d}_2$.
\end{remark}

\noindent
The following Lemma guarantees 
that the flow of a regular Hamiltonian is well-posed on $\th_{\tw}$. 
Moreover it shows how regular Hamiltonians changes 
under conjugation through flows.

\begin{lemma}{\bf (Hamiltonian flow).}\label{ham flow}
Let $0<\rho< r $,  and $S\in\cH_{r+\rho}(\th_{\mathtt w})$ with 
\begin{equation}\label{stima generatrice}
\abs{S}_{r+\rho,\tw} \leq {\delta}:= \frac{\rho}{8 e\pa{r+\rho}}\,, 
\end{equation} 
Then the time $1$-Hamiltonian flow 
$\Phi^1_S: B_r(\th_{\mathtt w})\to B_{r + \rho}(\th_{\mathtt w})$  
is well defined, analytic, symplectic with
\begin{equation}\label{pollon}
\sup_{u\in  B_r(\th_{\mathtt w})} \norm{\Phi^1_S(u)-u}_{\th_{\mathtt w}}
\le
(r+\rho)  \abs{S}_{r+\rho,\tw}\leq \frac{\rho}{8 e}\,.
\end{equation}
Moreover, for any $H\in \cH_{r+\rho}(\th_{\mathtt w})$ we have that
$H\circ\Phi^{1}_S= e^{ L_{S} } H\in\cH_{r}(\th_{\mathtt w})$ 
and
\begin{align}\label{tizio}
\abs{\es H}_{r,\tw} & \le 2 \abs{H}_{r+\rho,\tw}\,,
\\
\abs{\pa{\es - \id}H}_{r,\tw}
&\le  \delta^{-1}
\abs{S}_{r+\rho,\tw}
\abs{H}_{r+\rho,\tw}\,,\label{caio}
\\
\abs{\pa{\es - \id - \set{S,\cdot}}H}_{r,\tw} &\le \frac12 \delta^{-2}
\abs{S}_{r+\rho,\tw}^2
\abs{H}_{r+\rho,\tw}\,.\label{sempronio}
\end{align}
More generally for any $h\in\N$ and any sequence  
$(c_k)_{k\in\N}$ with $| c_k|\leq 1/k!$, we have 
\begin{equation}\label{brubeck}
\abs{\sum_{k\geq h} c_k L^k_S\pa{H}}_{r,\tw} \le 
2 |H|_{r+\rho,\tw} \big(|S|_{r+\rho,\tw}/2\delta\big)^h\,.
\end{equation}
\end{lemma}

\begin{proof}
Follows verbatim by Lemma $2.1$ in \cite{BMP:CMP} with $\eta=0$
and ${\rm ad }_{S} \rightsquigarrow L_{S}$.
\end{proof}

The following classical Lemma gives \emph{a priori} estimates 
on the time of definition of  flows generated by a wider class of Hamiltonians.
\begin{lemma}\label{tempotempo}
Let $\mathcal{N}\in \mathcal{A}_{r}(\tw)$ and $R\in \mathcal{H}_{r}(\th_{\tw})$
(recall Def. \ref{Hreta}) for some $r>0$.
Assume that 
\begin{equation}\label{hyp:kernel}
{\rm Re}(X_{\mathcal{N}}(v),v)_{\th_\tw}=0\,,\;\;\;\forall\, v\in \th_{\tw}\,. 
\end{equation}
Consider the dynamical system
\[
\dot{v}=X_{\mathcal{N}}(v)+X_{R}(v)\,,\;\;\;\; v(0)=v_0\,,\quad |v_0|_{\tw}\leq \frac{3}{4}r\,.
\]
Then one has
\[
\big||v(t)|_{\tw}-|v_0|_{\tw}\big|\leq \frac{r}{8}\,,\quad \forall \, |t|\leq \frac{1}{8|R|_{r,\tw}}\,.
\]
\end{lemma}
\begin{proof}
See Lemma $5.4$ in \cite{BMP:CMP}.
\end{proof}

\vspace{0.3em}
\noindent
{\bf Resonant Hamiltonians.} 
We define the \emph{resonant} subset of $\mathcal{M}$ (see \eqref{mass-momindici})  as
\begin{equation}\label{resonant set}
\mathtt{R} = \set{(\al,\bt) \in \mathcal{M}\,:\, 
\al_j = \bt_j \, \vee \al_j = \bt_{-j} \, 
\forall j\in\Z }\,,
\end{equation}
and we denote by $\cK_r(\th_\tw)$  the subset of \emph{resonant} Hamiltonians, i.e.
\begin{equation}\label{ker}
\cK_r(\th_\tw) = \Big\{H\in\cH_{r}(h_{\tw})\, : 
\sum_{(\al,\bt) \in\mathtt{R}} H_{\al,\bt} u^{\al}\bar{u}^{\bt}\Big\}\,.
\end{equation}

\begin{remark}\label{rmk:resonant set}
Let $(\al,\bt)\in \mathtt{R}$ and $\ell=\al-\bt$. The condition in \eqref{resonant set}
implies that 
\[
\ell_{j}+\ell_{-j}=\al_j-\bt_j+\al_{-j}-\bt_{-j}\equiv0\,,\qquad \forall\, j\in\Z\,.
\]
\end{remark}

\noindent
The following results regards a fundamental properties of resonant Hamiltonians.
\begin{lemma}{\bf (Flows of Kernel Hamiltonians).}\label{lem:kernel}
Let 
\[
\tf: \th_\tw \to \th_\tw\,, \quad u\mapsto \tf(u) =  (\tf_{j} u_j)_{j\in\Z}\,,
\qquad \tf_{j} = \tf_{-j}\,,\qquad \forall j\in\Z\,.
\]
 Then any $H\in\cK_r(\th_\tw)$ poisson commutes with $\norm{\tf(u)}^2_\tw$. 
\end{lemma}

\begin{proof}
Let $H\in \cK_r(\th_\tw)$ (see \eqref{ker}).
Using \eqref{poipoisson} one has that 
\[
\begin{aligned}
\big\{H, \sum_{j\in\Z} \tw_j^2\,\tf_j^2 |u_j|^2\big\} 
&= \sum_{\al,\bt \in \mathtt{R} }H_{\al,\bt} \{ u^{\al}\bar{u}^{\bt},  
\sum_{j\in\Z} \tw_j^2\,\tf_j^2 |u_j|^2 \} 
= \sum_{\al,\bt \in \mathtt{R} }H_{\al,\bt} 
\sum_{\ell\in\Z} (\al_\ell - \bt_\ell)\tw_\ell^2\tf_{\ell}^2 u^{\al}\bar{u}^{\bt}\,.
\end{aligned}
\]
Since since $\tf$ is even in $\ell$ (as well as $\tw$) 
we have that the right hand side of the equation above reads 
\[
\sum_{\al,\bt \in \mathtt{R} }H_{\al,\bt} 
\sum_{\ell > 0} (\al_\ell + \al_{-\ell} - \bt_{-\ell} - \bt_\ell)
\tf(\ell)^2 u^{\al}\bar{u}^{\bt} = 0\,,
\]
where in the last inequality we used Remark \ref{rmk:resonant set}. 
This concludes the proof.
\end{proof}
\begin{remark}\label{rmk:kernel}
Let $\mathcal{N}\in \mathcal{K}_{r}(\th_{\tw})$. 
We have that
\[
2{\rm Re}(X_{\mathcal{N}}(v),v)_{\th_{\tw}}=\{\mathcal{N}(v), |v|^{2}_{{\tw}}\}=0\,,
\]
by Lemma \ref{lem:kernel}.  Hence we deduce that
the vector field $X_{\mathcal{N}}$ satisfies the condition \eqref{hyp:kernel}
in Lemma \ref{tempotempo}.
\end{remark}

\begin{remark}\label{oddKerham}
Let $H\in \mathcal{K}_r(\th_\tw)$ and assume $H=\Pi^{(\td)}H$
for some $\td\geq1$. 
By using \eqref{resonant set}, \eqref{ker}
one can check that $H\equiv0$ if $\td$ is \emph{odd}.
\end{remark}

\section{Small divisors}\label{sec:small div}
Recalling that 
\begin{equation}\label{dispLaw}
\begin{aligned}
\omega&:=\omega(\mathtt{m}):=(\omega_{j})_{j\in \mathbb{Z}}\in \mathbb{R}^{\mathbb{Z}}\,,
\\
\omega_{j}&:=\omega_{j}(\mathtt{m}):
=\sqrt{|j|^{4}+\mathtt{m}}\,,\qquad j\in \mathbb{Z}\,,\qquad \mathtt{m}\in[1,2]\,,
\end{aligned}
\end{equation}
and  the resonant set 
\[
\mathtt{R} = \set{(\al,\bt) \in \mathcal{M}\,:\, 
\al_j = \bt_j \, \vee \al_j = \bt_{-j} \, 
\forall j\in\Z }\,,
\]
we now give arithmetic conditions on non- resonant indexes belonging to the following set:
\begin{equation}\label{restrizioni indici}
\Lambda:=\big\{\ell\in \mathbb{Z}^{\mathbb{Z}}\; : \; 
\ell:=\alpha-\beta\,,\; \forall (\alpha,\beta)\in \mathtt{R}^{c}\big\}\,.
\end{equation}

Finally, given a vector $\ell:=(\ell_i)_{i\in \mathbb{Z}}\in \Lambda$
consider the set
$\mathcal{A}(\ell):=\{i\in \mathbb{Z} \,:\, \ell_i\neq0\}$.
We define the map
\begin{equation}\label{polloarrosto}
\ell\mapsto \mathtt{d}:=\mathtt{d}(\ell)\in \mathbb{N}
\end{equation}
where $\mathtt{d}(\ell):=\#\mathcal{A}(\ell)$. We call $\mathtt{d}(\ell)$ the \emph{cardinality} of $\ell$,
i.e. the number of components of $\ell$ which are different form zero.

\begin{definition}{\bf (Diophantine frequencies).}\label{defi:totalesmalldiv}
%
%
Given $\gamma>0$  we denote by $\mathtt{D}_{\gamma}$ the set of 
\emph{diophantine} frequencies 
\begin{equation}\label{diofSet}
\mathtt{D}_{\gamma}:=
\Big\{ \omega\in \R^{\Z} : |\omega\cdot\ell|\geq\prod_{n\in\mathbb{Z}} 
\frac{\gamma^{\mathtt{d}(\ell)}}{(1+|\ell_{n}|^{2}\langle n\rangle^{2})^{\tau}}\,,\; \tau:=\mathtt{d}(\ell)(\mathtt{d}(\ell)+2)\,,\;\;
\forall\ell\in \Lambda \Big\}\,,
\end{equation}
where  $\langle n\rangle:=\max\{1,|n|\}$ for any $n\in \mathbb{Z}$.
\end{definition}

The key proposition of this section guarantees that, for ``almost all'' choices of the 
parameter $\mathtt{m}$ the frequency vector $\omega$
in \eqref{dispLaw} belongs to the diophantine set in \eqref{diofSet}.
Our aim is to prove the following result.
\begin{proposition}{\bf (Measure estimates).}\label{prop:meas}
There exists a positive measure set $\mathfrak{M}\subseteq [1,2]$
such that
for any $\mathtt{m}\in\mathfrak{M}$, the vector $\omega(\mathtt{m})$
belongs to the diophantine set of frequencies $\mathtt{D}_{\gamma}$ in \eqref{diofSet}.
Moreover,  there exists a positive constant $\mathtt{C}$ such that
\[
\meas([1,2]\setminus \mathfrak{M})\leq \mathtt{C}\gamma\,.
\]
\end{proposition}

The proof of the proposition above involves several argument which will be discussed below.
First of all let us define the quantity (see \eqref{mass-momindici})
\begin{equation}\label{smallDiv}
\psi(\mathtt{m},\ell):=\omega\cdot\ell\,, \qquad \forall \ell\in \mathcal{M}\,,
\end{equation}
and recall that we shall provide lower bounds on $\psi(\omega, \ell)$
only for $\ell$ belonging to the set $\Lambda$ in \eqref{restrizioni indici}.
Moreover, according to the notation \eqref{polloarrosto}, 
we can write the function in \eqref{smallDiv} as
\begin{equation}\label{pizzapomo}
\psi(\mathtt{m},\ell)=\sum_{i=1}^{\mathtt{d}}\ell_{j_i}\omega_{j_i}\,,
\qquad j_i\in\mathbb{Z}\,.
\end{equation}

\vspace{0.5em}
\noindent
{\bf Estimates of a single ``bad set''.}
We have the following.
\begin{lemma}\label{lem:vander}
For any $\ell\in \Lambda$ there exists $0\leq k\leq \mathtt{d}(\ell)-1$ such that
\begin{equation}\label{albero2}
|\partial_{\mathtt{m}}^{k}\psi(\mathtt{m},\ell)|\geq \prod_{i=1}^{\mathtt{d}}
\frac{1}{(1+|\ell_{j_i}|^{2}\langle j_i\rangle^{2})^{{\mathtt{d}(\ell)}}}\,.
\end{equation}
\end{lemma}

\begin{proof}
To lighten the notation we shall write $\mathtt{d}$ instead of $\mathtt{d}(\ell)$.\\
Given $\ell\in\Lambda$, after a reordering of the indexes we can write 
$\ell = (\bar\ell,0),$  where $\bar\ell = (\ell_{j_1},\ldots,\ell_{j_{\td}})$. 
Without loss of generality, 
we can always assume that the vector $\bar{\ell}$ satisfies 
\begin{equation}\label{Nosuperact}
j_i\neq-j_k\,,\qquad \forall\; j,k=1,\ldots,\mathtt{d}\,.
\end{equation}
Indeed,  the d-pla $({j_1},\ldots,{j_{\td}})$ can be written as
\[
(k_1,\ldots, k_{p}, q_1,-q_1, q_2, -q_2\, \ldots, q_{r} , - q_{r})\,,\qquad 0\leq p\leq \mathtt{d}
\]
for some $0\leq p\leq \mathtt{d}$ and $p+2r=\mathtt{d}$, 
where $k_i$, $i=1,\ldots, p$ satisfy \eqref{Nosuperact}. 
The small divisors has the form
\[
\omega\cdot \bar{\ell}=\sum_{i=1}^{\mathtt{d}}\omega_{j_i}\ell_{j_i}=
\sum_{i=1}^{p}\omega_{k_i}\ell_{k_i}+\sum_{i=1}^{r} \omega_{q_i}(\ell_{q_i}+\ell_{-q_i})\,.
\]
Hence we can define
\[
\tilde{\ell}=(\tilde{\ell}_{k_1},\ldots, \tilde{\ell}_{k_p}, \tilde{\ell}_{q_1}, \ldots, \tilde{\ell}_{q_r} )\,,
\qquad {\rm where} \quad \left\{\begin{aligned}
&\tilde{\ell}_{k_i}=\ell_{k_i}\,,\quad i=1,\ldots p\,,
\\
& \tilde{\ell}_{q_i}=\ell_{q_i}+\ell_{-q_i}\,,\quad i=1,\ldots, r\,.
\end{aligned}\right.
\]
Since $\ell\in \Lambda$
it is not possible that at the same time $p=0$ and $\ell_{q_i}+\ell_{-q_i}=0$
 for any $i=1,\ldots, r$. Otherwise $\ell$ is a resonant vector (recall \eqref{resonant set}
 and Remark \ref{rmk:resonant set}). 
 As a consequence up to reducing the length of $\tilde{\ell}$ to $\tilde{\td}=\td(\tilde{\ell})\leq \td(\ell) $
 (by eliminating the components for which $\ell_{q_i}+\ell_{-q_i}=0$), 
 we have obtained a vector satisfying condition \eqref{Nosuperact}
 with $\tilde{\td}\leq \td$.

Hence from now on we consider $\ell\in \Lambda$  with $\td(\ell)=\td$ and satisfying \eqref{Nosuperact}.
\noindent
Notice that, for any $k\geq1$, 
\begin{equation}\label{albero3}
\partial_{\mathtt{m}}^{k}\psi(\mathtt{m},\ell)
=
\sum_{i=1}^{\mathtt{d}}\ell_{j_i}\partial_{\mathtt{m}}^{k}\omega_{j_i}
=
\Gamma(k)\sum_{i=1}^{\mathtt{d}}\ell_{j_i}(\omega_{j_i})^{1-2k}\,,
\qquad \Gamma(k):=\frac{(-1)^{k+1}}{2^{k}}(2k-3)!!\,. 
\end{equation}
Let us define $\mathtt{a}:=(\mathtt{a}_i)_{i=0,\ldots,\mathtt{d}}\in \mathbb{R}^{\mathtt{d}}$
as
\begin{equation}\label{albero4}
\partial_{\mathtt{m}}^{k}\psi(\mathtt{m},\ell)
=
\mathtt{a}_{k+1}\,,\qquad k=0,\ldots,d\mathbb{d}-1\,.
\end{equation}
Our aim is to prove that there is at least one component of the vector $\mathtt{a}$
satisfying the bound \eqref{albero2}.
In view of \eqref{albero3} we rewrite \eqref{albero4} as
\begin{equation}\label{albero6}
\Gamma M O \ell=\mathtt{a}\,,
\end{equation}
where
\begin{equation}\label{albero10}
\begin{aligned}
\Gamma&:=\left(
\begin{matrix}
1 &\ldots & \ldots&0 \\
0 & \Gamma(1) & \ldots &\vdots\\
\vdots &\ldots & \ddots & \vdots \\
0&\ldots &\ldots & \Gamma(\mathtt{d}-1)
\end{matrix}
\right)\,,
\qquad
O:=\left(
\begin{matrix}
\omega_{j_1} &\ldots & \ldots&0 \\
0 & \omega_{j_2} & \ldots &\vdots\\
\vdots &\ldots & \ddots & \vdots \\
0&\ldots &\ldots & \omega_{j_\mathtt{d}}
\end{matrix}
\right)\,,
\\
M&:=\left(
\begin{matrix}
1 &\ldots & \ldots&1 \\
\omega_{j_1}^{-2} & \ldots & \ldots &\omega_{j_{\mathtt{d}}}^{-2}\\
\vdots &\ldots & \ldots & \vdots \\
\omega_{j_{1}}^{-2(\mathtt{d}-1)}&\ldots &\ldots &\omega_{j_{\mathtt{d}}}^{-2(\mathtt{d}-1)}
\end{matrix}
\right)\,.
\end{aligned}
\end{equation}
Notice that the matrix $M$ is a Vandermonde matrix.
Moreover using that $\ell\in \Lambda$ and that \eqref{Nosuperact} holds, 
its  determinant is given by
\[
{\rm det}(M)=\prod_{i\neq k}(\omega_{j_i}^{-2}-\omega_{j_k}^{-2})
\neq 0\,,
\]
so that the matrix $M$ is invertible.
It is also easy to check that
\[
\max_{i,k=1,\ldots,\mathtt{d}}|(M^{-1})_{i}^{k}|\leq (\mathtt{d}-1)!
\prod_{i\neq k}\frac{\omega_{j_i}^{2}\omega_{j_k}^{2}}{\omega_{j_i}^{2}-\omega_{j_k}^{2}}
\lesssim 2^{-\mathtt{d}}\mathtt{d}^{-1}\big(\prod_{i=1}^{\mathtt{d}}\omega_{j_i}\big)^{\mathtt{d}}\sim
2^{-\mathtt{d}}\mathtt{d}^{-1}\big(\prod_{i=1}^{\mathtt{d}}|j_i|^{2}\big)^{\mathtt{d}}\,.
\]
Recalling \eqref{albero10} we note
\[
\max_{i=1,\ldots,\mathtt{d}}|(\Gamma^{-1})_{i}^{i}|\leq 2^{\mathtt{d}}\,,\quad
\max_{i=1, \ldots,\mathtt{d}}|(O^{-1})_{i}^{i}|\leq 1\,.
\]
Therefore
\begin{equation}\label{albero7}
\max_{i,k=1,\ldots,\mathtt{d}}|\big((\Gamma MO)^{-1})_{i}^{k}|\lesssim
\mathtt{d}^{-1}
\big(\prod_{i=1}^{\mathtt{d}}|j_i|^{2}\big)^{\mathtt{d}}\lesssim 
\mathtt{d}^{-1}
\Big(\prod_{i=1}^{\mathtt{d}}
\frac{1}{(1+|\ell_{j_i}|^{2}\langle j_i\rangle^{2})}\Big)^{-\mathtt{d}}\,.
\end{equation}
Since by \eqref{albero6}, we have $\ell=(\Gamma MO)^{-1}\mathtt{a}$,
we deduce
\[
1\leq |\ell|\lesssim \mathtt{d}
\max_{i,k=1,\ldots,\mathtt{d}}|\big((\Gamma MO)^{-1})_{i}^{k}|\|\mathtt{a}\|_{\ell^{\infty}}\,,
\]
which, together with \eqref{albero7}, implies the bound \eqref{albero2}.
\end{proof}

Now we need the following result 
(see for example Lemma B.1 \cite{Eliasson:cetraro} ):
\begin{lemma}\label{v.112}
Let  $\mathfrak{g}(x)$ be a $C^{n+1}$-smooth 
function on the segment $[1,2] $ such that 
\[
|\mathfrak{g}'|_{C^n} =\beta \quad {\rm and}\qquad  
\max_{1\le k\le n}\min_x|\partial^k \mathfrak{g}(x)|=\sigma\,.
\]
Then  one has
\[
\meas(\{x\mid |\mathfrak{g}(x)|\leq\rho\} )\leq 
C_n \left(\beta \s^{-1}+1\right) 
(\rho \s^{-1})^{1/n}\,.
\]
\end{lemma}

\noindent
Now, for any fixed $\ell\in \Lambda$  and $\eta>0$,
we define the ``bad set'' of parameters 
\begin{equation}\label{diofSet44}
\mathcal{B}(\ell):=
\Big\{
\mathtt{m}\in [1,2] \, :\, |\omega\cdot\ell|
\leq
\prod_{n\in\mathbb{Z}}\frac{\gamma^{\mathtt{d}}}{(1+|\ell_{n}|^{2}\langle n\rangle^{2})^{\tau}} 
\Big\}
\end{equation}
with $\tau$ as in \eqref{diofSet}.
Thanks to Lemma \ref{lem:vander} we shall apply Lemma \ref{v.112}
with $n=\mathtt{d}-1$ and
\[
\begin{aligned}
&
\s\geq \prod_{i=1}^{\mathtt{d}}
\frac{1}{(1+|\ell_{j_i}|^{2}\langle j_i\rangle^{2})^{\mathtt{d}}}\,,
\quad
\rho=\prod_{n\in\mathbb{Z}}
\frac{\gamma^{\mathtt{d}}}{(1+|\ell_{n}|^{2}\langle n\rangle^{2})^{\tau}}\,,
\quad
\beta\leq\mathtt{d}!\lesssim \prod_{i=1}^{\mathtt{d}}(1+|\ell_{j_i}|^{2}\langle j_i\rangle^{2})\,.
\end{aligned}
\]
Therefore we obtain
\begin{equation}\label{misuraBad}
\meas(\mathcal{B}(\ell))\lesssim\gamma^{\frac{\mathtt{d}}{\mathtt{d}-1}}
\Big(\prod_{i=1}^{\mathtt{d}}
\frac{1}{1+|\ell_{j_i}|^{2}\langle j_i\rangle^{2}}\Big)^{\frac{\tau}{\mathtt{d}}-\mathtt{d}-1 }\,.
\end{equation}

\begin{proof}[{\bf Proof of Proposition \ref{prop:meas}}]
We define (see \eqref{diofSet} and \eqref{diofSet44})
\[
\mathcal{B}:=\bigcup_{\ell\in\Lambda}\mathcal{B}(\ell)\,,
\]
and we set $\mathfrak{M}=\mathcal{B}^{c}$.
Then the thesis follows by using the sub-additivity of the Lebesgue measure, 
the bound \eqref{misuraBad}
and by reasoning as in the proof of Lemma $4.1$ in \cite{BMP:CMP}.
\end{proof}

\begin{remark}\label{rmk:ker2}
By Remark \ref{rmk:resonant set}  for any $(\al,\bt)\in \mathtt{R}$
and $\ell=\al-\bt$ one has that $\omega\cdot\ell\equiv0$ is identically zero  for $\mathtt{m} \in[1,2]$.
On the other hand, by Proposition \ref{prop:meas}, for any $\omega\in \tD_{\gamma}$
 one has $\omega\cdot \ell\neq0$ for any $\ell\in\Lambda$. 
\end{remark}

\section{Homological equation}

Given a diophantine vector $\omega\in \mathtt{D}_{\gamma}$, 
in view of Remark \ref{rmk:ker2} and 
by formula \eqref{def:adjaction} we  deduce that
\[
L_{\omega}H=0 \qquad \Leftrightarrow\qquad  H\in \cK_r(\th_\tw)\,.
\]

\noindent
Hence the operator $L_{\omega}$ is formally invertible
when acting on the subspace  
\begin{equation}\label{range2}
\cR_r(\th_\tw)=\cK_r(\th_\tw)^{\perp}:=
\Big\{
H\in\cH_{r}(h_{\tw})\, : 
\sum_{(\al,\bt) \in\mathtt{R}^{c}} H_{\al,\bt} u^{\al}\bar{u}^{\bt}
\Big\}\,,
 \end{equation}
containing  those Hamiltonians supported on monomials
$u^{\alpha}\bar{u}^{\beta}$ with $(\alpha,\beta)\in \mathtt{R}^{c}$.
We decompose the space of regular Hamiltonians $\cH_r(\th_\tw)$
as
\[
\cH_r(\th_\tw) = \cK_r(\th_\tw)\oplus \cR_r(\th_\tw)\,,
\]
and we denote by $\Pi_{\cK}$ and $\Pi_{\cR}$
 the  continuous projections	 
 on the subspaces $\cK_r(\th_\tw)$, $ \cR_r(\th_\tw)$.
 One can note 
 \begin{equation}\label{fame}
|\Pi_{\cK}H|^\wc_{r,\tw},
 |\Pi_{\cR}H|^\wc_{r,\tw} \le |H|^\wc_{r,\tw}\,.
\end{equation}

\noindent 
Obviously, for {diophantine}  frequency,
$\cR^\wc_{r}(\th_\tw)$ and	 $\cK^\wc_{r}(\th_\tw)$
represent the range and kernel 
of  $L_\omega$ respectively.


\begin{proposition}{\bf (Inverse of the adjoint action).}\label{shulalemma}
Fix $\mathtt{N}\in \mathbb{N}$, $r>0$, $p>1$ and $s>0$.
Consider $\mathtt{w}(s,p)$ (resp. $\mathtt{w}(p)$) 
and a Hamiltonian function $f\in \mathcal{R}_r(\th_\tw)\cap \mathcal{H}^{(\tN)}$ (see Def. \ref{def:scalingdegree} and recall \eqref{range2}). 
For any $\omega\in \mathtt{D}_{\gamma}$ 
the following holds.

\smallskip
\noindent
(case $(\SE)$)
There exists an absolute constant $\mathtt{C}>0$ (independent of $\tN$) such that for any 
 $0<\s\ll1$ one has
that 
\[
|L_{\omega}^{-1}f|_{r,\mathtt{w}(p,s+\s)}\leq J_0^{\SE} |f|_{r,\mathtt{w}(p,s)}\,,
\]
where  $L_{\omega}$ is in \eqref{def:adjaction} and 
\begin{equation}\label{controlJ0}
J_0^{\SE}:= J_0^{\SE}(\s,\tN)
:=\gamma^{-4\tN} \exp\exp\Big( \Big(\frac{\tN^{2} }{\s} \mathtt{C}\Big)^{\frac{1}{\mathtt{q}-1}}  \Big)
\,.
\end{equation}
\noindent
(case $(\SO)$) 
Fix $\zeta \geq (36 \tN)^2$. There exists an absolute constant $\tC>0$ such that  
\[
|L_{\omega}^{-1} f|_{r,\mathtt{w}(p+\zeta)}\leq J_0^{\SO} |f|_{r,\mathtt{w}(p)}\,,
\]
where 
\begin{equation}\label{controlJ0caseSob}
J_0^{\SO}:=J_{0}^{\SO}(\zeta,\tN):= \, \g^{-4\tN} e^{\tC\zeta}.
\end{equation}
\end{proposition}

\begin{proof}

\noindent
{\bf Case $(\SE)$}.
Since, by hypothesis, $f$ 
belongs  to the range of the operator $L_{\omega}$, the Hamiltonian
$L_{\omega}^{-1}f$ is well-defined with coefficients given by 
\[
(L_{\omega}^{-1}f)_{\al,\bt}= \frac{f_{\al,\bt}}{-\imm \omega\cdot(\al-\beta)}\,,\qquad 
\forall \,(\al,\bt)\in\mathtt{R}^{c}\,. 
\]
Recall the coefficients in 
\eqref{coeffSE}. In view of property \eqref{alberellobello} with 
$\tw' = \tw(p,s+\s)$ and $\tw = \tw(p,s)$ and formula \eqref{def:adjaction} ,
in order to get the result it is sufficient to estimate 
 the quantity
\begin{equation}
 J_0 := 
 \sup_{\substack{j\in\Z,\,  (\al,\bt)\in\Lambda \\ 
 \al_j+\bt_j\neq 0 \\ |\al-\bt|\leq \tN+2}}
\frac{c^{(j)}_{\ri,\mathtt{w}(s+\s,p)}(\al,\bt) }
{c^{(j)}_{\ri,\mathtt{w}(s,p)}(\al,\bt) |\omega\cdot (\al-\bt)|}\,.
\end{equation}
By an explicit computation using \eqref{coeffSE} we get
\[
J_0 =  
\sup_{\substack{j\in\Z,\,  (\al,\bt)\in\Lambda \\ 
\al_j+\bt_j\neq 0  \\
 |\al-\bt|\leq \tN+2}}
\frac{e^{-\s\pa{\sum_i\lambda({i}) (\al_i+\bt_i) -2\lambda({j})}}}{\abs{\omega\cdot{\pa{\al - \bt}}}}\,.
\]
By Lemma \ref{lem:constance2SE}, 
we just have to study the case in which \eqref{divisor} holds true.
Let  $\omega\in \mathtt{D}_{\gamma}$ (recall \eqref{polloarrosto}-\eqref{diofSet}).
Since $\ell=\al-\bt$, $|\ell|\leq \tN+2$ we notice that 
$\mathtt{d}=\mathtt{d}(\ell)\leq 4\tN$ and 
$\tau=\tau(\ell)\leq 36 \tN^2$.
Therefore we have
\[
\begin{aligned}
J_0 &\leq 
\gamma^{-\mathtt{d} } \exp\left(-
\s\Big(\sum_{i} \lambda({i}) (\al_i+\bt_i) -2\lambda({j}) \Big)
+{\sum_{i\in\Z} \tau\ln(1+(\al_i-\bt_i)\langle i\rangle^{2} ) }
\right)
\\
&\stackrel{\mathclap{\eqref{constance2SE}}}{\leq}
\gamma^{- 4\tN}
\exp \left(\sum_i\left[-\frac{\s \kappa }{63} \abs{\al_i - \bt_i}\lambda(\sqrt{\jap{i}})\right]
+ 36\tN^2 \ln{\pa{1 + \pa{\al_i - \bt_i}^{2}\jap{i}^{{2}}}} \right) 
\\
&\le \gamma^{- 4\tN}  \exp{ (- 144\tN^2 \sum_i H_i(|\al_i - \bt_i|)})\,,
\end{aligned}
\]
where for $0<\s\leq 1$, $i\in \Z$ , we defined 
\[
H_i(x) := \frac{\s \kappa }{63 \times 36\tN^2} x \lambda(\sqrt{\jap{i}}) - \ln{\pa{1 + \sqrt{x\jap{i}}}}\,,
\]
where $x:= |\al_i - \bt_i| \geq 1$.
By definition of $\lambda$ (recall \eqref{es:fgrowth}),
by denoting  
\begin{equation}\label{albero1tris}
\al= \frac{\s \kappa }{63 \times 36\tN^2} \,,
\end{equation}
 we observe that there exists $X(\al)$ such that the following inequalities hold: 
\[
\al x \lambda(\sqrt{\jap{i}}) - \ln{\pa{1 + \sqrt{x\jap{i}}}} 
\ge \al \lambda(\sqrt{x \jap{i}}) - \ln{\pa{1 + \sqrt{x\jap{i}}}} 
\ge 0\,,
\quad \mbox{if}\quad \jap i \ge X^2(\al)\,. 
\]
By an explicit computation one can check that
 \begin{equation}\label{albero1bis}
 X(\alpha)=
 \exp\Big\{ 
\big(\frac{2\cdot 63\cdot 36 \tN^{2}}{\s \kappa}\big)^{\frac{1}{\mathtt{q}-1}}
 \Big\}
 \leq e^{\big(\frac{2}{\alpha}\big)^{\frac{1}{\mathtt{q}-1}}}\,.
 \end{equation}
 \\
 Consequently
\begin{equation}\label{albero1}
J_0 \le \gamma^{- 4\tN} 
\exp\big(- 144\tN^2 \inf_{ x\ge 1} \sum_{i:\jap i\le X^2(\al)} H_i(x)\big)\,.
\end{equation}
Let us compute $\inf_{x\ge 1}  H_i(x)$. We have
\[
H_i(x) \ge \hat H_i(x):= \al x \lambda(\sqrt{\jap{i}}) - \ln{\pa{1 + \sqrt{x}}} 
- \ln{\pa{1 + \sqrt{\jap{i}}}} \,.
\]
Then, since
\[
\hat H'_i(x) = \al  \lambda(\sqrt{\jap{i}}) - \frac{1}{2\sqrt{x}(1 + \sqrt{x})}=0 
\qquad \Leftrightarrow \qquad 
 \sqrt{x} =  \frac12 \left(-1 + \sqrt{1+ \frac{2}{ \al \lambda(\sqrt{\jap{i}}) }}\,\right)\,,
\]
we deduce
\begin{align*}
H_i(x) &\ge \lambda(\sqrt{i}) \left(-\frac{\al}{2} 
+ \frac{\al}{2}\sqrt{1 + \frac{2}{\al(\sqrt{\jap{i}})}}\,\right)^2 
- \ln{\left(1 +\sqrt{1+ \frac{2}{ \al \lambda(\sqrt{\jap{i}}) }}\,\right) }
- \ln{\pa{1 + \sqrt{\jap{i}}}} 
\\&\ge    
- \ln{\left(1 +\sqrt{1+ \frac{2}{ \al \lambda(1) }}\,\right)} 
- \ln{\pa{1 + X(\al)}}\,.
\end{align*}
The latter bound, together with \eqref{albero1}, implies 
\[
\begin{aligned}
J_0 &\le \gamma^{- 4\tN}
\pa{\left(1 +\sqrt{1+ \frac{2}{ \al \lambda(1) }}\,\right) \pa{1 + X(\al)}}^{ 144 \tN^2 X^2(\al)}
\\
&
\leq  \gamma^{- 4\tN}
\pa{\left(1 +\sqrt{1+ \frac{2}{ \al \ln^{\mathtt{q}}2}}\,\right)\pa{1 + X(\al)}}^{ 144 \tN^2 X^2(\al)}\,.
\end{aligned}
\]
By  \eqref{albero1tris} and \eqref{albero1bis} it follows the desired bound \eqref{controlJ0}
choosing a suitable  constant $\mathtt{C}>0$ large enough.

\medskip
\noindent
{\bf case $(\SO)$.}
We proceed as in the case $(\SE)$. 
By definition of the coefficients in the Sobolev case \eqref{coeffSobo} we have
\begin{equation}
J_0 := 
\sup_{\substack{j\in\Z,\,  (\al,\bt)\in\Lambda 
\\ \al_j+\bt_j\neq 0 
\\ |\al-\bt|\leq \tN+2}}
\frac{c^{(j)}_{\ri,\mathtt{w}(p + \delta)}(\al,\bt) }
{c^{(j)}_{\ri,\mathtt{w}(p)}(\al,\bt) |\omega\cdot (\al-\bt)|}\, 
= 
\sup_{\substack{j\in\Z,\,  (\al,\bt)\in\Lambda 
\\ \al_j+\bt_j\neq 0 
\\ |\al - \bt| \le \tN + 2}} 
\frac{ \lfloor j\rfloor^{2\delta}}{\abs{\omega\cdot{\pa{\al - \bt}}}}
\prod_{i\in\Z}\lfloor i\rfloor^{-\delta(\alpha_i+\beta_i)}.
\end{equation}
By the  diophantine condition \eqref{diofSet} we have
\[
J_0 \leq  \g^{-4\tN} 
\sup_{\substack{j\in\Z,\,  (\al,\bt)\in\Lambda 
\\ \al_j+\bt_j\neq 0 
\\ |\al - \bt| \le \tN + 2}} 
\Big(\frac{\jjap{j}^2}{\prod_{i\in\Z}\jjap{i}^{\al_i + \bt_i}} \Big)^\delta 
\prod_{i\in\Z}\big((1+|\al_i-\bt_i|^2)\langle i\rangle^2\big)^{\tau}\,.
\]
By Lemma \ref{lem:constance2SE} we only have 
to consider the case in which \eqref{divisor} holds. Recalling that 
$|\al| + |\bt| = \tN + 2$, $\td \le 4\tN$, $\zeta \ge (36\tN)^2$,
we can apply  Lemma \ref{stimaSob:lem}.  The bound \eqref{seisettedelta}
implies the estimate \eqref{controlJ0caseSob}.
\end{proof}

\section{A Birkhoff normal form step}

\textbf{Notations.} Let $r>r'>0$,  $\sigma, s>0$, $\zeta,p>0$, 
$\tK\gg1$ and $1\leq \tN\leq \tK-1$.
In the following we shall write 
\begin{equation}\label{parampara}
\tw=\tw(s,p)\quad ({\rm resp.} \;\; \tw=\tw(p))\qquad 
{\rm and} \qquad 
\tw'=\tw(s+\sigma,p)\quad  ({\rm resp.} \;\; \tw'=\tw'(p+\zeta))\,.
\end{equation}
We consider an Hamiltonian function of the form 
\begin{equation}\label{Hstep}
H = D_\omega + \sum_{\td=1}^{\tN-1}Z^{(\td)}
+  \sum_{\td=\tN}^{\tK}R^{(\td)}+ R^{(\geq \tK+1)}\,,
\end{equation}
where $D_\omega$ is in \eqref{quadraticBeam}
and
\[
\begin{aligned}
&Z^{(\td)}\in \cK^\wc_{\ri}(\th_{\twi})\cap\mathcal{H}^{(\td)}\,,
\qquad 1\leq \td\leq \tN-1\,,
\\&
R^{(\td)}\in \mathcal{H}^\wc_{\ri}(\th_{\twi})\cap\mathcal{H}^{( \td)}\,,
\qquad\tN\leq \td\leq \tK\,,
\\&
R^{(\geq\tK+1)}\in \mathcal{H}^\wc_{\ri}(\th_{\twi})\cap\mathcal{H}^{(\geq\tK+1)}\,.
\end{aligned}
\]
In the case $\tN=1$ we assume $Z^{(\td)}\equiv0$.
We set
\begin{equation}\label{def:epsiepsi}
\epsilon_{\td}:=|R^{(\td)}|_{r,\tw}\,,\;\;\tN\leq \td\leq \tK\,,
\qquad
\epsilon_{\tK+1}:=|R^{(\geq \tK+1)}|_{r,\tw}\,.
\end{equation}
\begin{lemma}{\bf (Birkhoff normal form step).}\label{dolcenera}
Consider the Hamiltonian $H$ in \eqref{Hstep}
and fix $\omega\in \mathtt{D}_{\gamma}$.
	Assume 
that
\begin{equation}\label{viadelcampo}
 J_0^{\star} \Big(\sum_{\td=\tN}^{\tK}\epsilon_{\td}+\epsilon_{\tK+1}\Big)
\leq\delta 
\qquad
\text{with}
\quad
\delta:=\frac{\ri-\rf}{16e\ri}
\,,
\end{equation}
where $J_0^{\star}=J_0^{\SE}(\s,\tN)$  in \eqref{controlJ0}
(respectively  $J_0^{\star}=J_0^{\SO}(\zeta,\tN)$  in \eqref{controlJ0caseSob}).

\noindent
Then there exists a change of variables
\begin{eqnarray}
	\label{pollon3}
	& \Phi\ :\ B_\rf(\th_{\twf})\ \to\ 
B_{\ri}(\th_{\twf})\,,
	\end{eqnarray}
such that
\begin{equation}\label{HstepPIUUNO}
H\circ\Phi= D_\omega + \sum_{\td=1}^{\tN}Z_{+}^{(\td)}+  \sum_{\td=\tN+1}^{\tK}R_{+}^{(\td)}+ R_{+}^{(\geq \tK+1)}\,,
\end{equation}
where 
\[
\begin{aligned}
&Z_{+}^{(\td)}\in \cK^\wc_{\ri}(\th_{\twi})\cap\mathcal{H}^{(\td)}\,,
\qquad 1\leq \td\leq\tN\,,
\\&
R_{+}^{(\td)}\in \mathcal{H}^\wc_{\ri}(\th_{\twi})\cap\mathcal{H}^{( \td)}\,,
\quad \tN+1\leq \td\leq \tK\,,
\\&
R_{+}^{(\geq\tK+1)}\in \mathcal{H}^\wc_{\ri}(\th_{\twi})\cap\mathcal{H}^{(\geq\tK+1)}\,.
\end{aligned}
\]
	Moreover the following estimates hold
\begin{align}
&Z_{+}^{(\td)}:=Z^{(\td)}\,,\quad 1\leq \td\leq\tN-1\,,\qquad 
|Z_{+}^{(\tN)}|_{\rf,\twf}
\leq
\epsilon_{\tN}
\,,
\label{signorinabis}
\\
|R_{+}^{(p)}|_{\rf,\twf}
&\leq 
\epsilon_{p}+
\sum_{\substack{j\geq2  \\ (j-1)\tN+\tN=p}}
\frac{\epsilon_{\tN}}{j!} \left(\frac{\epsilon_{\tN}J_0^{\star}}{2\delta}\right)^{j-1}
+
\sum_{\substack{1\leq j,\td\leq \tK \\ j\tN+\td=p}}
\frac{1}{j!}\left(\frac{\epsilon_{\tN}J_0^{\star}}{2\delta}\right)^{j} |Z^{(\td)}|_{r,\tw}	
\label{patata}
\\&	
\qquad 
+\sum_{\substack{1\leq j\leq \tK \\ \tN\leq\td\leq \tK \\j\tN+\td=p}}
\frac{\epsilon_{\td}}{j!} \left(\frac{\epsilon_{\tN}J_0^{\star}}{2\delta}\right)^{j}\,,
 \qquad \tN+1\leq \td\leq \tK\,,\nonumber
 \\
 |R_{+}^{(\geq\tK+1)}|_{\rf,\twf}
&\leq 
\epsilon_{\tK+1}+2\sum_{\td=\tN}^{\tK} 
\left(\frac{\epsilon_{\tN}J_0^{\star}}{2\delta}\right)^{\left[\frac{\tK+1-\td}{\tN}\right]}\epsilon_{\td}
+2\left(\frac{\epsilon_{\tN}J_0^{\star}}{2\delta}\right)^{\tK+1}(|Z|_{r,\tw}+\epsilon_{\tN})\,.
\label{signorina}
\end{align}
Finally, for any $\s^{\sharp}\geq 0$, $\zeta^{\sharp}\geq 0$, 
setting $\tw^{\sharp}:=\tw(s+\s+\s^{\sharp},p)$ 
(resp. $\tw^{\sharp}:=\tw(p+\zeta+\zeta^{\sharp})$)
assume the further  conditions 
\begin{equation}\label{viadelcampo'}
 \widetilde{J_0}^{\star}
 \Big(\sum_{\td=\tN}^{\tK}\epsilon_{\td}+\epsilon_{\tK+1}\Big)
\leq\delta\,,
\end{equation}
where $\widetilde{J_0}^{\star}=J_0^{\SE}(\s^{\sharp},\tN)$  in \eqref{controlJ0}
(respectively  $\widetilde{J_0}^{\star}=J_0^{\SO}(\zeta^{\sharp},\tN)$  in \eqref{controlJ0caseSob}).
Then 
\begin{equation}\label{pollon4}
\begin{aligned}
& \Phi_{\big|B_\rf(\th_{\tw^\sharp})}\ :\ B_\rf(\th_{\tw^\sharp})\ \to\ 
B_{\ri}(\th_{\tw^\sharp})\,,
\\
\sup_{u\in  B_\rf(\th_{\tw^\sharp})} &\norm{\Phi(u)-u}_{\th_{\tw^\sharp}}
\le
\ri \widetilde{J_0}^{\star} |R^{(\tN)}|_{\ri,\twi}\,.
\end{aligned}
\end{equation}
\end{lemma}

\begin{proof}
Recalling \eqref{Hstep}
we define
\begin{equation}\label{omoeq2}
\begin{aligned}
&Z_+:=\sum_{\td=1}^{\tN}Z_{+}^{(\td)}\,,\quad Z_{+}^{(\td)}:=Z^{(\td)}\,,\quad
1\leq\td\leq\tN-1\,,
\qquad
Z_{+}^{(\tN)}:=\Pi_{\mathcal{K}}R^{(\tN)}\,.
\end{aligned}
\end{equation}
By \eqref{omoeq2}, \eqref{def:epsiepsi}, 
\eqref{fame} and \eqref{bound proiezione} we deduce that 
$Z_{+}\in \cK^\wc_{\ri}(\th_{\twi})\cap\mathcal{H}^{(\leq \tN)}$
and satisfies the bound \eqref{signorinabis}.
Let
\begin{equation}\label{omoeq1}
S:=L_{\omega}^{-1}(\Pi_{\mathcal{R}}R^{(\tN)})
\end{equation}
be the unique solution of the homological equation 
$L_{\omega}S=\{S,D_{\omega}\}=\Pi_{\mathcal{R}}R^{(\tN)}$.
By Proposition \ref{shulalemma} we have that 
$S\in \mathcal{R}_{r}(\th_{\tw'})\cap \mathcal{H}^{(\tN)}$
and satisfies the estimate
\begin{equation}\label{cavolfiore}
|S|_{r,\mathtt{w}'}\leq J_0^{\star} |R^{(\tN)}|_{r,\mathtt{w}}
\stackrel{\eqref{def:epsiepsi}}{\leq}J_0^{\star}\epsilon_{\tN}\,,
\end{equation}
where $J_0^{\star}=J_0^{\SE}(\s,\tN)$  in \eqref{controlJ0}
(respectively  $J_0^{\star}=J_0^{\SO}(\zeta,\tN)$  in \eqref{controlJ0caseSob}).
We now apply Lemma \ref{ham flow} with
$(r,\tw)\rightsquigarrow(\rf,\twf)$ and $\rho:=r-\rf.$
Note that \eqref{viadelcampo} and \eqref{cavolfiore}
imply \eqref{stima generatrice}.
Setting  $\Phi:=\Phi_S^1$
we have that the conjugated Hamiltonian reads 
\begin{align*}
H\circ\Phi
&= D_\omega + \sum_{\td=1}^{\tN-1}Z^{(\td)} 
+  \Pi_{\mathcal{K}}R^{(\tN)}+\{D_{\omega},S\}+\Pi_{\mathcal{R}}R^{(\tN)}
+(e^{L_{S}}-\id -\{\cdot,S\}) D_\omega  
\\&+ (e^{L_{S}}-\id)(\sum_{\td=1}^{\tN-1}Z^{(\td)} +R^{(\tN)})
+e^{L_{S}}\big(\sum_{\td=\tN+1}^{\tK}R^{(\td)}+R^{(\geq \tK+1)}\big)
\\&
\stackrel{{\eqref{omoeq1},\eqref{omoeq2}}}{=}
 D_\omega + Z_{+}+R_{+}\,,
\end{align*}
where
\[
\begin{aligned}
R_{+}&:=  - \sum_{j=2}^\infty\frac{\pa{L_{S}}^{j-1}}{j!} \Pi_{\mathcal{R}}R^{(\tN)}
\\&
+ (e^{L_{S}}-\id)(\sum_{\td=1}^{\tN-1}Z^{(\td)}
+\sum_{\td=\tN}^{\tK}R^{(\td)}+R^{(\geq \tK+1)})
+\sum_{\td=\tN+1}^{\tK}R^{(\td)}+R^{(\geq \tK+1)}\,.
\end{aligned}
\]
Therefore we have
\begin{equation}\label{Rplus}
\begin{aligned}
R_{+}&:=\sum_{p=\tN+1}^{\tK}R_{+}^{(p)}+R_{+}^{(\geq\tK+1)}\,,
\\
R_{+}^{(p)}&:=R^{(p)}+\sum_{\substack{j\geq2  \\ j\tN+\tN=p}}
\frac{\pa{L_{S}}^{j-1}}{j!} \Pi_{\mathcal{R}}R^{(\tN)}
+
\sum_{\substack{1\leq j,\td\leq \tK \\j\tN+\td=p}}
\frac{\pa{L_{S}}^{j}}{j!} Z^{(\td)}
+
\sum_{\substack{1\leq j\leq \tK \\\tN\leq\td\leq \tK\\j\tN+\td=p}}
\frac{\pa{L_{S}}^{j}}{j!} R^{(\td)}
\end{aligned}
\end{equation}
and $R_{+}^{(\geq\tK+1)}$ defined by difference.
Moreover by, the explicit formul\ae\,\eqref{Rplus}, Lemma \ref{ham flow}, bounds
\eqref{cavolfiore}, \eqref{bound proiezione}, 
the smallness assumption
\eqref{viadelcampo}, Remark \ref{rmk:scalaPoi} 
and the monotonicity property (see Proposition \ref{crescenza}) 
we get $R_{+}\in \mathcal{R}^\wc_{\ri'}(\th_{\twi'})\cap\mathcal{H}^{(> \tN)}$
which satisfies \eqref{patata}-\eqref{signorina}.\\
Finally, let us assume \eqref{viadelcampo'}.
 By Proposition \ref{shulalemma}
 let $S^\sharp= L_\omega ^{-1} \Pi_{\mathcal{R}}R^{(\tN)}$ in 
 $\cR_{r}(\th_{\tw^\sharp})$ be the 
  solution of the homological equation 
 $L_\omega S^\sharp = \Pi_{\mathcal{R}}R^{(\tN)}$ on 
  $B_{r}(\th_{\tw^\sharp})\subseteq B_{r}(\th_{\twf})$ for any $\tw^\sharp \geq \tw'$.
 Since $S$ and $S^\sharp$ solve the same linear
 equation on 
 $B_{r}(\th_{\tw^\sharp})$, we have that
 $$
 S^\sharp=S_{\big| B_{r}(\th_{\tw^\sharp})}\,.
 $$ 
By Proposition \ref{shulalemma} we get
\begin{equation}\label{cavolfiore'}
\abs{S}_{r,\tw^\sharp}
\leq 
\widetilde{J_0}^{\star} | R^{(\tN)}|_{r,\twi}\,.	
\end{equation} 
We now apply Lemma \ref{ham flow} with
$(r,\tw)\rightsquigarrow(r,\tw^\sharp)$ 
and $\rho:=r-\rf.$
Note that \eqref{viadelcampo'} and \eqref{cavolfiore'}
imply \eqref{stima generatrice}.
Then \eqref{pollon4}
follows by \eqref{pollon} and \eqref{cavolfiore'}.
\end{proof}

\section{The iterative scheme}
Here we apply repeatedly Lemma \ref{dolcenera}. 
Let $\bar{r},s_0,p>0$, $0<\gamma<1$, fix a natural number $\tK\ge 1$
and
define $\tw_0:=\tw(s_0,p)$ (resp. $\tw_0:=\tw(p)$).
Consider a Hamiltonian $H$ such that $H - D_\omega\in\cH_{\bar{r}}(\th_{\tw_{0}})$ 
satisfying
\begin{equation}\label{hamIniz}
\begin{aligned}
&H:=D_{\omega}+R_0\,,\qquad R_0=\sum_{\td=1}^{\tK}R_0^{(\td)} + R_0^{(\geq\tK+1)}\,,
\\&
R_0^{(\td)}\in \mathcal{H}^\wc_{\bar{r}}(\th_{\tw_{0}})\cap\mathcal{H}^{(\td)}\,,
\qquad
R_0^{(\geq\tK+1)}\in \mathcal{H}^\wc_{\bar{r}}(\th_{\tw_{0}})\cap\mathcal{H}^{(\geq\tK+1)}\,,
\end{aligned}
\end{equation}
where $D_{\omega} $ is in \eqref{quadraticBeam}, 
$\omega\in D_{\gamma}$ (see \eqref{diofSet}). 
Consider the constant $\tC>0$ provided by Proposition \ref{shulalemma}
and define 
\begin{equation}\label{conditionRstar}
\begin{aligned}
r_0^{\star} &:= \min \set{
\bar r, (\frac{4^{\tK+3}|R_0|_{\bar{r},\tw_{0}}}{\bar{r}}  J_{\tK}^{\star}
32e\tK)^{-1}}\,,
\qquad \star\in\{\SE, \SO\}\,,
\\&
J_{\tK}^{\SE}:= \gamma^{-4\tK}
\exp\big(e^{\frac{\tK^2\tC}{s_0}} \big)\,,
\qquad 
J_{\tK}^{\SO}:=\gamma^{-4\tK} \exp\big(\mathtt{C} 2^{12} \tK^3\big)\,.
\end{aligned}
\end{equation}

\begin{remark}\label{rescalingHaminiziale}
Without loss of generality we can always assume that $|R_0|_{\bar{r},\tw_{0}}\leq1$.
Indeed if $|R_0|_{\bar{r},\tw_{0}}>1$ one can choose
$\widetilde{r}<\bar{r}$ such that (recall Lemma \ref{gasteropode})
\[
|R_0|_{\widetilde{r},\tw_{0}}\leq \left(\frac{\widetilde{r}\,}{\bar{r}}\right)
|R_0|_{\bar{r},\tw_{0}}\leq 1\,.
\]
\end{remark}
The main result of this section is the following.
\begin{theorem}{\bf (Birkhoff normal form).}\label{mainthm:BNF}
Consider $H$ in \eqref{hamIniz} .
Then, for any $0<r_0\leq r^{\star}_0$
with $\star\in\{\SE, \SO\}$,
there exists a symplectic map
\begin{equation}\label{sonno1}
\begin{aligned}
&{\bf \Phi} \; : \; B_{\frac{r_0}{2}}(\th_{\tw_{f}})\ \to\
B_{r_{0}}(\th_{\tw_{f}})
\\&
\sup_{u\in  B_{\frac{r_0}{2}}(\th_{\tw_{f}})} 
\norm{{\Phi}(u)-u}_{\th_{\tw_{f}}}
\le
\tC_1^{\star} r_0^{2}\leq \frac{r_0}{8}\,,
\qquad \tC_1^{\star}:= \frac{|R_0|_{\bar{r},\tw_{0}}}{\bar{r}} 
J_{\tK}^{\star}\,,
\end{aligned}
\end{equation}
where $\tw_{f}:=\tw(\tfrac{3}{2}s_0,p)$ (resp. $\tw_{f}=\tw(p+\zeta)$ with $\zeta=36^2\sum_{i=1}^{\tK}i^2$)   
such that the following holds.
The Hamiltonian 
\begin{equation}\label{sonno2}
\begin{aligned}
&H_{f}:=H\circ{\bf \Phi}:=
D_{\omega}+\mathfrak{Z}+\mathfrak{R}\,,
\\& 
\mathfrak{Z}\in  \cK^\wc_{\frac{r_0}{2}}(\th_{\twi_f})\cap\mathcal{H}^{(\leq \tK)}\,,
\qquad
\mathfrak{R}\in  \cH^\wc_{\frac{r_0}{2}}(\th_{\twi_f})\cap\mathcal{H}^{(\geq\tK+1)}
\end{aligned}
\end{equation}
satisfies
\begin{equation}\label{sonno3}
|\mathfrak{Z}|_{\frac{r_0}{2},\twi_f}\leq \mathtt{C}_2^{\star}r_0^2\,,
\qquad\;
|\mathfrak{R}|_{\frac{r_0}{2},\twi_f}\leq\mathtt{C}_3^{\star} r_0^{\tK+1}
\end{equation}
with
\[
\mathtt{C}_2^{\star}:=\frac{16 e\tK  \abs{R_0}^{\wc}_{\bar{r},\tw}  
4^{\tK+1}}{\bar{r}^2}J_{\tK}^{\star}\,,
\qquad \;\;\;
\mathtt{C}_3^{\star}:=
\frac{\abs{R_0}^{\wc}_{\bar{r},\tw} (16 e\tK   4^{\tK+2})^\tK}{\bar{r}^{\tK+1}}
(J_{\tK}^{\star})^{\tK}\,.
\]
\end{theorem}

\noindent
The proof of Theorem \ref{mainthm:BNF} is based on the following 
 iterative scheme.

\vspace{0.3em}
\noindent
\textbf{Setting of parameters.} For any $0\le k \le \tK$, let us recursively define: 
\begin{equation}\label{parametri}
\begin{aligned}
&r_k= r_0(1- \frac{k}{2\tK})\,,
\qquad  
\delta_k = \frac{r_{k} - r_{k+1}}{16 e r_{k}}\,,
\qquad 
s_k = s_0(1 + \frac{k}{2\tK})\,,
\quad \zeta_k:= (36 k)^2
\\
 &\s_k:=s_0 \frac{k}{2\tK}\,,
 \qquad 
 \tw_k:=\tw(s_k,p)\,,
 \qquad 
 {({\rm resp.} \; \tw_k:=\tw(p+\sum_{i=1}^{k}\zeta_i))}\,.
\end{aligned}
\end{equation}

Moreover, let us define 
\begin{equation}\label{patata2}
\eps:= \pa{\frac{r_0}{\bar{r}}}\,,
\qquad
\mathtt{R}_0:=|R_0|_{\bar{r},\tw_0}\,.
\end{equation}
By Lemma \ref{gasteropode}
and Remark \ref{rescalingHaminiziale} we have that
\begin{equation}\label{scala0}
\abs{R_0^{(\td)}}^{\wc}_{r_0,\tw_{0}}\leq \e^{\td} \mathtt{R}_0\,,
\qquad  1\leq \td\leq \tK\,,
\qquad
\abs{R_0^{(\geq\tK+1)}}^{\wc}_{r_0,\tw_{0}}\leq  \e^{\tK+1} \mathtt{R}_0\,.
\end{equation}
Let us introduce 
\begin{equation}\label{arancia1}
 J_k:=J_0^{\SE}(\sigma_k, k)\,, \quad ({\rm resp.} \;\;  
 J_k:=J_0^{\SO}(\sum_{i=1}^{k}\zeta_i, k))\qquad 0\le k \le \tK\,,
\end{equation}
where $J_0^{\SE}$, $J_0^{\SO}$ are introduced in \eqref{controlJ0}, \eqref{controlJ0caseSob}
respectively, 
and assume the following smallness  condition:
\begin{equation}\label{cond:small}
\mathtt{R}_0{4^{\tK+3}} J_{\tK} \,\eps \leq \delta_0\,.
\end{equation}
We now prove the following.
\begin{lemma}{\bf (Iteration lemma).}\label{lem:iterazione}
The following holds true for any $0\leq k\leq \tK$:

\vspace{0.3em}
\noindent
${\bf (S1)}_k$ there are Hamiltonians $H_{k}$ of the form
\begin{equation}\label{hamk}
\begin{aligned}
&H_k= D_\omega +Z_{k} +R_{k}\,,
\\&
Z_{k}:= \sum_{\substack{1\leq\td\leq k\\ \td\; {\rm even} }}Z_{k}^{(\td)}\,,\qquad
 R_{k}:= \sum_{\td= k+1 }^{\tK}R_{k}^{(\td)}+
R_{k}^{(\geq\tK+1)}\,, 
\end{aligned}
\end{equation}
where $Z_0\equiv 0$, and 
\begin{equation}\label{hamk2}
\begin{aligned}
 Z^{(\td)}_{k}&\in \cK^\wc_{\ri_k}(\th_{\twi_k})\cap\mathcal{H}^{(\td)}\,,
\qquad 1\leq \td\leq k
\\
R_k^{(\td)}&\in \mathcal{H}^\wc_{\ri_k}(\th_{\twi_k})\cap\mathcal{H}^{( \td)}\,,
\qquad k+1\leq \td\leq \tK\,,
\\
R_k^{(\geq\tK+1)}&\in \mathcal{H}^\wc_{\ri_k}(\th_{\twi_k})\cap\mathcal{H}^{( \geq\tK+1)}\,;
\end{aligned}
\end{equation}

\noindent
${\bf (S2)}_k$ one has, for $1\leq k\leq \tK$,
\begin{equation}\label{smallcondk}
J_k
\left(
\sum_{\td=k}^{\tK}|R_{k-1}^{(\td)}|_{r_{k-1},\tw_{k-1}}+
|R_{k-1}^{(\geq\tK+1)}|_{r_{k-1},\tw_{k-1}}
\right)
\le \delta_{k-1}\,;
\end{equation}

\noindent
${\bf (S3)}_k$ one has, for $1\leq k\leq \tK$,
\begin{align}
\norm{Z_k^{(\td)}}_{r_k,\tw_k} &\leq \e^{\td}
\mathtt{R}_0(4^k J_{\tK}\delta_0^{-1})^{\td-1}
2^{\td-1}\,,
\qquad 1\leq \td\leq k\,,
\label{smallnormk1}
\\
\norm{R_{k}^{(\td)}}_{r_k,\tw_k} &\le \e^{\td}
\mathtt{R}_0({4^k} J_{\tK}\delta_0^{-1})^{\td-1}
{2^{k-1}}\,,
\quad
k+1\leq \td\leq \tK\,,
\label{smallnormk2}
\\
\norm{R_{k}^{(\geq\tK+1)}}_{r_k,\tw_k} &\le \e^{\tK+1}\mathtt{R}_0
(4^k J_{\tK}\delta_0^{-1})^{\tK}
2^{k}\,;
\label{smallnormk3}
\end{align}

\noindent
${\bf (S4)}_k$  for any $1\leq k\leq \tK$ there are maps
$\Phi_k \; : \; B_{r_k}(\th_{\tw_k})\ \to\ 
B_{r_{k-1}}(\th_{\tw_k})$
such that
\begin{equation}\label{hamkkk}
H_k=H_{k-1}\circ \Phi_k\,.
\end{equation}
Moreover, for any $k\leq n\leq \tK$,  
 one has
\begin{equation}\label{sepultura2}
\Phi_k \; : \; B_{r_k}(\th_{\tw_n})\ \to\ 
B_{r_{k-1}}(\th_{\tw_n})
\end{equation}
with 
\begin{equation}\label{stimaPhiK}
\sup_{u\in  B_{r_k}(\th_{\tw_{n}})} \norm{\Phi_k(u)-u}_{{\tw_{n}}}
\le 
r_{k-1}\mathtt{R}_0\frac{1}{2^k} J_{\mathtt{K}}\e\,.
\end{equation}
\end{lemma}
\begin{proof}
We reason inductively and
 apply iteratively Lemma \ref{dolcenera}.
Assume that  ${\bf (Si)}_k$, ${\bf i}=1,\ldots,4$,
hold for $0\leq k\leq \tK-2$.
We claim that
\begin{equation}\label{arancia2}
J_{k+1}\left(\sum_{\td=k+1}^{\tK}|R_{k}^{(\td)}|_{r_k,\tw_k}+
|R_{k}^{(\geq\tK+1)}|_{r_k,\tw_k}\right)\leq \delta_k\,,
\end{equation}
which is condition ${\bf (S2)}_{k+1}$
(see 
\eqref{smallcondk} with $k\rightsquigarrow k+1$).
First of all notice that (recall \eqref{controlJ0}, \eqref{parametri})
\begin{equation}\label{dolcenera3}
J_{k+1}\leq J_{\tK}\,,\qquad \forall\, 0\leq k\leq \tK-1\,.
\end{equation}
By the inductive assumption ${\bf (S3)}_k$-\eqref{smallnormk2}  (see \eqref{hamk})
one can note that
\begin{equation*}
\begin{aligned}
\sum_{\td=k+1}^{\tK}&|R_{k}^{(\td)}|_{r_k,\tw_k}+
|R_{k}^{(\geq\tK+1)}|_{r_k,\tw_k}
\\&
\leq \sum_{\td=k+1}^{\tK}\mathtt{R}_0
\e^{\td}(4^{k}J_{\tK}\delta_0^{-1})^{\td-1}2^{k-1}
+\mathtt{R}_0\e^{\tK+1}(4^{k}J_{\tK}\delta_0^{-1})^{\tK}2^{k}
\\&
\leq 2^{k-1}\mathtt{R}_0
\e^{k+1}(4^{k}J_{\tK}\delta_0^{-1})^{k}
\left[
\sum_{j=0}^{\tK-(k+1)}
\e^{j}(4^{k}J_{\tK}\delta_0^{-1})^{j}+
\e^{\tK-k}(4^{k}J_{\tK}\delta_0^{-1})^{\tK-k}2
\right]
\\&\stackrel{\eqref{cond:small}}{\leq}
\mathtt{R}_0 \e^{k+1}(4^{k}J_{\tK}\delta_0^{-1})^{k} 2^{k}\,.
\end{aligned}
\end{equation*}
Recalling \eqref{arancia1}, \eqref{controlJ0},
we deduce that
\[
\begin{aligned}
J_{k+1}\Big(
\sum_{\td=k+1}^{\tK}|R_{k}^{(\td)}|_{r_k,\tw_k}
&+
|R_{k}^{(\geq\tK+1)}|_{r_k,\tw_k}
\Big)
\leq\mathtt{R}_0\e^{k+1}(4^{k}J_{\tK}\delta_0^{-1})^{k} 2^{k}J_{k+1}
 \\&
 \leq\mathtt{R}_0
 ( \e 4^{k}J_{\tK}\delta_0^{-1})^{k+1} 2^{k} 4^{-k}\delta_0
  \stackrel{\eqref{cond:small}}{\leq}
  \delta_0\stackrel{\eqref{parametri}}{\leq}\delta_k\,,
\end{aligned}
\]
which proves the claim. Condition \eqref{arancia2} 
implies the smallness assumption \eqref{viadelcampo}
by setting
\[
\epsilon_{\td}:=|R_{k}^{(\td)}|_{r_k,\tw_k}\,,
\qquad
\epsilon_{\tK+1}:=|R_{k}^{(\geq\tK+1)}|_{r_k,\tw_k}\,.
\]
Therefore Lemma \ref{dolcenera} applies
with 
\begin{equation}\label{dolcenera2}
\tN,\s, r, r', \tw,\tw'\rightsquigarrow k+1, \s_k, r_k, r_{k+1}, \tw_k,\tw_{k+1}\,, 
\quad
Z^{(\td)},R^{(\td)},R^{(\geq\tK+1)} \rightsquigarrow Z^{(\td)}_{k}, R_{k}^{(\td)}, R_{k}^{(\geq\tK+1)}\,.
\end{equation}
Then we obtain a map $\Phi_{k+1}$ satisfying 
\eqref{hamkkk}
with $k\rightsquigarrow k+1$ (see \eqref{pollon3}-\eqref{HstepPIUUNO}).
In particular the new hamiltonian $H_{k+1}$ has the form
\eqref{hamk} (with $k\rightsquigarrow k+1$)
for some Hamiltonians $Z_{k+1}, R_{k+1}$.
This proves ${\bf (S1)}_{k+1}$.
Let us check the ${\bf (S3)}_{k+1}$, i.e. 
the bounds \eqref{smallnormk1}-\eqref{smallnormk3}
with $k\rightsquigarrow k+1$.
By \eqref{signorinabis} we have
\begin{equation}\label{def:passo}
Z_{k+1}=\sum_{\td=1}^{k+1}Z^{(\td)}_{k+1}\,,
\qquad
Z^{(\td)}_{k+1}:=Z_{k}^{(\td)}\,,\quad1\leq\td\leq k\,,
\quad
Z^{(k+1)}_{k+1}:=\Pi_{\mathcal{K}}R^{(k+1)}_{k}\,.
\end{equation}
For $1\leq \td\leq k$ one has
\[
|Z_{k+1}^{(\td)}|_{r_{k+1},\tw_{k+1}}\stackrel{\eqref{def:passo}}{=}
|Z_{k}^{(\td)}|_{r_{k+1},\tw_{k+1}}\stackrel{\eqref{smallnormk1}_{k}}{\leq}
\mathtt{R}_0
\e^{\td}(4^k J_{\tK}\delta_0^{-1})^{\td-1}
2^{\td-1}
\leq \mathtt{R}_0\e^{\td}(4^{k+1} J_{\tK}\delta_0^{-1})^{\td-1}
2^{\td-1}
\]
where the last inequality is trivial.
Moreover for $\td=k+1$ one has
\[
|Z_{k+1}^{(k+1)}|_{r_{k+1},\tw_{k+1}}\stackrel{\eqref{def:passo},\eqref{fame}}{\leq}
|R_{k}^{(k+1)}|_{r_{k+1},\tw_{k+1}}
\stackrel{\eqref{smallnormk2}_{k}}{\leq}\mathtt{R}_0
\e^{k+1}(4^k J_{\tK}\delta_0^{-1})^{k}
2^{k-1}\leq\mathtt{R}_0
\e^{k+1}(4^{k+1} J_{\tK}\delta_0^{-1})^{k}
2^{k}\,.
\]
This proves the bound \eqref{smallnormk1} with $k\rightsquigarrow k+1$ on the 
Hamiltonians $Z_{k+1}^{(\td)}$ with $1\leq \td\leq k+1$.
We notice that $Z_{k+1}\equiv0$ when $k+1$ is odd (see Remark \ref{oddKerham}).
We now prove the estimate \eqref{smallnormk2} with $k\rightsquigarrow k+1$
on $R_{k+1}$.
By \eqref{patata}, \eqref{def:epsiepsi}, \eqref{dolcenera2}, 
\eqref{arancia1}, \eqref{dolcenera3}, \eqref{parametri},
and using \eqref{smallnormk1}, \eqref{smallnormk2} and the fact that $\mathtt{R}_0\leq1$,
we get, for any $k+2\leq p\leq \tK$, 
\begin{equation*}
\begin{aligned}
|R_{k+1}^{(p)}|_{r_{k+1},\tw_{k+1}}
&\leq \mathtt{R}_0\e^{p}(4^k J_{\tK}\delta_0^{-1})^{p-1}2^{k-1}
\\&
+\mathtt{R}_0\sum_{\substack{j\geq2  \\(j-1)(k+1)+(k+1)=p}}
\frac{\e^{k+1}(4^k J_{\tK}\delta_0^{-1})^{k}2^{k-1}}{j!} 
\left(\frac{\e^{k+1}(4^k J_{\tK}\delta_0^{-1})^{k}2^{k-1}J_{\tK}}{2\delta_0}\right)^{j-1}
\\&
+\mathtt{R}_0
\sum_{\substack{1\leq j,\td\leq \tK \\j\tN+\td=p}}
\frac{\e^{\td}(4^k J_{\tK}\delta_0^{-1})^{\td-1}2^{k-1}}{j!}
\left(\frac{\e^{k+1}(4^k J_{\tK}\delta_0^{-1})^{k}2^{k-1}J_{\tK}}{2\delta_0}\right)^{j}
	\label{signorinatris}
\\&	
	+\mathtt{R}_0\sum_{\substack{1\leq j\leq \tK \\
	\tN\leq\td\leq \tK
	\\
	j\tN+\td=p}}
	\frac{\e^{\td}(2^k J_{\tK}\delta_0^{-1})^{\td-1}2^{k-1}}{j!}
\left(\frac{\e^{k+1}(2^k J_{\tK}\delta_0^{-1})^{k}2^{k-1}J_{\tK}}{2\delta_0}\right)^{j}
\\&\leq\mathtt{R}_0
\e^{p}(4^k J_{\tK}\delta_0^{-1})^{p-1}2^{k-1}
\\&+\mathtt{R}_0
\e^{p}(4^k J_{\tK}\delta_0^{-1})^{p-1}2^{k-1}
\sum_{j\geq2}\frac{1}{j!}\frac{2^{(k-1)(j-1)}}{2^{(j-1)} 4^{k(j-1)}}
\\&+\mathtt{R}_0
\e^{p}(4^k J_{\tK}\delta_0^{-1})^{p-1}2^{k-1} 2\cdot\sum_{j\geq1}
\frac{1}{j!}\frac{2^{(k-1)j}}{2^{j}4^{kj}}
\\&\leq\mathtt{R}_0
\e^{p}(4^k J_{\tK}\delta_0^{-1})^{p-1}2^{k-1}\Big(1 + \sum_{j\geq2}\frac{1}{j!}
\Big(\frac{2^{k-1}}{2 \cdot 4^{k}}\Big)^{j-1} +
\sum_{j\geq1}\frac{1}{j!}\Big(\frac{2^k}{2\cdot 4^k}\Big)^{j}\Big)
\\&\leq\mathtt{R}_0
\e^{p}(4^k J_{\tK}\delta_0^{-1})^{p-1}2^{k-1}\Big(
\sum_{j\geq0}\frac{1}{j!}\Big(\frac{1}{2^k}\Big)^{j}
\Big)
\\&\leq\mathtt{R}_0
\e^{p}(4^k J_{\tK}\delta_0^{-1})^{p-1}2^{k-1} \sqrt{e}
\leq
\\&\leq\mathtt{R}_0
\e^{p}(4^k J_{\tK}\delta_0^{-1})^{p-1}2^{k}
\leq \mathtt{R}_0
\e^{p}(4^{k+1} J_{\tK}\delta_0^{-1})^{p-1}2^{k}\,,
\end{aligned}
\end{equation*}
which is the \eqref{smallnormk2} with $k\rightsquigarrow k+1$.
The estimate for the term $R_{k+1}^{(\geq\tK+1)}$ follows similarly.

It remains to show the ${\bf (S4)}_{k+1}$, i.e. the bound \eqref{stimaPhiK} with $k\rightsquigarrow k+1$.
By estimate \eqref{pollon4} (recalling again \eqref{dolcenera2}) and taking 
$k\leq n\leq \mathtt{K}$
we get
\[
\begin{aligned}
\sup_{u\in  B_{r_{k+1}}(\th_{\tw_{n}})} 	\norm{\Phi_{k+1}(u)-u}_{{\tw_{n}}}&\leq
r_k J_{k+1}
|R_{k}^{(k+1)}|_{r_k,\tw_k}
\\&
\stackrel{\eqref{smallnormk2}}{\leq} r_k J_{\tK}
\e^{k+1}
\mathtt{R}_0({4^k} J_{\tK}\delta_0^{-1})^{k}
{2^{k-1}}\,,
\end{aligned}
\]
which implies the thesis \eqref{stimaPhiK} using the smallness condition \eqref{cond:small}.
\end{proof}

\begin{proof}[{\bf Proof of Theorem \ref{mainthm:BNF}.} ]
The condition  \eqref{conditionRstar} and  the choice of $\e$ in \eqref{patata2} 
imply the smallness condition \eqref{cond:small}, so the iterative lemma \ref{lem:iterazione}
applies.
By 
\eqref{pollon4}  we have (recall \eqref{parametri})
\begin{equation*}
\Phi_k \; : \; B_{r_k}(\th_{\tw_\tK})\ \to\ 
B_{r_{k-1}}(\th_{\tw_\tK})
\end{equation*}
with 
\begin{equation*}
\sup_{u\in  B_{r_k}(\th_{\tw_{\tK}})} 	\norm{\Phi_k(u)-u}_{{\tw_{\tK}}}
\le 
r_{k-1}\mathtt{R}_0\frac{1}{2^k} J_{\tK}^{\star}\e\,.
\end{equation*}

Then,  defining 
\[
{\bf \Phi}:=\Phi_1\circ\cdots\circ\Phi_{\tK}\,,
\]
we immediately get 
\[
{\bf \Phi} \; : \; B_{\frac{r_0}{2}}(\th_{\tw_{\tK}})\ \to\
B_{r_{0}}(\th_{\tw_{\tK}})
\]
which, adding and subtracting $\id$ to each $\Phi_i$ with $i = 1,\ldots \tK$ entails
$$
{\bf \Phi}  - \id = (\Phi_1 - \id)\circ\Phi_2\circ\cdots\circ \Phi_\tK + (\Phi_2 - \id) \circ \Phi_3 \circ\cdots \circ\Phi_\tK + \cdots + \Phi_\tK - \id\,.
$$
Hence we get  the estimate
\begin{equation*}
\sup_{u\in  B_{\frac{r_0}{2}}(\th_{\tw_{\tK}})}  \norma{{\bf \Phi}(u) - u}_{\tw_\tK} 
\le 
\mathtt{R}_0 J_{\tK}^{\star}\e \sum_{j=0}^{\tK-1} 
\frac{r_j}{2^{j+1}}\le r_0\mathtt{R}_0 J_{\tK}^{\star}\e \,,
\end{equation*}
which
implies \eqref{sonno1}
by using the definitions of $J_{\tK}$, $\e$ and $\mathtt{R}_0$,
the smallness assumption \eqref{conditionRstar} on $r_0$
and setting $\tw_{f}=\tw_{\tK}$.

The new Hamiltonian $H\circ{\bf \Phi}$ is equal to $H_{\tK}$
given in \eqref{hamk} with $k\rightsquigarrow \tK$.
We then set $\mathfrak{Z}:=Z_{\tK}$
and $\mathfrak{R}:=R_{\tK}$.
The \eqref{sonno2} follows.
The bounds \eqref{sonno3} follow by \eqref{smallnormk1}-\eqref{smallnormk3}
recalling \eqref{arancia1}, \eqref{scala0}, \eqref{patata2}.
\end{proof}

 \section{Times of stability}
 In this section we provide the proof of our main stability results for the solutions 
 of the equation \eqref{eq:beam}.
 
 \noindent
 The proof basically rely on a combination of the Birkhoff normal form result 
 of Theorem \ref{mainthm:BNF} and Lemma \ref{tempotempo}.
In view of section 
\ref{sec:hamstructurebeam}
we have that equation \eqref{eq:beam}
is equivalent to \eqref{eq:beamComp}.
In particular, recalling \eqref{beam5},  for any 
$(\psi_0,\psi_1)\in H^{s,p+1}\times H^{s,p-1}$ we define
\begin{equation}\label{claimclaim1}
u_0:=\frac{1}{\sqrt{2}}\Big(\omega^{\frac{1}{2}}\psi_0+{\rm i} \omega^{-\frac{1}{2}}\psi_0\Big)\,.
\end{equation}
One can notice that condition \eqref{main:smallcond} (resp. \eqref{main:smallcondSob})
implies that the function $u_0$ in \eqref{claimclaim1} satisfies
\begin{equation}\label{smallnessu0}
\|u_0\|_{s,p}\leq \delta\,.
\end{equation}
Assume now that  $u(t)$ is  a solution of \eqref{eq:beamComp}, with initial condition
$u(0)=u_0$,  
satisfying 
\begin{equation}\label{claimclaim}
\sup_{t\in[0,T_0]}\|u(t)\|_{s,p}\leq 2\delta 
\qquad {\rm for\;some}\qquad T_0>0\,.
\end{equation}
Hence  the solution $(\psi(t), \partial_t\psi(t))$ of \eqref{eq:beam}
with initial conditions $(\psi_0,\psi_1)$  satisfies the
\emph{a priori} bound
\[
\|\psi(t)\|_{s,p+1}+\|\partial_t\psi(t)\|_{s,p-1}\leq 4\|u(t)\|_{s,p}\leq 8\delta\,,
\qquad \forall\, t\in[0,T_0]\,,
\]
which implies \eqref{boundsol1} (resp. \eqref{boundsol1Sob}).

The discussion above implies that in order to proof Theorems \ref{main:subexp}, \ref{main:sobol}
we just have to prove the claim \eqref{claimclaim}
on solutions of \eqref{eq:beamComp} with initial conditions satisfying \eqref{smallnessu0}
and we shall provide suitable lower bounds on the lifespan $T_0>0$.

In order to apply our abstract Birkhoff normal form result we need some preliminary results.
More precisely we shall prove that the 
the Hamiltonian function $H$
in  \eqref{beamHam} (see also \eqref{beamHamFourier}) 
can be written in the form \eqref{hamIniz} with
\begin{equation}\label{hamRR00}
R_0:=\mathtt{H}_{\geq3}:=
\int_{\mathbb{T}} F\Big( \frac{\omega^{-1/2}(u+\bar{u})}{\sqrt{2}}\Big)\ {dx}\,,
\end{equation}
where $F$ is the analytic function in \eqref{nonlineF}.
This is the content of the following lemma.
\begin{lemma}\label{lem:applicoSub}
Let $R>0$ as in \eqref{nonlineF} and  consider 
the Hamiltonian $R_0$ in \eqref{hamRR00}.

\vspace{0.3em}
\noindent
 {\bf \emph{Case} $(\SE)$}. 
Fix  $s,p>0$ as in Theorem \ref{main:subexp} and let
(recall \eqref{es:fgrowth})
\begin{equation}\label{abbondanza}
s_0:=\frac{2}{3}s\,,\quad \tw_0:=\tw(s_0,p)\,.
\end{equation}
For any $\bar{r}>0$
satisfying  (see \eqref{costanti algebra})
\begin{equation}\label{stimaNEM2}
2 \mathtt{C}_{\rm alg}(p) \bar{r} <R\,,
\end{equation}
one has that the Hamiltonian $R_0$ in \eqref{hamRR00}
belongs to the space $\cH_{\bar{r}}(\th_{{\mathtt w}_0})$ of regular Hamiltonians
(see Def. \ref{Hreta}),
and 
\begin{equation}\label{stimaNEM}
|R_0|_{\bar{r},\tw_0}\leq
\mathtt{C}(p, R) 
|F|_{R} \bar{r}<+\infty\,,\qquad \mathtt{C}(p, R) :=\frac{8C_{\rm alg}(p)}{R^3}\,.
\end{equation}

\vspace{0.3em}
\noindent
 {\bf  \emph{Case} $(\SO)$}.  
Fix  $s=0$ and let 
\begin{equation}\label{abbondanzaSob}
\tw_0:=\tw(p_0)\,,
\end{equation}
For any $p_0>1$ 
and for any $\bar{r}>0$
satisfying (recall \eqref{costanti algebra})
\[
\mathtt{C}_{{\rm alg}, \mathtt{M}}(p_0) \bar{r} <R\,,
\] 
one has that the Hamiltonian $R_0$ in \eqref{hamRR00}
belongs to the space $\cH_{\bar{r}}(\th_{{\mathtt w}_0})$ of regular Hamiltonians
and 
\begin{equation}\label{stimaNEMSob}
|R_0|_{\bar{r},\tw_0}\leq \frac{\mathtt{C}_{{\rm alg}, \mathtt{M}}(p_0)}{R} |F|_{R} \bar{r}<+\infty\,.
\end{equation}
\end{lemma}

\begin{proof}
Let us consider the case $(\SE)$.
It follows using the analyticity 
of the function $F$ in \eqref{nonlineF} and Lemma \ref{algebra}.
and reasoning as in Proposition $5.2$ in \cite{BMP:CMP}.

\noindent
By analyticity we have that 
\[
\begin{aligned}
R_0=\int_{\mathbb{T}}F\Big( \frac{\omega^{-1/2}(u+\bar{u})}{\sqrt{2}}\Big) dx
&= 
\sum_{d=3}^\infty F^{(d)}\int_{\mathbb{T}} (\frac{\omega^{-1/2}(u+\bar{u})}{\sqrt{2}})^d  dx
\\&= 
\sum_{d=3}^\infty \frac{F^{(d)}}{2^{d/2}} 
\sum_{k=0}^d 
\binom{d}{k} \int_{\mathbb{T}}(\omega^{-\frac{1}{2}}u)^k(\omega^{-\frac{1}{2}}\bar{u})^{d-k}dx \,,
\end{aligned}
\]
hence, passing to the Fourier basis, we get
\[
 R_0=\sum_{d=3}^\infty \frac{F^{(d)}}{2^{d/2}}
 \sum_{k=0}^d \binom{d}{k} 
 \sum_{\sum_{i=1}^{k}j_i=\sum_{p=k+1}^{d}j_{p}}
 \!\!\!\!\!
 C^{(d)}(j_1,\ldots,j_{d})
 u_{j_1} \cdots  u_{j_k}\bar{u}_{j_{k+1}}\cdots\bar{u}_{j_{d}}
\]
where the coefficients $C^{(d)}(j_1,\ldots,j_{d})$ are symmetric in $j_i$, $i=1,\ldots,d$, 
and have the form
\begin{equation}\label{cddcdd}
C^{(d)}(j_1,\ldots,j_{d}):=\prod_{i=1}^{d}\omega^{-\frac{1}{2}}(j_{i})
\stackrel{\eqref{omegoneBeam}}{=}\prod_{i=1}^{d}\frac{1}{\sqrt[4]{|j_i|^{4}+\mathtt{m}}}\leq1\,.
\end{equation}
In view of \eqref{eq:beamComp} we now compute the first component of the
vector field $X_{R_0}$. We have
\[
\begin{aligned}
X_{R_0}^{(j)}&:=-{\rm i}\partial_{\bar{u}_{j}} R_0
=-{\rm i}
\sum_{d=3}^\infty \frac{F^{(d)}}{2^{d/2}}
\sum_{k=0}^d \binom{d}{k} (d-k) \mathtt{a}^{(d,k)}(j)\,,
\\
\mathtt{a}^{(d,k)}(j)&:=
\sum_{\sum_{i=1}^{k}j_i-\sum_{p=k+1}^{d-1}j_{p}=j}
 \!\!\!\!\!
 C^{(d)}(j_1,\ldots,j)
 u_{j_1} \cdots  u_{j_k}\bar{u}_{j_{k+1}}\cdots\bar{u}_{j_{d-1}}\,.
\end{aligned}
\]
where we used the symmetry of the coefficients  $C^{(d)}(j_1,\ldots,j)$.
Notice 
that
\[
|X_{\underline{R_0}}(u)|_{{\tw(s,p)}}\leq 
|X_{\underline{R_0}}(\underline{u})|_{{\tw(s,p)}}\leq 
\sum_{d=0}^\infty \frac{|F^{(d)}|}{2^{d/2}}
 \sum_{k=0}^d \binom{d}{k} (d-k) |\mathtt{a}^{(d,k)}|_{{\tw(s,p)}}
\]
where $\underline{u}=(|u_{j}|)_{j\in\mathbb{Z}}$ and  for any $0\leq k\leq d$, $d\geq3$ we set
$\mathtt{a}^{(d,k)}:=(\mathtt{a}^{(d,k)}(j))_{j\in\mathbb{Z}}$.
Notice moreover that 
\[
|\mathtt{a}^{(d,k)}(j)|\leq |(\underbrace{\underline{u}\star \ldots \underline{u}}_{k }\star 
\underbrace{\bar{\underline{u}}\star\ldots \star\bar{\underline{u}}}_{d-k-1})_{j}|\,,\qquad
\forall\, j\in\mathbb{Z}\,.
\]
Therefore, using estimate \eqref{stimacalgcalg} in Lemma \ref{algebra}, we obtain
\[
\begin{aligned}
|X_{\underline{R_0}}(u)|_{{\tw(s,p)}}&\leq 
\sum_{d=3}^\infty \frac{|F^{(d)|}}{2^{d/2}}
 \sum_{k=0}^d \binom{d}{k} (d-k)
( \Calg)^{d-2}
 |u|_{\mathtt{w}(s,p)}^{d-1}
 \\&
 \leq 
 \sum_{d=3}^\infty d 2^{\frac{d}{2}-1} |F^{(d)}|
 ( \Calg)^{d-2}
 |u|_{\mathtt{w}(s,p)}^{d-1}
 \leq 
 4\sum_{d=3}^\infty  |F^{(d)}|
 ( 2 \Calg)^{d-2}
 |u|_{\mathtt{w}(s,p)}^{d-1}\,.
 \end{aligned}
\]
Then
\[
|R_0|_{\bar{r},\tw_0}\leq 
 \sum_{d=3}^\infty  |F^{(d)}|
 ( 2 \Calg \bar{r})^{d-2} \,,
\]
which implies the bound \eqref{stimaNEM}
using  \eqref{stimaNEM2}
and the hypothesis \eqref{nonlineF}.
This proves item $(i)$.

\noindent
The case $(\SO)$ in item $(ii)$ can be proved by  following almost word by word 
the proof of item $(i)$
using the estimate \eqref{stimacalgcalgMM} instead of \eqref{stimacalgcalg} in Lemma \ref{algebra}.
\end{proof}

\subsection{Sub-exponential stability and proof of Theorem \ref{main:subexp}}\label{sec:subexpstability}
For future convenience we set
\begin{equation}\label{sceltabarr}
\bar{r}:=\frac{R}{\mathtt{C}_{\rm alg}(p)\mathtt{C}(p,R)|F|_{R}}\,,
\end{equation}
where $\mathtt{C}(p,R)$ is given in \eqref{stimaNEM}. 
Therefore by case $\SE$ in Lemma \ref{lem:applicoSub} we have that the Hamiltonian 
$H$ in  \eqref{beamHam} can be written, for any $\tK\geq1$, in the form 
\eqref{hamIniz} 
with  $s_0, \tw_0$ as in \eqref{abbondanza}.
Our aim is to apply Theorem \ref{mainthm:BNF}
to the Hamiltonian $H=D_{\omega}+R_0$ with $R_0$ in \eqref{abbondanza}.
Recalling the parameters in \eqref{parametri}-\eqref{patata2}
we fix
\begin{equation}\label{cond:r0piccolo}
r_0= 2\delta\,,
\end{equation}
and
\begin{equation}\label{Kappone}
\tK = \tK(r_0) := \left[
\Big(\frac{\g^4s_0}{2^{10}\tC}\Big)^{\frac{1}{2}}
\Big(\frac{1}{2^8}\ln\ln\frac{\bar{r}}{r_0}\Big)^{\frac{\mathtt{q}-1}{2}}
\right]\,,
\end{equation}
where $[\cdot]$ is the integer part and 
where $\tC>0$ is the absolute constant given by Proposition \ref{shulalemma}.

\vspace{0.5em}
\noindent
We claim that  \eqref{cond:r0piccolo} and \eqref{smalldelta0} implies 
that $r_0$ satisfies the smallness condition \eqref{conditionRstar}
with $\tK(r_0)$ in \eqref{Kappone} and $J_{\tK}^{*}=J_{\tK}^{\SE}$.
In other words we prove that 
\begin{equation*}
\mathtt{C}_0:= 32 e4^3 |R_0|_{\bar{r},\tw_0} \tK\left(\frac{4}{\gamma^{4}}\right)^{\tK}
\exp\exp\Big(\big({\frac{\tK^2\tC}{s_0}} \big)^{\frac{1}{\mathtt{q}-1}}\Big)
\frac{r_0}{\bar{r}} \leq 1\,.
\end{equation*}
In fact,  with that choice of $\tK$, we can easily check that conditions \eqref{cond:r0piccolo} 
and \eqref{smalldelta0}
 actually implies a much stronger bound:
\begin{equation*}
\begin{aligned}
\mathtt{C}_0 
&\le 2^{16} |R_0|_{\bar{r},\tw_0} 
\exp\set{\exp\Big( \big(\tK^2 \frac{2^4\tC}{\g^{4}s_0}\big)^{\frac{1}{\mathtt{q}-1}}\Big) - 
\ln \frac{\bar{r}}{r_0}}
\\& 
\stackrel{\eqref{Kappone}}{\le} 2^{16} |R_0|_{\bar{r},\tw_0} 
\exp\set{\exp\big(2^{-14}\ln\ln\frac{\bar{r}}{r_0}\big) - 
\ln \frac{\bar{r}}{r_0} } 
\\& 
\stackrel{\eqref{stimaNEM}}{\le} 
2^{16} \mathtt{C}(p,R) |F|_{R} \bar{r}
\exp\set{\pa{\ln\frac{\bar{r}}{r_0}}^{\frac{1}{2^{14}}} - 
\ln\frac{\bar{r}}{r_0} } 
\\&
\le  
2^{16}\mathtt{C}(p,R) |F|_{R} \bar{r} 
\exp\set{-\frac{1}{2} \ln\frac{\bar r}{r_0} } \le 1
\end{aligned}
\end{equation*}
where the last bound holds provided that 
 \begin{equation}
\label{cond:r0 3} 
 r_0/\bar{r} \le \exp 2^{\frac{2^{14}}{1 - 2^{14}}}\,,
 \end{equation}
 and 
\begin{equation}\label{cond:r0 2}
r_0 \le \frac{\bar{r}}{2^{32} (\mathtt{C}(p,R) |F|_{R} \bar{r} )^2}
\stackrel{\eqref{sceltabarr}}{=}
\frac{\mathtt{C}_{\rm alg}(p)}{2^{28} C(p,R) |F|_{R}R}\,,
\end{equation}
are satisfied. 
Since $r_0=2\delta$ and $\delta\leq \delta_{\SE}$, we have that condition 
\eqref{smalldelta0} implies \eqref{cond:r0 2}
if one requires
 \begin{equation}\label{condC21}
 C_{2}=C_{2}(p,R)\geq \frac{2^{29}C(p,R)R}{\mathtt{C}_{\rm alg}(p)}\,.
 \end{equation}
 Using \eqref{cond:r0piccolo} we notice that \eqref{smalldelta0} implies \eqref{cond:r0 3}
if one requires
 \begin{equation}\label{condC22}
  C_{2}=C_{2}(p,R)\geq\frac{8\mathtt{C}_{\rm alg}(p)C(p,R)}{R}\,.
 \end{equation}
Theorem \ref{mainthm:BNF}, together with Lemma \ref{tempotempo} 
(recall also Remark \ref{rmk:kernel} ), implies 
that the solution $u(t)$ of \eqref{eq:beamComp} evolving from initial 
data satisfying \eqref{smallnessu0}
remains in the ball of radius $2\delta$ for time $t\in [0,{T}_0]$ 
(i.e. the bound \eqref{claimclaim} is satisfied) with
(see the bound \eqref{sonno3})
%
%
%
%
 \begin{equation*}
 \begin{aligned}
{T}_0 := 
\frac{\bar{r}^{\tK + 1}}{8 r_0^{\tK+1}\norm{R_0}_{\bar{r},\tw_f}} 
\frac{(16 e\tK 4^{\tK + 2})^{-\tK }}{\pa{\g^{-4\tK}\exp\exp\tK^2c/s_0}^\tK}  
&\ge  \frac{1}{8 \norm{R_0}_{\bar{r},\tw_f}} 
\frac{\bar{r}}{r_0} 
\frac{\pa{\frac{\bar{r}}{r_0}}^\tK}{\Big(\pa{\frac{2^9}{\gamma^4}}^\tK \tK\exp\exp\tK^2c/s_0\Big)^\tK} \,.
\end{aligned}
 \end{equation*}
 We have to show that, thanks to \eqref{Kappone} and \eqref{sceltabarr}
 and the smallness of $r_0$, the time ${T}_0$ 
 above satisfies the bound \eqref{longtime1}.
 Actually we prove the slightly better bound
 \[
 T_0\geq \frac{1}{8 \norm{R_0}_{\bar{r},\tw_f}}
 \frac{\bar{r}}{r_0} \exp\pa{\frac{1}{2}\ln\pa{\frac{\bar{r}}{r_0}}
\pa{\frac{\g^4s_0}{2^{16}\mathtt{C}}\ln\ln\frac{\bar{r}}{r_0}}^{\frac{\mathtt{q}-1}{2}}}\,.
 \] 
 Indeed  $\delta_{\SE} \leq \bar{r}/2$ using  \eqref{sceltabarr} and \eqref{smalldelta0}
 and  taking  $C_2(p,R)$ as in \eqref{condC22}.
 Let us observe that 
 $$
 \begin{aligned}
 \exp\pa{\tK\ln\frac{r_0}{\bar{r}} + \tK^2\ln\pa{\frac{2^9}{\gamma^4}}}
 \pa{\tK\exp\{e^{\tK^2\mathtt{C}/s_0}\}}^\tK 
 &\le  
 \exp\pa{\tK\ln\frac{r_0}{\bar{r}}} 
 \exp\pa{\tK^2\ln\pa{\frac{2^9}{\gamma^4}}+ {\tK^2e^{\tK^2\mathtt{C}/s_0}}}
 \\
 &\le 
 \exp\pa{\tK\ln\frac{r_0}{\bar{r}}} \exp\exp\pa{\frac{2^{10}\mathtt{C}}{\g^4s_0}\tK^2 },
 \end{aligned}
 $$
 hence
 \begin{equation}
 \label{time}
 \begin{aligned}
{T}_0 &\ge 
\frac{1}{8 \norm{R_0}_{\bar{r},\tw_f}}
 \frac{\bar{r}}{r_0}  \exp\pa{\tK\ln\frac{\bar{r}}{r_0} 
- \exp\pa{\frac{2^{10}\mathtt{C}}{\g^4s_0}\tK^2}} 
\\&
\stackrel{\eqref{Kappone}}{\ge} 
\frac{1}{8 \norm{R_0}_{\bar{r},\tw_f}} \frac{\bar{r}}{r_0}  \exp\pa{\left[
\Big(\frac{\g^4s_0}{2^{10}\tC}\Big)^{\frac{1}{2}}
\Big(\frac{1}{2^8}\ln\ln\frac{\bar{r}}{r_0}\Big)^{\frac{\mathtt{q}-1}{2}}
\right]\ln\frac{\bar{r}}{r_0} - \exp\left( \left(\frac{1}{2^8}\ln\ln\frac{\bar{r}}{r_0}\right)^{\mathtt{q}-1}\right)} 
\\& 
{\ge} 
\frac{1}{8 \norm{R_0}_{\bar{r},\tw_f}} \frac{\bar{r}}{r_0}  \exp\pa{\left[
\Big(\frac{\g^4s_0}{2^{10}\tC}\Big)^{\frac{1}{2}}
\Big(\frac{1}{2^8}\ln\ln\frac{\bar{r}}{r_0}\Big)^{\frac{\mathtt{q}-1}{2}}
\right]\ln\frac{\bar{r}}{r_0} - \exp\left(\left( \ln\ln^{1/{2^8}}\frac{\bar{r}}{r_0}\right)^{\mathtt{q}-1}\right)} \\
&
{\ge} 
\frac{1}{8 \norm{R_0}_{\bar{r},\tw_f}} \frac{\bar{r}}{r_0}  \exp\pa{\left[
\Big(\frac{\g^4s_0}{2^{10}\tC}\Big)^{\frac{1}{2}}
\Big(\frac{1}{2^8}\ln\ln\frac{\bar{r}}{r_0}\Big)^{\frac{\mathtt{q}-1}{2}}
\right]\ln\frac{\bar{r}}{r_0} - \exp\left(\left( \ln\ln^{1/{2^8}}\frac{\bar{r}}{r_0}\right)\right)} \\
&
\ge \frac{1}{8 \norm{R_0}_{\bar{r},\tw_f}} \frac{\bar{r}}{r_0}  \exp\pa{\ln\frac{\bar{r}}{r_0}
\pa{\Big(\frac{\g^4s_0}{2^{10}\tC}\Big)^{\frac{1}{2}}
\Big(\frac{1}{2^8}\ln\ln\frac{\bar{r}}{r_0}\Big)^{\frac{\mathtt{q}-1}{2}} - 2}} 
\\&
\ge \frac{1}{8 \norm{R_0}_{\bar{r},\tw_f}} \frac{\bar{r}}{r_0}  
\exp\pa{\frac{1}{2}\ln\pa{\frac{\bar{r}}{r_0}}
\pa{
\Big(\frac{\g^4 s_0}{2^{18}\mathtt{C}}\Big)^{1/2}\Big(\ln\ln\frac{\bar{r}}{r_0} \Big)^{\frac{\mathtt{q}-1}{2}}}}
\end{aligned}
 \end{equation}
where we required 
that
\begin{align}
r_0 &\le \frac{\bar{r}}{\exp\exp 2^8}\,, \label{cardano1}
\\
r_0&\leq \bar{r} \exp\exp\pa{-\Big(\frac{2^{22}\mathtt{C}}{\g^4s_0}\Big)^{\frac{1}{\mathtt{q}-1}}}\,.
\label{conditio r}
\end{align}
We have that conditions
\eqref{smalldelta0} and \eqref{sceltabarr} imply \eqref{cardano1}
if one requires
 \begin{equation}\label{condC23}
 C_{2}=C_{2}(p,R)\geq \frac{2\exp\exp \{2^8\} \mathtt{C}_{\rm alg}(p)C(p,R)}{R}\,.
 \end{equation}
 The  bound \eqref{conditio r} is implied by the \eqref{smalldelta0}
setting $\mathtt{c}=3\cdot 2^{15}\mathtt{C}$
and 
\begin{equation}\label{sceltaC1}
C_1=C_1(p,R)\geq \frac{R}{\mathtt{C}_{\rm alg}(p)C(p,R)}\,.
\end{equation}
The bound \eqref{time}, together with \eqref{stimaNEM}, implies 
the lower bound \eqref{longtime1} by setting 
\begin{equation}\label{sceltaC3}
C_{3}=C_{3}(p,R)\geq 16 C(p,R)\,.
\end{equation}
Theorem \ref{main:subexp} follows by the discussion above choosing 
the constant
$C_1$ as in \eqref{sceltaC1}, $C_{3}$ as in \eqref{sceltaC3}, $C_2$ satisfying 
\eqref{condC21}, \eqref{condC22}, \eqref{condC23}
where $\mathtt{C}_{\rm alg}(p)$ is given in \eqref{costanti algebra}
and $C(p,R)$ in \eqref{stimaNEM}.

\subsection{Sobolev stability and proof of Theorem \ref{main:sobol}}
In order to prove Theorem \ref{main:sobol}
we reason as done in Section \ref{sec:subexpstability} and we assume 
that the initial condition  $u_0$ satisfies (see \eqref{normnorma})
\begin{equation}\label{smallnessu0Sob}
\|u_0\|_{p}:=\|u_0\|_{0,p}\leq \delta\,.
\end{equation}
Fix 
\begin{equation}\label{KapponeSob}
\tK:=\tK(p):=\left[\Big(\frac{p-1}{2^4 3^4}\Big)^{\frac{1}{3}}\right]-1\,,\qquad 
p_0:=p-\zeta:=p-2^4 3^4\sum_{i=1}^{\tK}i^2\,.
\end{equation}
Recall that in Theorem \ref{main:sobol} we
required  $p>2^{6}(36)^{2}+1$.
Then one can check that
\begin{equation}\label{boundsKappone}
1\leq \widetilde{c} (p-1)^{\frac{1}{3}}\leq \tK(p)\leq p^{1/3}\,,\qquad 
\widetilde{c}:=\frac{1}{2(36)^{2/3}}\,,\quad p_0>1\,.
\end{equation}
For future convenience we set
\begin{equation}\label{sceltabarrSob}
\bar{r}:=\frac{R}{\mathtt{C}_{{\rm alg}, \mathtt{M}}(p_0) |F|_{R}}\,.
\end{equation}
Therefore by item $(ii)$ in Lemma \ref{lem:applicoSub} we have that the Hamiltonian 
$H$ in  \eqref{beamHam} can be written, for any $\tK\geq1$, in the form 
\eqref{hamIniz} 
with  $\tw_0$ as in \eqref{abbondanzaSob}.
Here again we apply Theorem \ref{mainthm:BNF}
to the Hamiltonian $H=D_{\omega}+R_0$ with $R_0$ in \eqref{abbondanza}.
Recalling the parameters in \eqref{parametri}-\eqref{patata2}
we fix
\begin{equation}\label{cond:r0piccoloSob}
r_0= 2\delta\,.
\end{equation}
We claim that  \eqref{cond:r0piccoloSob} and \eqref{smalldelta0} imply 
that $r_0$ satisfies the smallness condition \eqref{conditionRstar}
with $\tK(p)$ in \eqref{KapponeSob} and $J_{\tK}^{*}=J_{\tK}^{\SO}$.
That is, we show that
\begin{equation*}
\mathtt{C}_0:= 32 e4^3 |R_0|_{\bar{r},\tw_0} \tK\left(\frac{4}{\gamma^{4}}\right)^{\tK}
\exp\Big( 2^{12}\mathtt{C} \tK^3\Big)
\frac{r_0}{\bar{r}} \leq 1\,.
\end{equation*}
First, we notice that
\begin{equation}\label{cardano2}
p_0>1\qquad \Rightarrow
\quad
1\leq \mathtt{C}_{{\rm alg},\mathtt{M}}(p_0):=\sqrt{2}\sqrt{2+\frac{2p_0+1}{2p_0-1}}\leq 2^3\,.
\end{equation}
Hence
\[
\begin{aligned}
\mathtt{C}_0&\leq  |R_0|_{\bar{r},\tw_0} \exp\Big\{2^{19}\mathtt{C}\ln(1/\gamma) \tK^3\Big\}
\frac{r_0}{\bar{r}}
\\&
\stackrel{\eqref{stimaNEMSob}, \eqref{sceltabarrSob}, \eqref{KapponeSob}}{\leq }
\frac{|F|_{R}}{R}\exp\Big\{2^{21}\mathtt{C}\ln(1/\gamma) p\Big\}r_0\leq 1\,,
\end{aligned}
\]
provided that
\[
r_0\leq \frac{R}{|F|_{R}}\exp\Big\{-2^{21}\mathtt{C}\ln(1/\gamma) p\Big\}.
\]
The last inequality follows from \eqref{cond:r0piccoloSob} 
and \eqref{smalldelta1Sob},
taking 
\begin{equation}\label{melone1}
\mathtt{c}\geq2^{21}\mathtt{C}\,.
\end{equation} 
In view of  Theorem \ref{mainthm:BNF} and Lemma \ref{tempotempo} 
we have 
that the solution $u(t)$ of \eqref{eq:beamComp} evolving from initial 
data satisfying \eqref{smallnessu0Sob}
remains in the ball of radius $2\delta$ for time $t\in [0,{T}_0]$ with
(recall the estimate \eqref{sonno3} and $J_{\tK}^{\SO}$ in \eqref{conditionRstar})
 \begin{equation*}
 \begin{aligned}
{T}_0 := 
\frac{\bar{r}^{\tK + 1}}{8 r_0^{\tK+1}\norm{R_0}_{\bar{r},\tw_f}} \frac{(16 e\tK 4^{\tK + 2})^{-\tK }}{\pa{\g^{-4\tK}\exp(2^{12}\mathtt{C}\tK^3)}^\tK}  \,.
\end{aligned}
 \end{equation*}
Observe that
\[
\begin{aligned}
T_0&\geq \frac{1}{8|R_0|_{\bar{r},\tw}}\frac{\bar{r}}{r_0}\left(\frac{\bar{r}}{r_0}\right)^{\tK}
\frac{1}{[\exp\big(2^{14}\mathtt{C}\ln(1/\gamma)\tK^3
\big)]^\tK}
\\&
\stackrel{\eqref{stimaNEMSob}, \eqref{sceltabarrSob}}{\geq}
\frac{R}{8 \mathtt{C}_{\rm alg}(p_0) |F|_{R}}\frac{1}{r_0}
\left(\frac{\bar{r}}{r_0}\right)^{\tK}
\frac{1}{[\exp\big(2^{14}\mathtt{C}\ln(1/\gamma)\tK^3
\big)]^\tK}
\\&
\geq \frac{R}{|F|_{R}}\frac{1}{r_0}\left(\frac{\bar{r}}{r_0}\right)^{\tK}
\frac{1}{2^6 [\exp\big(2^{14}\mathtt{C}\ln(1/\gamma)\tK^3
\big)]^\tK},
\end{aligned}
\]
since $p\geq p_0$ and \eqref{cardano2} holds.
We also remark that condition \eqref{smalldelta1Sob}
with 
\eqref{sceltabarrSob}-\eqref{cond:r0piccoloSob}
implies that $\delta_{\SO}\leq \bar{r}/4$. Finally,
by \eqref{boundsKappone}
one has
\[
\begin{aligned}
T_0&\geq 
 \frac{R}{|F|_{R}}\frac{1}{r_0}\left(\frac{\delta_{\SO}}{\delta}\right)^{\widetilde{c} (p-1)^{1/3}}
 \frac{1}{2^6 
 [\exp\big(p 2^{14}\mathtt{C}\ln(1/\gamma)
\big)]^p}
\\&
\geq
 \frac{R}{|F|_{R}}\frac{1}{r_0}\left(\frac{\delta_{\SO}}{\delta}\right)^{\widetilde{c} (p-1)^{1/3}}
 \frac{1}{
 [\exp\big(p 2^{15}\mathtt{C}\ln(1/\gamma)
\big)]^p}\,.
\end{aligned}
\]
Setting $\mathtt{c}=\max\{ 2(36)^{2/3}, 2^{21}\mathtt{C}\}$, we get the thesis.

\subsection{Sobolev stability optimization and proof of Corollary \ref{coroOttimo}} Let us fix $\delta$ such that
\begin{equation}\label{deltapiccolo}
0<\delta\leq \bar{ \delta}:=\delta_{\SO}\exp\Big\{-\mathtt{b}
\ln(1/\gamma)\Big\}
\end{equation}
where 
\begin{equation}\label{constB}
\mathtt{b}:=\max\big\{24\mathtt{c}^{2}\left(\frac{1}{48\mathtt{c}}\right)^{\frac{10}{9}},
24\mathtt{c}^{2}\big[2^6 (36)^{2}\big]^{5/3}
\big\}\,,
\end{equation}
and let us consider 
\begin{equation}\label{defpdelta}
p=p(\delta)=1+\left(\frac{1}{24\mathtt{c}^2\ln(1/\gamma)} \ln\big(\frac{\delta_{\SO}}{\delta}\big)\right)^{3/5}
\end{equation}
 where $\mathtt{c}>0$ is the absolute constant given by Theorem \ref{main:sobol} and $\delta_{\SO}$
 given in \eqref{smalldelta1Sob}.
%

\noindent
In order to prove Corollary \ref{coroOttimo}
we reason as above and we assume  that 
that 
\begin{equation*}
\|u_0\|_{p}:=\|u_0\|_{0,p}\leq \delta\,.
\end{equation*}
Our aim is to apply Theorem \ref{main:sobol}.
We shall verify that condition \eqref{smalldelta1Sob} holds for $\delta$ small enough.
First notice that \eqref{defpdelta}
implies 
\begin{equation}\label{defpdelta2}
\delta= \delta_{\SO}\exp\big\{-24\mathtt{c}^2(p-1)^{\frac{5}{3}}\ln(1/\gamma)\big\}\,.
\end{equation}
Then  smallness condition \eqref{smalldelta1Sob} translates in proving that
\[
\exp\{ \mathtt{c}\ln(1/\gamma)\Big[ p-24\mathtt{c}(p-1)^{5/3}\Big]\}\leq1\quad \Leftrightarrow
\quad 
p-24\mathtt{c}(p-1)^{5/3}\leq0\,,
\]
recalling that $0<\gamma<1$, which is true as long as

\[
\begin{aligned}
p&=1+\left(\frac{1}{24\mathtt{c}^2\ln(1/\gamma)} \ln\big(\frac{\delta_{\SO}}{\delta}\big)\right)^{3/5} 
\geq 1+\big(\frac{1}{48\mathtt{c}}\big)^{\frac{2}{3}}\,,\quad \Leftrightarrow
\\
& \ln(\delta_{\SO}/\delta)\geq 24\mathtt{c}^2 \left(\frac{1}{48\mathtt{c}}\right)^{\frac{10}{9}}\ln(1/\gamma)\,.
\end{aligned}
\]
This follows by \eqref{deltapiccolo}-\eqref{constB}.
With similar computations we can check that \eqref{deltapiccolo}-\eqref{constB}, 
together with \eqref{defpdelta}, yield
$p=p(\delta)>1+2^6(36)^{2}$.
Hence,Theorem \ref{main:sobol} applies, guaranteeing time of stability of the form
\[
\begin{aligned}
T_0&\geq\frac{R}{2 |F|_{R}\delta}
\left(\frac{\delta_{\SO}}{\delta}\right)^{\frac{1}{\mathtt{c}}(p-1)^{1/3}}
\exp\big\{ -p^2\mathtt{c}\ln(1/\gamma)\big\}
\\&
\stackrel{\eqref{defpdelta2}}{=}
\frac{R}{2 |F|_{R}\delta}
\exp\Big\{ 24\mathtt{c}(p-1)^{2}
\ln(1/\gamma)
-p^2\mathtt{c}\ln(1/\gamma)
\Big\}
\\&
\geq \frac{R}{2 |F|_{R}\delta}\exp\Big\{
\mathtt{c}\ln(1/\gamma) \Big(24(p-1)^{2}-p^2\Big)
\Big\}
\\&
\geq \frac{R}{2 |F|_{R}\delta}\exp\Big\{
\mathtt{c}(p-1)^{2}\ln(1/\gamma) 
\Big\}
\\&
\stackrel{\eqref{defpdelta}}{\geq}
\frac{R}{2 |F|_{R}\delta}\exp\Big\{
\frac{\mathtt{c}(\ln(1/\gamma))^{-1/5})}{(24\mathtt{c}^2)^{6/5}}(\ln(\delta_{\SO}/\delta) )^{1+\frac{1}{5}}
\Big\}\,,
\end{aligned}
\]
which is the stated bound \eqref{longtime1Sobcoro}.

\appendix

\section{Technical Lemmata}\label{app:techlem}
We collect some technical lemmata.

\begin{lemma}\label{constance beam sub-immersion}
Consider the function 
\begin{equation}\label{roadsport}
\lambda(x):=(\log(2+x))^{\mathtt{q}}\,,\;\;\; 1< \mathtt{q}\leq 2\,,\quad x > 0\,.
\end{equation}
Then, there exists a constant $1\leq \kappa\leq 5/4$ such that  for any integer $N\geq 4$
and $x_2\geq x_3\geq \cdots \geq x_N\geq 1$, 
 \begin{equation}\label{mortazza}
\sum_{\ell= 2}^N \lambda(x_\ell)
\geq \lambda\Big(\sum_{\ell= 2}^N x_\ell\Big)+
c
\sum_{\ell= 3}^N \lambda(x_\ell)\,,
\qquad
\text{with}\ \ 
c:=1-\frac{1}{\kappa}\geq 0\,.
\end{equation}
\end{lemma}

\begin{proof}
First of all we notice that
the function $\lambda$ is sub-linear i.e.  
\begin{equation}\label{arrosticini}
\lambda(x+y)\leq \lambda(x)+\lambda(y)\,,\qquad \forall\, x,y\, \in\, \mathbb N_+\,.
\end{equation}
Recall \eqref{es:fgrowth} and that $\langle x+y\rangle\leq\langle x\rangle 
+\langle y\rangle  $ 
and set
\[
F(x,y):=(\ln(2+ x+y))^\mathtt{q}-(\ln(2+x))^{\mathtt{q}}\,,\qquad 
\forall x,y\in \mathbb{R}\,,\quad x,y\geq0\,.
\]
To obtain \eqref{arrosticini}
it is sufficient to show that,
for any $y\in \mathbb{R}$, $y\geq1$, one has 
\begin{equation}\label{arrosticini2}
F(x,y)\leq (\ln(2+y))^{\mathtt{q}}\,,\qquad
\forall x\in \mathbb{R}\,,\quad x\geq1\,.
\end{equation}
We claim that for \emph{fixed} $y\geq1$ the function $[1,+\infty]\ni x\mapsto F(x,y)$
is decreasing. Then \eqref{arrosticini2} follows since $F(0,y)\leq (\ln(2+y))^{\mathtt{q}}$.
We have
\[
\partial_{x}F(x,y)=\mathtt{q}\left\{
\frac{(\ln(2+x+y))^{\mathtt{q}-1}}{2+x+y}
-\frac{(\ln(2+x))^{\mathtt{q}-1}}{2+x}
\right\}\,,
\]
hence $\partial_{x}F(x,y)=0$ for some $x\geq1$ if and only if
\begin{equation}\label{arrosticini3}
g(x,y)=\frac{(\ln(2+x))^{\mathtt{q}-1}}{2+x}\,,
\end{equation}
where
\[
g(x,y):=\frac{(\ln(2+x+y))^{\mathtt{q}-1}}{2+x+y}\,.
\]
The right hand side of the equation \eqref{arrosticini3} does not depend in $y\in\mathbb{R}$.
One can check (using also that $1<\mathtt{q}\leq 2$) 
that for fixed $x$, the function $y\mapsto g(x,y)$ is decreasing if $y\geq1$.
Hence \eqref{arrosticini3} cannot be satisfied, and, on the other hand, one has
\[
g(x,y)< \frac{(\ln(2+x))^{\mathtt{q}-1}}{2+x}\,,\quad x,y\geq1\,,
\] 
which implies $\partial_{x}F(x,y)<0$ and hence the claim. Then 
\eqref{arrosticini2} holds true.
We claim that
there is a constant $1\leq \kappa\leq 5/4$  such that
\begin{align}
x \lambda(1) &\geq \kappa \lambda(x)\,,\qquad \forall\, x\geq 1\,, \label{porchetta}
\\
\lambda'(x) &\geq \kappa \lambda'(2 x)\,,\qquad\forall\, x\geq  1\,,\label{coppa}
\\
x \lambda(2) &\geq \kappa \lambda(2x)\,,\qquad\forall\, x\geq 1\,.\label{ariccia}
\end{align}
By using \eqref{roadsport} one has that conditions 
\eqref{porchetta}, \eqref{coppa}, \eqref{ariccia} are equivalent to
\begin{equation}\label{ponte1}
\kappa\leq \min_{i=1,2,3}\big\{\inf_{x\geq1}h_{1}(x), \inf_{x\geq1}h_{2}(x), \inf_{x\geq1}h_{3}(x)\big\}\,,
\end{equation}
setting
\[
h_1(x):=x\Big(\frac{\log(3)}{\log(2+x)}\Big)^{\mathtt{q}}\,,
\quad
h_2(x):=\frac{2+2x}{2+x}\Big(\frac{\log(2+x)}{\log(2+2x)}\Big)^{\mathtt{q}-1}\,,
\quad
h_3(x):=x\Big(\frac{\log(4)}{\log(2+2x)}\Big)^{\mathtt{q}}\,.
\]
Notice that
\[
\begin{aligned}
h_{1}'&=\Big(\frac{\log(3)}{\log(2+x)}\big)^{\mathtt{q}}
\Big[1-\frac{\mathtt{q}x}{2+x}\frac{1}{\log(2+x)}\Big]\,,
\\
h_{3}'&=\Big(\frac{\log(4)}{\log(2+2x)}\big)^{\mathtt{q}}
\Big[1-\frac{2\mathtt{q}x}{2+2x}\frac{1}{\log(2+2x)}\Big]\,,
\\
h_2'&=\frac{1}{(2+x)^2}\Big(\frac{\log(2+x)}{\log(2+2x)}\big)^{\mathtt{q}-1}\Big[
2-(\mathtt{q}-1)\Big(
\frac{4+2x}{\ln(2+2x)}-\frac{2+2x}{\ln(2+x)}
\Big)\Big]\,.
\end{aligned}
\] 
Consider the function $h_1(x)$. One can note that
\[
\begin{aligned}
&h_1'(x)\geq 0 \;\;\; \Leftrightarrow\;\;\; \mathtt{q}x\leq (2+x)\ln(2+x) \;\;\; \Leftrightarrow\;\;\;
1\leq \frac{1}{\mathtt{q}}\ln(2+x)+\frac{2}{2+x} \,,
\end{aligned}
\]
which is true for $x\geq1$, since $1<\mathtt{q}\leq 2$.  Hence $\min_{h\geq1} h_1(x)\geq1$.
With similar computations one can check that the constant $\kappa$ in \eqref{ponte1}
can be fixed as $\kappa=1$. Then \eqref{porchetta}, \eqref{coppa}, \eqref{ariccia} hold.


\smallskip
\noindent
We are now in position to prove  \eqref{mortazza}.
Let $1\leq N_0\leq N$ such that
\[
x_{N_0+1}=\ldots =x_N=1\,, \qquad x_{N_0}\geq 2\,.
\]
Let $N_1:=N-N_0\geq 0$. Define 
\[
g(x_2,x_3,\ldots,x_{N_0})
:=
\sum_{\ell= 2}^{N_0} \lambda(x_\ell)
- \lambda\Big(N_1+\sum_{\ell= 2}^{N_0} x_\ell\Big)
-c\sum_{\ell= 3}^{N_0} \lambda(x_\ell)+(1-c)N_1 \lambda(1)\,.
\]
In order to prove \eqref{mortazza} we have to show that $g\geq 0.$
If $N_0=1$ we have that
\[
g=-\lambda(N-1)+(1-c)(N-1)\lambda(1)
\stackrel{\eqref{porchetta}}{\geq}
[(1-c-\frac{1}{\kappa})(N-1)]\lambda(1)\geq 0\,.
\]
Assume now that $N_0\geq 2.$
We have
\begin{equation}\label{cippa}
\partial_{x_2} g=\lambda'(x_2)-\lambda'\Big(N_1+\sum_{\ell= 2}^{N_0} x_\ell\Big) \geq 0\,,
\end{equation} 
 since  $\lambda'(x)$ is decreasing for $x\geq 2$.  
 If $N_0=2$ we have that
\begin{equation*}
\begin{aligned}
g(x_2)\geq g(2)
&=
\lambda(2)-f(N)+(1-c)(N-2)\lambda(1)
\\
&\stackrel{\eqref{arrosticini}}\geq
-\lambda(N-2)+(1-c)(N-2)\lambda(1)
\stackrel{\eqref{porchetta}}\geq
(1-c-\frac{1}{\kappa})(N-2)\lambda(1)\geq 0\,.
\end{aligned}
\end{equation*}
Assume now that $N_0\geq 3$. 
Then, since $x_2\geq x_3$, by \eqref{cippa} we get
 \begin{equation*}
 \begin{aligned}
 g(x_2,x_3,&\ldots,x_{N_0})
 \geq
 g(x_3,x_3,\ldots,x_{N_0})
\\& 
=(2-c)\lambda(x_3)
- \lambda\Big(N_1+2x_3+\sum_{\ell= 4}^{N_0} x_\ell\Big)
+(1-c)\sum_{\ell= 4}^{N_0} \lambda(x_\ell)
+(1-c)N_1 \lambda(1)
\\&
=:g_3(x_3,\ldots,x_{N_0})\,.
\end{aligned}
\end{equation*}
Proceeding analogously, since $\lambda'$ 
is decreasing and positive, we get
\begin{equation}\label{lippa}
\begin{aligned}
\partial_{x_3}  g_3(x_3,\ldots,x_{N_0})
&=
2\Big(
\Big(1-\frac{c}{2}\Big)\lambda'(x_3)
- \lambda'\Big(N_1+2x_3+\sum_{\ell= 4}^{N_0} x_\ell\Big)
\Big)
\\&
\geq
2\left(
\Big(1-\frac{c}{2}\Big)\lambda'(x_3)
- \lambda'(2x_3)
\right)
\stackrel{\eqref{coppa}}\geq
2\Big(1-\frac{c}{2}-\frac{1}{\kappa}\Big)\lambda'(x_3)
\geq 0\,.
\end{aligned}
\end{equation}
Let us note that by \eqref{arrosticini} and \eqref{porchetta}
we have
\[
\lambda(N+1)\leq \lambda(5)+\left(1+\frac{N-5}{\kappa}\right)\lambda(1)\,,
\qquad \forall\, N\geq 4\,.
\]
Then, if $N_0=3,$ and recalling that $\kappa\leq 5/4$ we get
\begin{equation*}
\begin{aligned}
g_3(x_3)
&\stackrel{\mathclap{\eqref{lippa}}}\geq g_3(2)
\\&
=(2-c)\lambda(2)- \lambda(N+1)+(1-c)(N-3) \lambda(1)
\\&
\geq (2-c)\lambda(2)
- \lambda(5)-\Big(1+\frac{N-5}{\kappa}\Big)\lambda(1)
+(1-c)(N-3) \lambda(1)
\\&
\stackrel{\mathclap{\eqref{ariccia}}}\geq
\Big(2-c-\frac{5}{2\kappa}\Big)\lambda(2)
+\Big((1-c)(N-3)-1-\frac{N-5}{\kappa}\Big)\lambda(1)
\\&
\stackrel{\mathclap{\eqref{porchetta}}}\geq
\Big( 
\frac{4-2c}{\kappa} -\frac{5}{\kappa^2}+
(1-c)(N-3)-1-\frac{N-5}{\kappa}\Big)\lambda(1)
\geq 0\,,
\end{aligned}
\end{equation*}
by definition of $c$.
Otherwise, for $N_0\geq 4$ we get, since $x_3\geq x_4$,
\begin{equation*}
\begin{aligned}
g_3(x_3,x_4,&\ldots,x_{N_0})\stackrel{\eqref{lippa}}{\geq}
 g_3(x_4,x_4,\ldots,x_{N_0})
\\&
=(3-2c)\lambda(x_4)
- \lambda\Big(N_1+3x_4+\sum_{\ell= 5}^{N_0} x_\ell\Big)
+(1-c)
\sum_{\ell= 5}^{N_0} \lambda(x_\ell)
+(1-c)N_1 \lambda(1)
\\&
=:g_4(x_4,\ldots,x_{N_0})
\end{aligned}
\end{equation*}
Proceeding in this way we obtain by induction,  for $3\leq n\leq N_0$ 
functions
\begin{equation*}
\begin{aligned}
g_n&=g_n(x_n,\ldots,x_{N_0})
\\&
:=
 \big((n-1)-(n-2)c\big)\lambda(x_n)
- \lambda\Big(N_1+(n-1)x_n+\sum_{\ell= n+1}^{N_0} x_\ell\Big)
\\&
+(1-c)
\sum_{\ell= n+1}^{N_0} \lambda(x_\ell)
+(1-c)N_1 \lambda(1)\,,
\end{aligned}
\end{equation*}
with
\[
g_{n-1}(x_{n-1},\ldots,x_{N_0})
\geq 
g_n(x_n,\ldots,x_{N_0})\,,
\qquad
\forall\, x_{n-1}\geq x_n\,.
\]
Indeed, since
\begin{equation*}
\begin{aligned}
\partial_{x_n} &g_n(x_n,\ldots,x_{N_0})=
\\&
=
 \big((n-1)-(n-2)c\big)\lambda'(x_n)
- (n-1)\lambda'\Big(N_1+(n-1)x_n+\sum_{\ell= n+1}^{N_0} x_\ell\Big)
\\&
\geq 
(n-1)
\left[
\left(1-\frac{n-2}{n-1}c\right)\lambda'(x_n)
- \lambda'(2 x_n)
\right]
\\&
\stackrel{\eqref{coppa}}
\geq
(n-1)
\left(1-\frac{n-2}{n-1}c-\frac{1}{\kappa}\right)\lambda'(x_n)
\geq 0\,,
\end{aligned}
\end{equation*}
we have that
for every $x_n\geq x_{n+1}$
\[
g_n(x_n,x_{n+1},\ldots,x_{N_0})
\geq 
g_n(x_{n+1},x_{n+1},\ldots,x_{N_0})
=
g_{n+1}(x_{n+1},\ldots,x_{N_0})\,.
\]
In conclusion we get
\begin{equation*}
\begin{aligned}
g(x_2,x_3,&\ldots,x_{N_0})
\geq
g_{N_0}(x_{N_0})
\\&
=
 \big((N_0-1)-(N_0-2)c\big)\lambda(x_{N_0})
- \lambda\big(N-N_0+(N_0-1)x_{N_0}\big)
\\&
+(1-c)(N-N_0) \lambda(1)
\end{aligned}
\end{equation*}
with $\partial_{x_{N_0}} g_N(x_{N_0})\geq 0$.
Then, recalling that we are in the case $N_0\geq 4,$  
we get
\begin{equation*}
\begin{aligned}
g_{N_0}(x_{N_0})&\geq g_{N_0}(2)=
\\&
 =\big((N_0-1)-(N_0-2)c\big)\lambda(2)
- \lambda(N+N_0-2)
+(1-c)(N-N_0) \lambda(1)
\\&\stackrel{\eqref{arrosticini}}\geq
(N_0-1) \lambda(2)-c (N_0-2) \lambda(2)
-\lambda(2N_0-2)- \lambda(N-N_0)
+(1-c)(N-N_0) \lambda(1)
\\&
\stackrel{\eqref{ariccia},\eqref{porchetta}}\geq
\left[\left(1-\frac{1}{\kappa}\right)(N_0-1) 
-c (N_0-2)
\right]\lambda(2)
+\left(1-c-\frac{1}{\kappa}\right)(N-N_0) \lambda(1)
\\&
=c\lambda(2)\geq 0\,.
\end{aligned}
\end{equation*}
This complete the proof of \eqref{mortazza}.
\end{proof}
{
\begin{remark}
From the proof of Lemma \ref{constance beam sub-immersion} actually we can deduce that,
if $\lambda(x)$ is defined as in \eqref{roadsport}, 
one can choose the constant $\kappa=1$ in \eqref{porchetta}-\eqref{ponte1}. This means 
that the constant $c$ in \eqref{mortazza} can be chosen to be $0$. However, we state it in this more general form to make it transparent also for a different choice of $\lambda(x)$, provided it is sublinear and satisfies conditions \eqref{porchetta}-\eqref{ponte1}.
\end{remark}
}

\begin{lemma}[\emph{Lemma C.1 in \cite{BMP:CMP}}]\label{chiappa}
For $p,\beta>0$ and $x_0\geq 0$ we have that 
\[
\max_{x\geq x_0} x^p e^{-\beta x}=
\begin{cases}
&(p/\beta)^p e^{-p}
\quad \mbox{if}  \; \ \  
x_0\leq p/\beta\,,
\\
&x_0^p e^{-\beta x_0}
\quad \mbox{if} \; \ \  
x_0> p/\beta\,.
\end{cases}
\]
\end{lemma}
\begin{lemma}\label{luchino}
	Let $x_1\geq x_2\geq \ldots\geq x_N\geq 2.$ Then
	$$
	\frac{\sum_{1\leq\ell\leq N} x_\ell}{\prod_{1\leq\ell\leq N} \sqrt{x_\ell}}
	\leq 
	\sqrt{x_1}+\frac{4}{ \sqrt{x_1}}\,.
	$$
\end{lemma}

\begin{proof}
By induction over $N$.
The result is obvious for $N=1$. 
Let assume it for $N$ and show it for $N+1$.
We have  
	\begin{eqnarray*}
	&&\frac{\sum_{1\leq\ell\leq N+1} x_\ell}{\prod_{1\leq\ell\leq N+1} \sqrt{x_\ell}}
	\leq
	\frac{(\sum_{1\leq\ell\leq N} x_\ell) +x_{N+1}}{(\prod_{1\leq\ell\leq N} \sqrt{x_\ell})\sqrt{x_{N+1}}} 
	\\
	&\leq& 
	\left(\sqrt{x_1}+\frac{4}{ \sqrt{x_1}}\right)
	\frac{1}{\sqrt{x_{N+1}}}+\frac{\sqrt{x_{N+1}}}{\sqrt 2}\,.
	\end{eqnarray*}
	since $x_1\geq x_2\geq \ldots\geq x_N\geq 2$
	implies  $\prod_{1\leq\ell\leq N} \sqrt{x_\ell}\geq\sqrt 2.$
	It remains to prove that
	$$
\left(\sqrt{x_1}+\frac{4}{ \sqrt{x_1}}\right)
	\frac{1}{\sqrt{x_{N+1}}}+\frac{\sqrt{x_{N+1}}}{\sqrt 2}	
	\leq
	\sqrt{x_1}+\frac{4}{ \sqrt{x_1}}\,.
	$$
	Denoting $t:=\sqrt{x_1}$ and $s:=\sqrt{x_{N+1}}$,
	the above inequality is equivalent to 
	$$
	f(t,s):=2 t^2 s-\sqrt 2 t s^2+8s-2 t^2-8\geq 0
	$$
	for $\sqrt 2\leq s\leq t.$ Since $f$ is a concave function of $s$ we have that
	$$f(t,s)\geq \min\{ f(t,\sqrt 2), f(t,t)\}\,.$$
	It is immediate to see that 
	$$
	\min_{t\geq \sqrt 2} f(t,\sqrt 2)=7\sqrt 2 -9>0\,,
	\qquad
	\min_{t\geq \sqrt 2} f(t,t)=12\sqrt 2-16>0\,,
	$$
	showing that $f(t,s)>0$ for $\sqrt 2\leq s\leq t$
	and concluding the proof.
\end{proof}
In the following it will be convenient to use the following way of reordering of the indexes
$j\in \Z$ appearing in the Hamiltonian \eqref{HamPower}.

\begin{definition}\label{n star}
Consider a vector $v=\pa{v_i}_{i\in \Z}$  $v_i\in \N$, $|v|<\infty$. 
	
\noindent
$(i)$ We denote by $\na=\na(v)$ the vector $\pa{\na_l}_{l\in I}$ 
(where $I\subset \N$ is finite)  
which is the decreasing rearrangement of
\[
\{\N\ni h> 1\;\; \mbox{ repeated}\; v_h + v_{-h}\; \mbox{times} \} 
\cup 
\set{ 1\;\; \mbox{ repeated}\; v_1 + v_{-1} + v_0\; \mbox{times}  }
\]
	
\noindent
$(ii)$ Define the vector $m=m(v)$ as the reordering of the elements of the set
\[
\set{j\neq 0 \,,\quad \mbox{repeated}\quad  \abs{u_j} \;\mbox{times}}\,,
\]
where $D<\infty$ is its cardinality, such that
$|m_1|\ge |m_2|\ge \dots\geq |m_D|\ge 1$. 	
\end{definition}

\begin{remark}
We  observe that the number $N:=|\al|+|\bt|$ is the cardinality of \,$\na$
and that, 
by momentum conservation, there 
exists a choice of $\s_i = \pm1, 0$ such that 
\begin{equation}\label{pi e cappucci}
\sum_l \sigma_l\na_l=0\,,\qquad 
\end{equation}
with $\sigma_l \neq 0$  if  $\na_l \neq 1$.
Hence, 
\begin{equation}\label{eleganza}
\na_1\le  \sum_{l\ge 2}\na_l\,,
\end{equation}
Indeed, if $\sigma_1 = \pm 1$, 
the inequality follows directly from \eqref{pi e cappucci}; 
if $\sigma_1 = 0$, then $\na_1=1$, hence $\na_l = 1$  $\forall l$. 
\end{remark}

\medskip
 Given $\al\neq\bt\in\N^\Z,$ with $|\al|+|\bt|<\infty$
 we consider $m=m(\al-\bt)$ and $\na=\na(\al+\bt)$.	
If we denote by $D$ the cardinality of $m$ and $N$ the one of $\na$ we have 
\begin{align}
D+\al_0+\bt_0&\le N\,, \label{cappella}
\\
(|m_1|,\dots,|m_D|,\underbrace{1,\;\dots \;,1}_{N-D\;\rm{times}} )\, 
&\leq\,
\pa{\na_1,\dots \na_N}\,.\label{abbacchio}
\end{align}
Set $\s_l= {\rm sign}(\al_{m_l}-\bt_{m_l})$.
For every function $g$ defined on $\Z$ we have that
\begin{equation}\label{pula2}
\begin{aligned}
\sum_{i\in\Z} g(i) |\al_i-\bt_i|
&=
g(0)|\al_0-\bt_0|+
\sum_{l\geq 1} g(m_l)\,,
\\
\sum_{i\in\Z} g(i) (\al_i-\bt_i)
&=
g(0)(\al_0-\bt_0)+
\sum_{l\geq 1} \s_l g(m_l)\,.
\end{aligned}
\end{equation}

\begin{lemma}\label{lem:constance2SE}
For all $(\al,\bt)\in \mathcal{M}$ (see \eqref{mass-momindici}) the following holds.

\noindent
$(i)$ If 
\begin{equation}\label{divisor}
\sum_i (\al_i-\bt_i)|i|^2 \le 10 \sum_i |\al_i-\bt_i| \,,
\end{equation}
then we have
\begin{equation}\label{constance2SE}
\sum_i |\al_i-\bt_i| \lambda(\sqrt{\jap i})  \le
63 \sum_{l\ge 3} \lambda(\na_l) \le 
\frac{63}{\kappa} \pa{\sum_i \pa{\al_i+\bt_i}\lambda(\jap{i})- 2\lambda(\jap{j})}\,,
\end{equation}
where $\lambda$ is in \eqref{es:fgrowth}
and
\begin{equation}\label{constanceSOB}
\prod_i(1+\abs{\al_i-\bt_i}{\jap{i}}) 
\le e^{27}
N^6\prod_{l=3}^N\na_l^{\tau_0}\,.
\end{equation}
where
$N=|\al|+|\bt|$. 

\noindent
$(ii)$ If on the contrary \eqref{divisor} does not hold then
\begin{equation}\label{pool1}
|\omega\cdot(\alpha-\beta)|\geq1\,,
\end{equation}
where $\omega$ is given in \eqref{dispLaw}.
\end{lemma}
\begin{proof}
Let us prove item $(i)$. Inequality \ref{constanceSOB} 
is contained in \cite[Lemma 7.1]{BMP:CMP} 
so we send the reader there for the related proof. 
For inequality \eqref{constance2SE} we proceed as follows.
We claim  that, given  $g$ defined on $\Z$ and non negative,  even
and not decreasing on $\N$,   if $\al\neq\bt$ one has
 \begin{equation}\label{pula}
\sum_{i\in\Z} g(i) |\al_i-\bt_i|\leq
2g(m_1)+
\sum_{l\geq 3} g(\na_l)\,.
\end{equation}
Indeed,  by \eqref{pula2} we note that
 \begin{equation*}
 \begin{aligned}
\sum_{i\in\Z} g(i) |\al_i-\bt_i|
&=
g(0)|\al_0-\bt_0|+
\sum_{l\geq 1} g(m_l)
\\&\leq 
g(1)(\al_0+\bt_0)+2g(m_1)+
\sum_{l\geq 3} g(m_l)\,.
\end{aligned}
\end{equation*}
Hence \eqref{pula} follows by
\eqref{cappella} and \eqref{abbacchio}.
Applying \eqref{pula} with
$g(x) =f(\sqrt{x})$,
we have
\begin{align*}
\sum_i |\al_i-\bt_i| g(\jap i) & \le   
2 g(m_1) + \sum_{l\ge 3} g(\na_l) 
\le 
2 g(31\sum_{l\ge 3} \na_l^2) + \sum_{l\ge 3} g(\na_l)
\\ & 
\le 
62 \sum_{l\ge 3} g(\na_l^2) +  \sum_{l\ge 3} g(\na_l) \le
63 \sum_{l\ge 3} \lambda(\na_l)\,,
\end{align*}
which implies the first inequality in \eqref{constance2SE}.
The second one follows by \eqref{mortazza}.

\smallskip
\noindent
We now show item $(ii)$.  
Notice that
\[
\left|\sqrt{i^{4}+m}-i^2\right|\leq \frac{m}{2i^{2}}\leq 1\,,\quad m\in [1,2]\,,\;\;\; i\in\Z\,.
\]
Then, by triangular inequality, we have
\[
\begin{aligned}
\left|\sum_{i\in \mathbb{Z}}(\al_i-\bt_i) \sqrt{i^{4}+m}\right|
&\geq 
\left|\sum_{i\in \mathbb{Z}}(\al_i-\bt_i) i^{2}\right|-
\left|\sum_{i\in \mathbb{Z}}(\al_i-\bt_i) \big(\sqrt{i^{4}+m}-i^2\big)\right|
\\&\geq
10 \sum_i |\al_i-\bt_i|- \sum_i |\al_i-\bt_i|\geq1\,,
\end{aligned}
\]
which is \eqref{pool1}.
\end{proof}

\begin{lemma}\label{stimaSob:lem}
Fix $\tN\geq1$, $\delta\geq(36\tN)^2$, $\tau\leq 36 \tN^2$ and $\td\geq 4\tN$.
Then one has
\begin{equation}\label{seisettedelta}
\mathtt{J}:=\sup_{ j\in\Z,\, (\al,\bt)\in\mathcal{A}} 
\Big(\frac{\jjap{j}^2}{\prod_{i\in\Z}\jjap{i}^{\al_i + \bt_i}} \Big)^\delta 
\prod_{i\in\Z}\big((1+|\al_i-\bt_i|^2)\langle i\rangle^2\big)^{\tau}\leq 
2^{{\delta}-1}(4^6 e^{27})^{72\tN^2}6^{\delta}\,.
\end{equation}
where $\mathcal{A}\subseteq \Lambda$
is the set of indexes $(\al,\bt)$ such that \eqref{divisor} holds and
\[
|\al| + |\bt| = \tN + 2\,,\qquad 
\al_j+\bt_j\neq 0 \,,\quad |\al - \bt| \le \tN + 2\,.
\]

\end{lemma}

\begin{proof}
By \eqref{constanceSOB}   (recall Definition \ref{n star})
 toghether with $|\al| + |\bt| = \tN + 2$  
and $\td \le 4\tN$ we get
\begin{equation*}
\mathtt{J} \leq 
\sup_{\substack{ 
\\ \al_j+\bt_j\neq 0 
\\ |\al - \bt| \le \tN + 2}} 
\Big(\frac{\jjap{j}^2}{\prod_{i\in\Z}\jjap{i}^{\al_i + \bt_i}} \Big)^\delta 
\pa{e^{27\tau}(\tN + 2)^{6\tau}}^2 \pa{\prod_{\ell = 3}^{\tN + 2} \na_\ell^{\frac{15}{2}\tau}}^2,
\end{equation*}
with $\na=\na(\al+\bt)$.
We claim that
\begin{equation}\label{filomena}
\tN + 2 \leq 4 \prod_{l= 3}^{\tN + 2}\jml{\na_l}^{\frac{1}{4\ln 2}}\,.
\end{equation}
Indeed if $\tN=0$, the inequality is trivial.
The case 
$\tN\geq 1$ follows by 
Lemma \ref{luchino}.
Recalling Def.  \ref{n star} we have
\begin{equation}\label{fiorentina2}
\prod_i\jml{i}^{\al_i+\bt_i}= \prod_{l\ge 1}\jml{\na_l}\,.
\end{equation}
Then
\begin{equation*}
\sup_{\substack {j,\al,\bt
\\ \al_j+\bt_j\geq 1 }} 
\frac{\jml{j}^2}{\prod_i\jml{i}^{\al_i+\bt_i}}
\leq 
\frac{\jml{\hat n_1}^2}{\prod_{l\ge 1}\jml{\na_l}}
=
\frac{\jml{\hat n_1}}{\prod_{l\ge 2}\jml{\na_l}}
\leq
\frac{\sum_{l\ge 2}\jml{\na_l}}{\prod_{l\ge 2}\jml{\na_l}} 
=
\frac{1}{\prod_{l\ge 3}\jml{\na_l}}
+
\frac{\sum_{l\ge 3}\jml{\na_l}}{\prod_{l\ge 2}\jml{\na_l}} \,,
\end{equation*}
where the last inequality holds by momentum conservation.
Recall that $\tau \le 36\tN^2$ and
 $\delta \ge (36\tN)^2$.  
Then, by \eqref{filomena} and by the fact that
$(a+b)^{\delta}\leq 2^{{\delta}-1}(a^{\delta}+b^{\delta})$
for $a,b\geq 0,$ ${\delta}\geq 1$, one has
\begin{equation*}
\begin{aligned}
\mathtt{J}&\leq 
2^{{\delta}-1}
\left(
\frac{1}{\prod_{l\ge 3}\jml{\na_l}^{\delta}}
+\frac{(\sum_{l\ge 3}\jml{\na_l})^{\delta}}{\prod_{l\ge 2}\jml{\na_l}^{\delta}} 
\right)
(4^6 e^{27})^{72\tN^2}\prod_{l\ge 3}\jml{\na_l}^{{\delta/2}}
\\&\leq  
2^{{\delta}-1}(4^6 e^{27})^{72\tN^2}
\left(
1+\frac{(\sum_{l\ge 3}\jml{\na_l})^{\delta}}{
\jml{\na_2}^{\delta}\prod_{l\ge 3}\jml{\na_l}^{{\delta}/2}} 
\right)
\\&\leq  
2^{{\delta}-1}(4^6 e^{27})^{72\tN^2}
\left(
1+\frac{(\jml{\na_3}^{1/2}+4)^{\delta}}{\jml{\na_2}^{\delta}} 
\right)\,,
\end{aligned}
\end{equation*}
where we used Lemma \ref{luchino}. Then the thesis follows.
\end{proof}


\begin{proof}[{\bf Proof of Proposition \ref{crescenza}}]\label{proofofcrescenza}
In all that follows we shall use systematically the fact that our Hamiltonians  
preserve 
the momentum 
and are zero at the origin.
 These facts imply that $\abs{\al} + \abs{\bt} \ge 1$.

\vspace{0.3em}
\noindent
\emph{Case} $\SE)$ 
Let us start by proving the bound \eqref{emiliaparanoica}. 
It follows by \eqref{alberellobello} in Lemma \ref{norme proprieta}
provided that (recall Remark \ref{rmk:basicembHam}) 
\begin{equation*}
\CM :=
\sup_{\substack{j\in\Z,\, \al,\bt\in\N^\Z,\\ \al_j+\bt_j\neq 0,
\\ \sum_{i}i(\al_i-\bt_i)=0}}
\frac{c^{(j)}_{r,\twf}(\al,\bt) }{c^{(j)}_{\ri,\twi}(\al,\bt) }
< \infty\,,
\end{equation*}
where $\tw' = \tw(p, s + \sigma)$ and where $\tw(p, s)$ 
is the weight defined in \eqref{peso sub}.
We actually  show that $C_0$ 
is equal to $1$.
By direct computation
\begin{equation}\label{pool2}
\frac{c^{(j)}_{r,\twf}(\al,\bt) }{c^{(j)}_{\ri,\twi}(\al,\bt) } 
= \exp\Big(-\s\Big(	\sum_i \lambda(\jap{i}) (\al_i+\bt_i) -2\lambda(\jap{j}) \Big) \Big)\,.
\end{equation}
By momentum conservation, inequality \eqref{eleganza} holds, 
which together with the sub-linearity of $\lambda$ and  
the definition of $\na$ gives the following chain of inequalities, 	
for all $\al,\bt$ in $\sum_\ast$ such that $\al_j+\bt_j\ne 0$:
\begin{equation*}
\begin{aligned}
\sum_i \lambda(\jap{i}) (\al_i+\bt_i) -2\lambda(\jap{j}) 
&\ge \sum_i \lambda(\jap{i}) (\al_i+\bt_i) -2\lambda(\na_1)  
\\&\ge 
\sum_{l\ge 1} \lambda(\na_l)  - \lambda(\na_1) -\lambda(\sum_{l\ge 2} \na_l)
\\&\ge 
\sum_{l\ge 2} \lambda(\na_l) - \lambda(\sum_{l\ge 2}\na_l)
\ge c \sum_{l\ge 3} \lambda(\na_l) \ge  0\,,
\end{aligned}
\end{equation*}
where the last inequality follows from \eqref{mortazza}. This concludes the proof.
	
\vspace{0.3em}
\noindent
\emph{Case} $\SO)$ 
In order to prove
the bound \eqref{emiliapara2} we follows the ideas in the proof of 
Proposition $6.3$ in \cite{BMP:CMP}
where the norm $|\cdot|_{r,\mathtt{w}(p)}$ (see Def. \ref{Hreta})
with $\mathtt{w}(p)$ in \eqref{peso sob}
is denoted by $\|\cdot\|_{r,p}$.
Again in view of  \eqref{alberellobello} in Lemma \ref{norme proprieta}
we only have to prove that (recall \eqref{coeffSobo})
\begin{equation}\label{pool22}
\sup_{\substack{j\in\Z,\, \al,\bt\in\N^\Z,\\ \al_j+\bt_j\neq 0,
\\ \sum_{i}i(\al_i-\bt_i)=0}}
\frac{c^{(j)}_{r,\twf}(\al,\bt) }{c^{(j)}_{\ri,\twi}(\al,\bt) } 
=
\sup_{\substack{j\in\Z,\, \al,\bt\in\N^\Z,\\ \al_j+\bt_j\neq 0,
\\ \sum_{i}i(\al_i-\bt_i)=0}}
\left(
\frac{\lfloor j\rfloor^{ 2}}{\prod_{i\in\Z}\lfloor i\rfloor^{\al_i+\bt_i} }
\right)^{p'}
1\,,
\end{equation}
where $\tw' = \tw(p+p')$ and where $\tw(p)$ 
is the weight defined in \eqref{peso sob}.

\noindent
We first show that the inequality holds in the case $j=0,\pm1$. Indeed we have
\[
\prod_{i}\lfloor i\rfloor^{\al_i+\bt_i}\geq \prod_i 2^{\al_i+\bt_i}=2^{\sum_{i}\al_i+\bt_i}\geq 4\,,
\]
since $\sum_{i}\al_i+\bt_i\geq2$.

\noindent
Consider now the case $|j|=\lfloor j\rfloor\geq2$. Since $\al_j+\bt_j\geq1$, the inequality
\eqref{pool22} follows by showing that
\begin{equation}\label{pool23}
\sup_{j,\al,\bt}
\frac{|j|}{\prod_{i\neq j}\lfloor i\rfloor^{\al_i+\bt_i} }\leq 1\,.
\end{equation}
By momentum conservation we have
\[
|j|\leq \sum_{i\neq j}|i|(\al_i+\bt_i)\leq \sum_{i\neq j}\lfloor i\rfloor (\al_i+\bt_i)\,.
\]
Then \eqref{pool23} follows by showing that
\[
\sup_{j,\al,\bt}
\frac{\sum_{i\neq j}\lfloor i\rfloor (\al_i+\bt_i)}{\prod_{i\neq j}\lfloor i\rfloor^{\al_i+\bt_i} }\leq 1\,,
\]
where we can restrict the sum and the product to the indexes $i$ such that
$\al_i+\bt_i\geq1$. The latter bound follows
by the fact that, given $x_k\geq1$,
\[
\frac{\sum_{2\leq k\leq n}k x_k}{\prod_{2\leq k\leq n}k^{x_{k}}}
\leq 1\,,
\]
as one can prove by induction oven 
$n$ and recalling that $n^x\geq nx$ for $n\geq2$, $x\geq1$. 
\end{proof}

\gr{Declarations}. Data sharing not applicable to this article as no datasets were generated or analyzed during the current study.

\noindent
Conflicts of interest: The authors have no conflicts of interest to declare.



\bibliography{biblioBeam}
\bibliographystyle{alpha}

\end{document}